\documentclass[11pt,leqno]{amsart}
\usepackage{graphicx}
\usepackage{geometry}
\usepackage{indentfirst,csquotes}
\usepackage{setspace}
\usepackage{caption}
\usepackage{subcaption}
\usepackage{relsize}

\usepackage{amssymb,amsthm,amsmath}
\usepackage{xcolor,paralist,hyperref,titlesec,fancyhdr,etoolbox}
\definecolor{mycitecolor}{rgb}{0.0,0.55,0.55}
\hypersetup{
  colorlinks=true,
  linkcolor=red,
  citecolor=mycitecolor,
  urlcolor=red!50!black
}

\newtheorem{theorem}{Theorem}[]

\newtheorem{lemma}[theorem]{Lemma}
\theoremstyle{remark}
\newtheorem{remark}[theorem]{Remark}

\DeclareOldFontCommand{\cal}{\normalfont\calshape}{\mathcal}

\makeatletter
\long\def\@makecaption#1#2{
  \vskip\abovecaptionskip
  \begingroup
  \small
  \setbox\@tempboxa\hbox{\normalfont\bfseries #1. \normalfont #2}
  \ifdim \wd\@tempboxa >\hsize
    \normalfont\bfseries #1. \normalfont #2\par
  \else
    \hbox to\hsize{\hfil\box\@tempboxa\hfil}
  \fi
  \endgroup
  \vskip\belowcaptionskip
}
\makeatother

\makeatletter
\providecommand{\@secnumpunct}{.\quad}
\makeatother

\renewcommand{\footnoterule}{
  \kern -3pt
  \hrule width \textwidth height 0.4pt
  \kern 2.6pt
}

\titleformat{\section}
  {\normalfont\sffamily\bfseries\large}
  {\thesection.}
  {0.5em}
  {}
\titleformat{\subsection}[runin]
  {\normalfont\sffamily\bfseries\normalsize}
  {\thesubsection.}
  {0.5em}
  {}
  [.]
\titlespacing*{\subsection}
  {0pt}
  {1.0ex plus .3ex minus .2ex}
  {0.8em}
  
\makeatletter
\newcommand{\mykeywords}[1]{\gdef\@mykeywords{#1}}
\gdef\@mykeywords{}
\newcommand{\myabstract}[1]{\gdef\@myabstract{#1}}
\gdef\@myabstract{}

\def\@maketitle{%
  \normalfont\normalsize
  \begin{center}
    {\sffamily\bfseries\large
    \@title\par}
    \vskip 20pt
    {\sffamily\normalsize
    \authors\par}
    \vskip 15pt
    {\color{gray}\rule{\textwidth}{3pt}\par}
  \end{center}
  \ifx\@myabstract\@empty
  \else
    \par\vspace{0.5em}
    \noindent{\sffamily\bfseries Abstract.} \@myabstract\par
    \vspace{0.5em}
  \fi
  \ifx\@mykeywords\@empty
  \else
    \par\vspace{0.5em}
    \noindent{\sffamily\bfseries Keywords.} \@mykeywords\par
  \fi
}
\makeatother

\makeatother

\usepackage{lipsum}

\usepackage{amsmath,amsfonts,bm}

\def\eqref#1{(\ref{#1})}

\def\1{\bm{1}}

\def\vv{{\bm{v}}}

\DeclareMathAlphabet{\mathsfit}{\encodingdefault}{\sfdefault}{m}{sl}
\SetMathAlphabet{\mathsfit}{bold}{\encodingdefault}{\sfdefault}{bx}{n}

\newcommand{\x}{\mathbf{x}}
\newcommand{\e}{\mathbf{e}}
\newcommand{\y}{\mathbf{y}}
\newcommand{\z}{\mathbf{z}}
\newcommand{\uu}{\mathbf{u}}
\newcommand{\bfv}{\mathbf{v}}

\newcommand{\p}{\mathbf{p}}

\newcommand{\I}{\mathbf{I}}
\newcommand{\X}{\mathbf{X}}
\newcommand{\Y}{\mathbf{Y}}
\newcommand{\Z}{\mathbf{Z}}
\newcommand{\A}{\mathbf{A}}
\newcommand{\U}{\mathbf{U}}
\newcommand{\V}{\mathbf{V}}
\newcommand{\C}{\mathbf{C}}
\newcommand{\Q}{\mathbf{Q}}

\newcommand{\B}{\mathbf{B}}
\newcommand{\M}{\mathbf{M}}

\newcommand{\D}{\mathbf{D}}
\newcommand{\G}{\mathbf{G}}

\newcommand{\PP}{\mathbf{P}}
\newcommand{\bfS}{\mathbf{S}}
\newcommand{\E}{\mathbf{E}}

\newcommand{\bfO}{\mathbf{O}}
\newcommand{\bfPhi}{\mathbf{\Phi}}
\newcommand{\bfPsi}{\mathbf{\Psi}}
\newcommand{\bfSigma}{\mathbf{\Sigma}}
\newcommand{\bfOmega}{\mathbf{\Omega}}
\newcommand{\bfTheta}{\mathbf{\Theta}}
\newcommand{\bfLambda}{\mathbf{\Lambda}}

\newcommand{\bfXi}{\mathbf{\Xi}}

\newcommand{\rank}{\mathsf{rank}}

\newcommand{\tr}{\mathsf{Tr}}
\newcommand{\fro}{\mathsf{F}}
\newcommand{\st}{\mathsf{St}}
\newcommand{\gr}{\mathsf{Gr}}

\newcommand{\diag}{\mathsf{diag}}

\newcommand{\Span}{\mathsf{span}}
\newcommand{\range}{\mathsf{range}}

\DeclareMathOperator*{\argmin}{arg\,min}

\usepackage{algorithm}
\usepackage{algpseudocode}
\usepackage{booktabs}
\usepackage{pifont}
\usepackage{subcaption}
\usepackage{wrapfig}
\usepackage{xr}

\algrenewcommand\algorithmiccomment[1]{\hfill\textit{// #1}}

\begin{document}
\pagestyle{plain}
{\fontfamily{lmr}\selectfont\title{Direction–Magnitude Decomposition for Low-Rank Matrix Optimization:\\Faster Convergence and Saddle-to-saddle Dynamics}}

\author[Y. Wei, L. Zhang, B. Li, N. He]{Yudong Wei\textsuperscript{\dag} \quad Liang Zhang\textsuperscript{\ddag} \quad Bingcong Li\textsuperscript{\ddag} \quad Niao He\textsuperscript{\ddag}}

\thanks{\textsuperscript{\dag} Department of Mathematics, ETH Zurich, Zurich, Switzerland (\texttt{yudwei@ethz.ch}).}

\thanks{\textsuperscript{\ddag} Department of Computer Science, ETH Zurich, Zurich, Switzerland (\texttt{liang.zhang@inf.ethz.ch}, \\\texttt{bingcong.li@inf.ethz.ch}, \texttt{niao.he@inf.ethz.ch}).}

\myabstract{\small
    Low-rank matrix optimization is often carried out via the Burer–Monteiro (BM) formulation, but choosing the factorization rank $r$ is delicate and can substantially slow optimization. We propose a unified framework, termed direction–magnitude decomposition (DMD), that decomposes the optimization variable to improve optimization efficiency even when the target rank is unknown. We develop two DMD-based approaches and establish their theoretical advantages on the canonical problem of matrix factorization. The first, overparameterized DMD, uses a rank $r$ larger than necessary and enjoys faster convergence as $r$ increases. The second, recursive DMD, is motivated by the incremental eigenpair learning, or saddle-to-saddle, behavior of overparameterized DMD. It achieves lower memory and computational costs, complementing overparameterized DMD. Both approaches are exponentially faster than gradient descent applied to the BM formulation. Numerical experiments on matrix factorization, sensing, and completion corroborate our theoretical findings and demonstrate the practical effectiveness of DMD.}

\mykeywords{\small low-rank matrix optimization, Riemannian optimization, saddle-to-saddle dynamics, \\direction-magnitude decomposition}

\maketitle

\begingroup
\renewcommand{\thefootnote}{\fnsymbol{footnote}}
\footnotetext[2]{Department of Mathematics, ETH Zurich, Zurich, Switzerland (\texttt{yudwei@ethz.ch}).}

\footnotetext[3]{Department of Computer Science, ETH Zurich, Zurich, Switzerland (\texttt{liang.zhang@inf.ethz.ch},\\
\texttt{bingcong.li@inf.ethz.ch}, \texttt{niao.he@inf.ethz.ch}).}
\endgroup

\section{Introduction}
Low-rank matrix optimization has attracted significant interest in machine learning, signal processing, and computer vision, with applications including collaborative filtering \cite{schafer2007collaborative, he2017neural}, graph learning \cite{xia2021graph, wang2021graph}, imaging science \cite{horstmeyer2015solving, barrett2013foundations}, and fine-tuning foundation models \cite{hu2022lora,li2026low}. In many such problems, the decision variable is a positive semidefinite (PSD) matrix $\A \in \mathbb{S}_+^{m}$ whose global optimum is known to be low rank. To exploit this low-rank structure, a widely adopted approach is the Burer–Monteiro (BM) formulation \cite{burer2003nonlinear}, which parameterizes $\A$ as $\A = \Y\Y^\top$ with a factor matrix $\Y \in \mathbb{R}^{m \times r}$ and $r < m$.
As a canonical example, this work focuses on the matrix factorization problem \cite{chi2019nonconvex, zhuo2024computational, li2019landscape, lyu2020online}:
\begin{align}\label{prob_bm}
    \min_{\Y \in \mathbb{R}^{m \times r}} \; \frac{1}{4} \| \Y \Y^\top - \A \|_\fro^2.
\end{align}
Although simple in form, \eqref{prob_bm} is a fundamental problem
that underlies many practical settings, including matrix sensing \cite{li2018algorithmic, zhong2015efficient, tong2021accelerating, shen2026escaping} and matrix completion \cite{candes2012exact, keshavan2010matrix, ongie2021tensor, eriksson2012high}. 

When the rank of the global minimizer $r_A:=\rank(\A)$ is known, one could choose $r=r_A$ for parameterizing $\Y$.
In practice, however, this rank is unknown a priori, and the most common approach is to adopt an overparameterized formulation with $r>r_A$ to ensure sufficient expressiveness. However, overparameterized BM formulation can be less efficient from an optimization standpoint. Recent results show that overparameterization can significantly hinder convergence and can even lead to exponential slowdowns \cite{xiong2024how}. The technical reason for this is that the variable $\Y$ jointly encodes both direction and magnitude information, which can interact unfavorably during optimization. 

To address this inefficiency, we develop a unified framework, termed direction–magnitude decomposition (DMD), that provides two approaches for improving optimization efficiency when the true rank $r_A$ is unknown. DMD remains memory efficient as BM formulation, but achieves faster convergence by decomposing the variable $\Y$ into direction and magnitude components and treating them separately.

The first approach, \textit{overparameterized DMD}, parameterizes both magnitude and direction components with additional parameters, i.e., $r>r_A$. Unlike overparameterized BM, however, our approach turns overparameterization into an advantage of optimization, i.e., larger $r$ leads to faster convergence.
We further establish several additional merits of DMD. First, it yields an exponentially faster convergence rate than gradient descent (GD) applied to the classical BM formulation. Second, by introducing a surrogate loss for the directional variable, we improve the $\kappa$-dependence of the iteration complexity from ${\cal O}(\kappa^4)$ to ${\cal O}(\kappa^2)$, where $\kappa$ is the condition number of $\A$. A detailed comparison with existing algorithms is given in Table \ref{tab:comparison}.

\begin{table}[t]
\caption{\small Comparison with existing algorithms for low-rank matrix factorization. ``EP", ``OP" and ``RE" stand for exactparameterized, overparameterized and recursive, respectively.\protect\footnotemark}
\vspace{1em}
\centering
\small
\renewcommand{\arraystretch}{1.7}
\begin{tabular*}{\textwidth}{@{\extracolsep{\fill}} c c c c c}
\toprule
\textbf{Formulation}  & \textbf{Algorithm} & \textbf{Iteration Complexity} & \textbf{Faster with OP}\\
\midrule
\eqref{prob_bm} 
& GD (EP)\cite{ye2021global}
&$\mathcal{O}\big(m^2r_A^4\kappa^4\!+\!m^2r_A^4\kappa^4\log(\frac{1}{\varepsilon\kappa})\big)$ & \ding{55}\\

\eqref{prob_bm}
& GD (OP)\cite{xiong2024how}
& $\Omega\big(\frac{\kappa^2}{\log(m r_A^2)\, \varepsilon}\big)$ & \ding{55}\\

\eqref{prob_dmd}
& RGD\cite{lion2025polar} 
& $\mathcal{O}\big(
\frac{m^4 r^3 r_A \kappa^4}{(r - r_A)^4}\!+\!
\frac{m^3 r^3 \kappa^4}{(r - r_A)^4}
\log(\frac{1}{\varepsilon})
\big)$ & \ding{51}\\

\midrule

\eqref{prob_dmd}
& OP DMD
& $\mathcal{O}\big(\frac{m^2r^2r_A^2\kappa^2}{(r-r_A)^4}\!+\!\frac{mrr_A\kappa^2}{(r-r_A)^2}\log(\frac{1}{\varepsilon})\big)$ & \ding{51}\\

\eqref{prob_dmd}
& RE DMD
& $\mathcal{O}\big(\kappa(r_A\log(m)+r_A^2)+r_A\kappa\log(\frac{1}{\varepsilon})\big)$ & --\\

\bottomrule
\end{tabular*}
\label{tab:comparison}
\end{table}
\footnotetext{OP DMD and RE DMD focus on symmetric objective matrices, whereas other algorithms in the comparison are developed for asymmetric objective matrices. RE DMD requires a separated spectrum assumption; see Section \ref{sec:comparison_op_re} for details.}

The second approach, termed \textit{recursive DMD}, is motivated by a detailed characterization of the optimization trajectory of overparameterized DMD. Specifically, we show that the trajectory passes sequentially through a series of saddle points, exhibiting a saddle-to-saddle behavior in which each saddle is associated with an eigenpair\footnote{An eigenpair $(\lambda,\uu)$ of a matrix $\A\in\mathbb{R}^{m\times m}$ consists of $\lambda\in\mathbb{R}$ and a nonzero $\uu\in\mathbb{R}^m$ such that $\A\uu=\lambda\uu$. In this work, we assume $\|\uu\|=1$ and refer to $(\lambda,\uu)$ as the leading eigenpair when $\lambda$ is the largest eigenvalue.} of the target matrix. 
This sequential eigen-learning perspective naturally suggests an alternative to overparameterization: one may use a rank-1 DMD parameterization to recursively extract leading eigenpairs until the full matrix is recovered. We observe a tradeoff between overparameterized and recursive DMD: while overparameterization is more generally applicable, recursive DMD can be more efficient when the eigengaps of $\A$ are favorable, requiring less memory and computation.

In summary, DMD offers a memory-efficient framework for low-rank matrix optimization. It does not require prior knowledge of $r_A$ and enjoys favorable optimization properties. We establish theoretical guarantees for matrix factorization and further demonstrate its broader applicability through numerical experiments on matrix sensing and matrix completion. Our main contributions are as follows.

\begin{enumerate}
    \item[\ding{118}] \textbf{Overparameterized DMD.} When parameterizing the direction and magnitude variables with overparameterization, we prove that DMD yields several improvements on iteration complexity. The dependence of the optimality error $\varepsilon$ is improved to ${\cal O}(\log (1/\varepsilon))$ compared with ${\cal O}(1/\varepsilon)$ in BM formulation. Compared to \cite{lion2025polar}, our complexity improves the $\kappa$-dependence to ${\cal O}(\kappa^2)$ rather than ${\cal O}(\kappa^4)$. Moreover, the complexity also inversely decreases as the level of overparameterization $r$ increases. A more detailed comparison on the iteration complexities can be found in Table \ref{tab:comparison}.

\item[\ding{118}] \textbf{Saddle-to-saddle dynamics.} Overparameterized DMD exhibits saddle-to-saddle dynamics, successively entering and escaping neighborhoods of $r_A$ saddle points before globally converging. Each transition from one saddle point to the next can be interpreted as learning a leading eigenpair of the objective matrix. 

\item[\ding{118}] \textbf{Recursive DMD.} 
Motivated by the saddle-to-saddle dynamics, we show that overparameterization is not always necessary to handle an unknown rank $r_A$. As an alternative, recursive DMD extracts one leading eigenpair of the objective matrix at each step, and progressively recovers the low-rank factorization. It enjoys lower memory and computational cost per iteration, while serving as an efficient complement to overparameterized DMD under favorable spectral conditions.

\item[\ding{118}] \textbf{Empirical validation.} We conduct extensive numerical experiments on representative low-rank matrix optimization problems, including matrix factorization, sensing, and completion. The results validate our theoretical findings, demonstrating the practical effectiveness of both overparameterized DMD and recursive DMD.
\end{enumerate}
\vspace{1em}

\textbf{Notational conventions.} Bold uppercase (lowercase) letters denote matrices (column vectors); $\tr(\cdot)$, $(\cdot)^\top$ and $\|\cdot\|_\fro$ refer to the trace, transpose and Frobenius norm of a matrix; $\|\cdot\|$ denotes the spectral ($\ell_2$) norm for matrices (vectors); $\sigma_i(\cdot)$ denotes the $i$-th largest singular value of a matrix and $\lambda_i(\cdot)$ denotes the $i$-th largest eigenvalue of a matrix. Moreover, $\mathbb{S}^m$ and $\mathbb{S}_+^m$ denote symmetric and positive semidefinite (PSD) matrices of size $m \times m$, respectively.

\section{Related work}
\label{sec:related_works}

\textbf{Low-rank matrix optimization.} Low-rank matrix optimization has attracted significant attention due to its broad applications ranging from phase retrieval to collaborative filtering \cite{candes2013phaselift,duchi2020conic, schafer2007collaborative,Srebro2010collaborative, udell2019big}. As a canonical problem, matrix factorization is popular in theoretical studies because of the intricate geometry of the associated loss landscapes, which contain numerous saddle points and can fail to satisfy global smoothness or a global Polyak–Łojasiewicz (PL) condition \cite{ge2017no,ma2023over}. 
In the exactparameterized setting, \cite{ye2021global} establishes linear convergence of GD applied to the BM formulation. However, the recent work \cite{xiong2024how} has shown that overparameterization may lead to an exponential slowdown in convergence. \cite{tong2021accelerating, xu2023power} introduce scaled update schemes with quasi-Newton characteristics and obtain linear convergence with mild $\kappa$-dependence in iteration complexity. In contrast, DMD is a purely first-order method that still achieves linear convergence.
The decomposition of direction and magnitude is partially inspired from \cite{li2024crucial}, which shows prior knowledge for the directions in BM formulation can induce exponentially faster convergence. 
Closely related to our setting, \cite{lion2025polar} considers a factorized formulation. Compared with their approach, our method adopts a fully decomposed manner and achieves an improved dependence on $\kappa$. Beyond matrix factorization, a large body of work has studied closely related low-rank matrix recovery problems, including matrix sensing and completion, which share the same goal of recovering a low-rank matrix but differ in their observation models, from full observations to linear measurements and partial observations \cite{tong2021accelerating,charisopoulos2021low,ding2024flat, tan2011rank, tropp2017practical, zhao2020matrix, shen2026escaping, stoger2021small}. For matrix sensing, prior work \cite{tan2011rank} studies finite-field low-rank recovery from random linear measurements and characterizes the information-theoretic limits for exact recovery. The convergence of GD under  small random initialization is studied in \cite{stoger2021small}.
For matrix completion, another line of work \cite{zhao2020matrix} formulates data recovery from a probabilistic perspective, using low-rank Gaussian copula models to provide both imputations and uncertainty estimates for missing entries. In this paper, we extend both overparameterized DMD and recursive DMD to matrix sensing and completion, and demonstrate their effectiveness through numerical experiments.

\textbf{Other parameterizations for low-rank optimization.} Several recent works have also begun to rethink the classical BM formulation. It is shown in \cite{ma2023over} that, for rank-one matrix sensing, a tensorized lift of the classical BM formulation enjoys a more favorable optimization landscape than vanilla BM. \cite{levin2025effect} shows that different parameterizations can lead to different strict saddles, while \cite{wei2025benefits} shows that parameterizations with the same expressiveness can induce significantly different optimization behavior. Our work provides further evidence that parameterization and optimization should be co-designed.
 
\textbf{Saddle-to-saddle dynamics.}
Saddle-to-saddle dynamics aligned with an incremental learning paradigm\footnote{Also known as deflation; see, e.g., \cite{ge2021understanding,anandkumar2014tensor,seddik2023optimizing, zhang2025saddle}.} has been observed for GD applied to the BM formulation \cite{li2020towards, jiang2023algorithmic, cao2022dynamics}. As shown in \cite{jiang2023algorithmic}, GD sequentially recovers the principal components of the objective matrix when started from a sufficiently small random initialization. Each stage of this process corresponds to the trajectory moving between saddle points associated with different ranks, giving rise to a characteristic saddle-to-saddle dynamics. Beyond matrix factorization, similar phenomena have also been observed in matrix sensing. The work \cite{jin2023understanding} establishes a detailed characterization of incremental learning dynamics for GD and demonstrates that the sequential recovery of spectral components persists even in underparameterized settings. In this paper, we characterize the saddle-to-saddle dynamics of overparameterized DMD. Owing to the explicit decomposition into direction and magnitude components, our analysis admits a more transparent geometric interpretation. More importantly, we show that the sequential learning of eigenpairs in DMD does not fundamentally rely on overparameterization. In particular, a rank-one DMD parameterization suffices to recover the leading eigenpair. 
This observation naturally leads to a recursive DMD approach for reducing both memory and computational complexity while preserving the favorable convergence behavior.

\textbf{Riemannian optimization.} Riemannian optimization is naturally connected to DMD for learning the direction variable, which is constrained to lie on the Stiefel manifolds \cite{tagare2011notes, lee2025equivariant, gao2021riemannian, chen2021decentralized}. By exploiting the underlying geometric structure, Riemannian optimization extends gradient-based methods to optimization problems with smooth manifold constraints. A well-established theoretical and algorithmic framework has been developed for such problems, including retraction-based updates, convergence guarantees, and efficient implementations; see, e.g., \cite{absil2008optimization,smith2014optimization,mishra2012riemannian,boumal2023introduction, hou2020analysis, park2026adaptive, li2026convergence}. The work \cite{li2026convergence} analyzes block majorization-minimization methods for constrained block-Riemannian optimization and establishes asymptotic convergence to stationary points for nonsmooth nonconvex objectives. Another recent work \cite{park2026adaptive} develops an adaptive Riemannian gradient method, which adopts a line-search-free adaptive stepsize rule and establishes a non-ergodic sublinear rate under local geodesic smoothness. In this work, we adopt a surrogate loss for the direction variable. The update performs a Riemannian gradient step on the surrogate objective, followed by a retraction back to the Stiefel manifold, and is provably faster than directly relying on the original objective.

\section{Direction-magnitude decomposition}
\label{sec:problem_reformulation}
Overparameterization, i.e., choosing $r>r_A$, can substantially slow down GD on \eqref{prob_bm}. In particular, it was shown in \cite{xiong2024how} that the reconstruction error $\|\Y_t\Y_t^\top-\A\|_\fro$ of GD cannot decay faster than $\Omega(1/t)$, where $t$ denotes the iteration number. This sublinear rate is \textit{exponentially} slower than the linear one achieved in the exactparameterized regime, where the true rank $r_A$ is known and $r=r_A$ is employed \cite{ye2021global}.

\subsection{DMD reformulation}
To mitigate the slowdown in convergence caused by overparameterization, we introduce a direction–magnitude decomposition framework induced by the polar decomposition; see Section 9.4.3 of \cite{golub2013matrix} for more details. Specifically, for any matrix $\Y\in\mathbb{R}^{m\times r},m\geq r$, its polar decomposition can be written as
\[
    \Y = \X \tilde{\bfTheta},
    \qquad \X\in\st(m,r),\,
    \tilde{\bfTheta}\in \mathbb{S}_{+}^r,
\]
where $\st(m,r):=\{\X\in\mathbb{R}^{m\times r}|\X^\top\X=\I_r\}$ denotes the Stiefel manifold and $\tilde{\bfTheta}$ is a PSD matrix. Geometrically, $\X$ is an orthonormal basis spanning an $r$-dimensional subspace, thereby encoding the direction component of $\Y$, while $\tilde{\bfTheta}$ captures its magnitude information. Substituting this factorization into \eqref{prob_bm}, we obtain
\begin{align}
    \min\limits_{\X\in\st(m,r),\tilde{\bfTheta}\in\mathbb{S}_+^r}\frac{1}{4}\|\X\tilde{\bfTheta}\tilde{\bfTheta}^\top\X^\top-\A\|_\fro^2.
\end{align}

This problem can be further simplified by absorbing the product $\tilde{\bfTheta}\tilde{\bfTheta}^\top$ into a single matrix $\bfTheta\in\mathbb{S}_+^r$ and relaxing the PSD constraint on $\bfTheta$ to only symmetry, i.e., $\bfTheta\in\mathbb{S}^r$. Notably, in the overparameterized regime, this relaxation preserves the same global optimum as the original formulation, while largely improving computational efficiency by avoiding costly operations, such as SVDs or matrix exponentials needed for optimization over PSD cones \cite{vandenberghe1996semidefinite, todd2001semidefinite, yurtsever2021scalable}.

With this direction-magnitude decomposition and the aforementioned simplification, we arrive at the following optimization problem:
\begin{align}\label{prob_dmd}
    \min\limits_{\X\in\st(m,r),\bfTheta\in\mathbb{S}^r}f(\X,\bfTheta):=\frac{1}{4}\|\X\bfTheta\X^\top-\A\|_\fro^2.
\end{align}
Similar reformulations have appeared in \cite{levin2025effect, wei2025benefits}. The former tackles local geometry around stationary points, while the latter establishes faster convergence rates for overparameterized matrix sensing problems. Our work differs by solving \eqref{prob_dmd} in a fully direction–magnitude decomposed manner, as detailed in the next subsection.

\subsection{DMD optimization}
The key idea behind our algorithm is a separation of optimality in \eqref{prob_dmd}. In particular, the optimal $\X$ identifies the eigenspace of $\A$ regardless of whether $\bfTheta$ is optimal. Noticing that this eigenspace can be learned from a closely related problem, principal component analysis (PCA) \cite{abdi2010principal, balzano2018streaming, greenacre2022principal}, we use PCA as a surrogate objective for learning $\X$.

Recall the Ky Fan's characterization of PCA \cite{fan1949theorem}. For a PSD matrix $\A\in\mathbb{S}_+^{m}$, the leading $r$-dimensional eigenspace can be obtained by minimizing $\frac{1}{4}\|\X\X^\top-\A\|_\fro^2$ over $\X\in\st(m,r)$. To handle the manifold constraint on the direction variable $\X$, we optimize it with Riemannian gradient descent. Denote  $\hat{\G}_t=(\X_t\X_t^\top-\A)\X_t$ as the Euclidean gradient (w.r.t. $\X_t$) of the PCA objective, the Riemannian gradient is thus
\begin{align}\label{rgd_X}
    \G_t:=(\I_m-\X_t\X_t^\top)\hat{\G}_t+\frac{1}{2}\X_t(\X_t^\top\hat{\G}_t-\hat{\G}_t^\top\X_t).
\end{align}
It can be seen that \eqref{rgd_X} is independent of $\bfTheta$. In other words, the optimization of the direction is fully decoupled from that of the magnitude. To ensure feasibility after each update, we further apply a polar retraction\footnote{Let $\X\in\st(m,r)$ and a point in its tangent space $\G\in {\cal T}_\X\st(m,r)$. The polar retraction for $\X + \G$ is given by $\mathcal{R}_\X(\G) = (\X + \G)(\I_r +  \G^\top \G)^{-1/2}$.}, leading to the following update rule:
\begin{align}\label{update_X}
    \X_{t+1}=(\X_t-\eta\G_t)(\I_r+\eta^2\G_t^\top\G_t)^{-1/2},
\end{align}
where $\eta>0$ denotes the stepsize.

For the magnitude variable $\bfTheta$, we employ GD on the original objective \eqref{prob_dmd} with stepsize $\mu>0$, given by
\begin{align}\label{update_theta}
    \bfTheta_t=(1-\frac{\mu}{2})\bfTheta_{t-1}+\frac{\mu}{2}\X_t^\top\A\X_t.
\end{align}
This update preserves the symmetry of $\bfTheta_t,t\geq0$ throughout the iterations.

In summary, the step-by-step procedure for solving \eqref{prob_dmd} is summarized in Algorithm~\ref{alg_DMD}. For convenience, we refer to this algorithm as overparameterized DMD.

\begin{algorithm}[t]
\caption{Overparameterized DMD for solving \eqref{prob_dmd}}\label{alg_DMD}
\begin{algorithmic}[1]
\State \textbf{Input:} Initial point $\X_0\in\st(m,r)$, stepsizes $\eta, \mu$, and number of iterations $T$
\For{$t = 0, 1, \ldots, T$}
    \State Update the magnitude variable $\bfTheta_t$ via \eqref{update_theta}
    \State Obtain the surrogate Riemannian gradient $\G_t$ via \eqref{rgd_X}
    \State Update the direction variable $\X_{t+1}$ via \eqref{update_X}
\EndFor
\State \textbf{Output:} $\X_T$, $\bfTheta_T$
\end{algorithmic}
\end{algorithm}

\section{Benefits of overparameterized DMD}
\label{sec:benefits_DMD}
We consider solving \eqref{prob_dmd} under random initialization, where the initial direction variable $\X_0$ is drawn uniformly random from the Stiefel manifold $\st(m,r)$. In practice, this can be generated as $\X_0=\Z_0(\Z_0^\top\Z_0)^{-1/2}$, where the entries of $\Z_0\in\mathbb{R}^{m\times r}$ are i.i.d. Gaussian random variables $\mathcal{N}(0,1)$ \cite{chikuse2012statistics}. 

Recall that the rank of $\A$ is denoted as $r_A$. Let the compact SVD of $\A$ be $\A=\U\bfSigma\U^\top$, where $\U\in\mathbb{R}^{m\times r_A}$ and $\bfSigma\in\mathbb{S}_+^{r_A}$. Without loss of generality, we assume that $\sigma_1(\bfSigma)=1$ and $\sigma_{r_A}(\bfSigma)=1/\kappa$ with $\kappa>1$ denoting the condition number.

\begin{theorem}\label{theorem_DMD}
    Consider solving the matrix factorization problem \eqref{prob_dmd} initialized with a random $\X_0 \in \st(m,r),r>r_A$.
    Suppose that $r_A \le \frac{m}{2}$. Algorithm \ref{alg_DMD} using stepsizes $\eta=\mathcal{O}(\frac{(r-r_A)^2}{mrr_A\kappa})$ and $\mu=2$ generates a sequence $\{\X_t,\bfTheta_t\}_{t=0}^{\infty}$. With high probability over the initialization, this sequence satisfies that for any $\varepsilon\in(0,1)$, we have $f(\X_{t_\varepsilon},\bfTheta_{t_\varepsilon})\leq\varepsilon$ after at most $t_\varepsilon=\mathcal{O}\big(\frac{m^2r^2r_A^2\kappa^2}{(r-r_A)^4}+\frac{mrr_A\kappa^2}{(r-r_A)^2}\log(\frac{1}{\varepsilon})\big)$ iterations.
\end{theorem}

\textbf{Proof sketch:}
The proof of Theorem \ref{theorem_DMD} is based on tracking the evolution of the subspace spanned by the direction variable $\X_t$. Since the target matrix has column space $\mathsf{span}(\U)$, convergence of the direction component amounts to $\mathsf{span}(\U)\subseteq \mathsf{span}(\X_t)$, or equivalently, to the vanishing of all principal angles between these two subspaces. We measure this alignment through
$\bfPhi_t:=\U^\top\X_t$, whose singular values equal the cosines of the principal angles \cite{bjorck1973numerical}. In particular, $\tr(\bfPhi_t\bfPhi_t^\top)\to r_A$ indicates that the two subspaces become fully aligned. The proof proceeds in two stages. In the first stage, starting from random initialization, we show that $\tr(\bfPhi_t\bfPhi_t^\top)$ increases monotonically from near zero to a constant-level alignment, say $r_A-0.5$. This establishes that the iterates escape saddle regions in polynomial time. In the second stage, once $\tr(\bfPhi_t\bfPhi_t^\top)>r_A-0.5$, we prove a contraction for the residual alignment error $r_A-\tr(\bfPhi_t\bfPhi_t^\top)=\tr(\I_{r_A}-\bfPhi_t\bfPhi_t^\top)$, which then decays linearly to zero. Once the correct subspace is identified, the corresponding magnitude component is automatically recovered. Combining the subspace alignment and magnitude recovery yields geometric decay of the reconstruction error and hence $\lim_{t\to\infty}\|\X_t\bfTheta_t\X_t^\top-\A\|_\fro=0$.

\begin{remark}[Discussion of Theorem \ref{theorem_DMD}]
    Comparing with other iteration complexity bounds in Table \ref{tab:comparison},
    Theorem \ref{theorem_DMD} highlights several advantages of overparameterized DMD in efficient optimization on matrix factorization.

    \emph{(i) Linear convergence.} Theorem \ref{theorem_DMD} shows that overparameterized DMD converges to the ground-truth matrix $\A$ at a linear rate. In contrast, based on the BM formulation, the convergence behavior of randomly initialized GD on \eqref{prob_bm} is weaker, where at best a sublinear convergence rate can be achieved under overparameterization \cite{xiong2024how}. 

    \emph{(ii) Improved $\kappa$-dependence.} In terms of the dependence on the condition number $\kappa$, overparameterized DMD improves the scaling from $\mathcal{O}(\kappa^4)$, which arises for GD in the exactparameterized setting \cite{ye2021global} and for Riemannian gradient descent (RGD) in the overparameterized setting \cite{lion2025polar}, to $\mathcal{O}(\kappa^2)$. This reduced $\kappa$-dependence indicates that the proposed algorithm is particularly suitable for ill-conditioned problems. Moreover, since $r-r_A<r<m$, our iteration complexity bound strictly improves upon those of both GD and RGD.
    
    \emph{(iii) Faster with overparameterization.} Because the additional parameters inevitably induce computation and memory overheads, it is natural to ask whether a higher level of overparameterization, i.e., a larger $r$, can also bring commensurate optimization gains. It can be seen from Table \ref{tab:comparison} that GD does not benefit from increasing $r$, whereas overparameterized DMD can effectively leverage it. Setting $r = p r_A$ for some $p > 1$, one can rewrite the iteration complexity as $\mathcal{O}\big(\frac{m^2 p^2 \kappa^2 }{(p-1)^4}+\frac{mp\kappa^2}{(p-1)^2}\log(\frac{1}{\varepsilon})\big)$, which decreases polynomially with $p$. To quantitatively understand the merits of overparameterization, we consider two cases.
    In the mildly overparameterized regime, where $r = r_A + c$ for some constant $c ={\mathcal O}(1)$, the convergence complexity reads $\mathcal{O}\big(m^2r_A^4\kappa^2+mr_A^2\kappa^2\log(\frac{1}{\varepsilon})\big)$. When the level of overparameterization increases to $r = c r_A $, the bound improves to $\mathcal{O}\big(m^2\kappa^2+m\kappa^2\log(\frac{1}{\varepsilon})\big)$. 
    Through comparison, it is readily seen that a larger $r$ yields up to a factor of ${\mathcal O}(r_A^2)$ reduction in iteration complexity.
\end{remark}

Lastly, Theorem \ref{theorem_DMD} implies that at most $\mathcal{O}(\frac{m^2 r^2 r_A^2 \kappa^2}{(r-r_A)^4})$ iterations are sufficient to escape all possible saddle points. A detailed characterization of the saddle structure and the escaping dynamics of overparameterized DMD is presented in the next section.

\section{The saddle-to-saddle dynamics of overparameterized DMD}\label{sec:saddle_phase}

In this section, we take a closer look at the saddle escape of overparameterized DMD for matrix factorization. As shown in Figure \ref{fig:saddle} and explained in detail below, overparameterized DMD traverses a sequence of saddles before reaching a global optimum. The saddle-to-saddle behavior is known for GD on BM formulation \eqref{prob_bm} \cite{li2020towards,jin2023understanding} and we now characterize this behavior for overparameterized DMD on \eqref{prob_dmd}. To this end, we first identify a family of saddle points of $f(\X,\bfTheta)$.

Let $\A_\rho:=\argmin\limits_{\rank(\hat{\A})\leq \rho}\, \|\hat{\A}-\A\|_\fro^2$ be the best rank-$\rho$ approximation of $\A$ and $\A_0=\bm{0}$. The saddle points are characterized as follows.

\begin{lemma}\label{lemma_saddle_points}
    For a given $\rho \in \{0, 1, \ldots, r_A -1 \}$, a point $(\X, \bfTheta)$ is a saddle point of $f(\X,\bfTheta)$ if $\bfTheta=\X^\top\A\X$ and $\X\bfTheta\X^\top = \A_\rho$.
\end{lemma}

Lemma \ref{lemma_saddle_points} characterizes a family of saddle points of $f(\X,\bfTheta)$ that are closely related to the best rank-$\rho$ approximations of $\A$. Building on this observation, we will show that along the optimization trajectory, overparameterized DMD successively learns $\A_\rho$ for increasing $\rho$, and ultimately converges to the ground-truth matrix $\A$.

We consider for simplicity a flow ($\eta \rightarrow 0$) variant of overparameterized DMD for solving problem \eqref{prob_dmd}, initialized at a random point $\X(0)\in\st(m,r)$ with $r>r_A$:
\begin{align}
	\dot{\X}(t) &= (\I_m-\X(t)\X(t)^\top)\A\X(t), \label{dynamics_cont}\\
	\bfTheta(t) &= \X(t)^\top\A\X(t).\nonumber
\end{align}

The saddle-to-saddle dynamics is shown in the following theorem.

\begin{theorem}\label{theorem_saddle_to_saddle}
    Consider the flow dynamics \eqref{dynamics_cont}, and suppose that $r_A\leq\frac{m}{2},\,\frac{\sigma_i(\A)}{\sigma_{i+1}(\A)}\geq10,\,i=1,\ldots,r_A-1$, and $r\geq c_rr_A$ for some constant $c_r>1$. For any $\sqrt{\frac{r}{m}}\leq\delta\leq\frac{1}{c_\delta\sqrt{\kappa}}$ with constant $c_\delta\geq\frac{6\sqrt{c_1}}{1-\frac{1}{c_r}}>1$, there exist times $0 = T_0 < T_1 < \cdots < T_{r_A-1}$, such that $\|\X(T_\rho)\bfTheta(T_\rho)\X(T_\rho)^\top - \A_\rho\|_\fro^2\le 5\delta,\,\rho=0,1,\ldots,r_A-1$ holds with high probability over the initialization. Moreover, $T_\rho$ can be upper bounded with 
    \begin{align*}
    T_\rho \leq\sum_{j=1}^\rho \frac{4}{\sigma_{j}} \log(\frac{m}{r}),  ~~~\text{for}~\rho = 1,\ldots,r_A-1.
    \end{align*}
\end{theorem}

\begin{figure}[t]
  \centering
  \hspace{0.1cm}
  \begin{subfigure}[b]{0.32\textwidth}
    \centering
    \vspace{-0.1cm}
    \includegraphics[width=\textwidth]{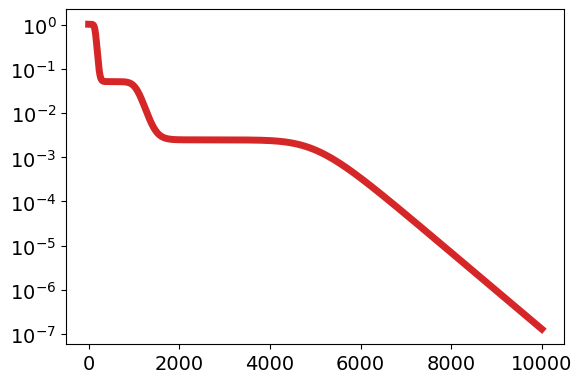}
    \caption{Squared reconstruction error}
    \label{fig:saddle_loss}
  \end{subfigure}
  \begin{subfigure}[b]{0.32\textwidth}
    \centering
    \vspace{-0.1cm}
    \includegraphics[width=\textwidth]{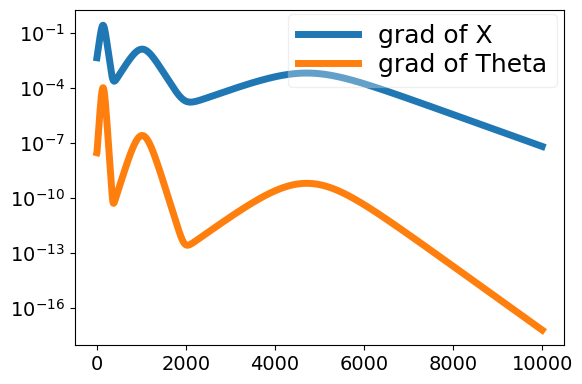}
    \caption{Gradient norm}
    \label{fig:saddle_grad}
  \end{subfigure}
  \hspace{0.1cm}
  \begin{subfigure}[b]{0.318\textwidth}
    \centering
    \vspace{-0.1cm}
    \includegraphics[width=\textwidth]{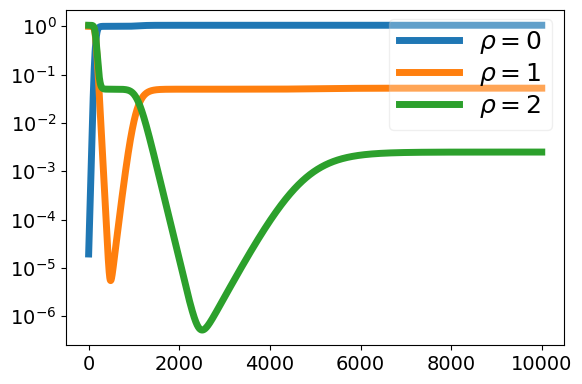}
    \caption{$\| \X_t\bfTheta_t\X_t^\top \!- \!\A_\rho \|_\fro^2$}
    \label{fig:saddle_norm}
  \end{subfigure}
  \\
  \hspace{0.25cm}
  \begin{subfigure}[b]{0.32\textwidth}
    \centering
    \includegraphics[width=\textwidth]{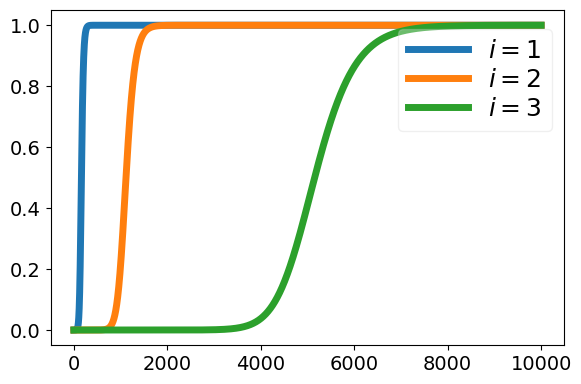}
    \caption{$\sigma_i(\bfPhi_t\bfPhi_t^\top)$}
    \label{fig:saddle_singular_Phi}
  \end{subfigure}
  \hspace{0.05cm}
   \begin{subfigure}[b]{0.305\textwidth}
    \centering
    \includegraphics[width=\textwidth]{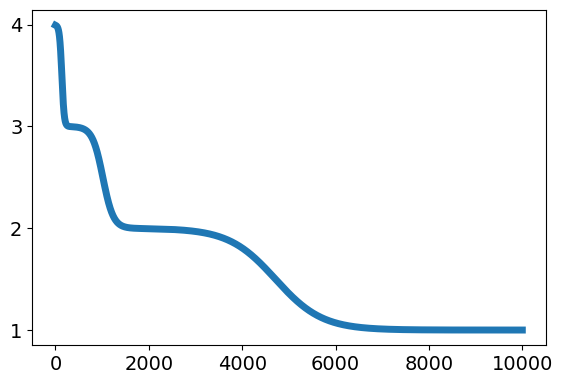}
    \caption{$\tr(\bfPsi_t\bfPsi_t^\top)$}
    \label{fig:Psi}
  \end{subfigure}
    \hspace{-0.15cm}
  \begin{subfigure}[b]{0.32\textwidth}
    \centering
    \includegraphics[width=\textwidth]{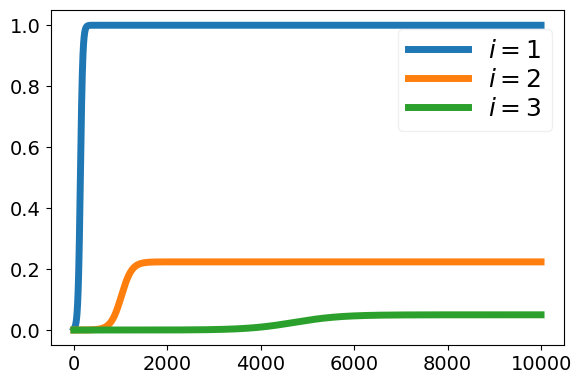}
    \caption{$\sigma_i(\bfTheta_t)$}
    \label{fig:saddle_singular_Theta}
  \end{subfigure}
  \caption{The saddle-to-saddle (i.e., sequential learning) behaviors of overparameterized DMD. The x-axis corresponds to the iteration number, and the y-axis follows the subfigure title. (a) Each plateau signifies a saddle point; (b) gradient norm at saddles drops by orders; (c) saddles strongly relate to the best rank-$\rho$ approximation of $\A$;
  (d) sequential learning in the alignment between $\X_t$ and $\U$; (e) sequential learning in the alignment between $\X_t$ and $\U_\perp$; and, (f) sequential pattern in the magnitude variable $\bfTheta_t$.}
  \label{fig:saddle}
\end{figure}

\textbf{Proof sketch:} The proof of Theorem \ref{theorem_saddle_to_saddle} is based on tracking the alignment between the direction variable $\X(t)$ and the leading eigenspaces of $\A$. Let $\A=\sum_{i=1}^{r_A}\sigma_i \uu_i\uu_i^\top$ and define the alignment variables $\phi_{i,j}(t):=\uu_i^\top \X(t)\X(t)^\top \uu_j$. In particular, $\phi_{i,i}(t)$ measures how much the current subspace $\Span(\X(t))$ captures the eigendirection $\uu_i$. Hence, learning $\A_\rho$ amounts to showing that $\phi_{i,i}(t)\to1$ for $i=1,\ldots,\rho$. The main step is to prove that these alignment variables grow sequentially. For the first direction, the spectral gap yields a logistic-type lower bound for $\phi_{1,1}(t)$, which implies that $\phi_{1,1}(t)$ grows to near one within time of $\frac{4}{\sigma_1}\log(\frac{m}{r})$. The rest of the proof proceeds by induction. Once the first $k-1$ directions have already been approximately learned, the well-separated spectrum ensures that $\phi_{k,k}(t)$ follows a logistic growth. As a result, the $k$-th direction is learned after an additional time of $\frac{4}{\sigma_k}\log(\frac{m}{r})$. Repeating this argument gives $T_\rho \leq\sum_{j=1}^\rho \frac{4}{\sigma_{j}} \log(\frac{m}{r})$. Finally, once the first $\rho$ eigendirections are aligned with $\Span(\X(T_\rho))$ and the remaining directions have only small overlap, we have $\X(T_\rho)\bfTheta(T_\rho)\X(T_\rho)^\top\approx\A_\rho$. This gives the desired reconstruction error bound at each stage.

This theorem shows that when the ratio $\frac{r}{m}$ is small, which is typically the case in low-rank settings where $r_A \ll m$ and $r$ is chosen slightly larger than $r_A$, overparameterized DMD is initialized in the vicinity of $\A_0=\bm{0}$ and then successively approaches and escapes from $\A_1, \A_2, \ldots, \A_{r_A-1}$. In other words, during the early stage, overparameterized DMD passes near a sequence of $r_A$ saddle points. 

This behavior is consistent with our empirical observations in Figure \ref{fig:saddle}, which traces the optimization trajectory of Algorithm \ref{alg_DMD} applied to \eqref{prob_dmd} on an instance of $m=2000$, $r_A=3$, $r=4$, and $\kappa=20$, using random initialization and stepsizes $\eta=0.02$ and $\mu=2$. Figure \ref{fig:saddle_loss} depicts the squared reconstruction error across iterations, where each plateau marks a saddle-escape event. This interpretation is further supported by Figure \ref{fig:saddle_grad}, which shows that the gradient norm becomes small whenever the trajectory approaches a saddle point. Furthermore, Figure \ref{fig:saddle_norm} confirms that these saddles are exactly those characterized in Lemma \ref{lemma_saddle_points}: the quantities $\|\X_t \bfTheta_t \X_t^\top - \A_\rho\|_\fro^2$ for $\rho = 0,\ldots,r_A-1$ approach zero sequentially, indicating that iterates successively visit neighborhoods of the low-rank approximations $\A_0,\A_1,\ldots,\A_{r_A-1}$ before converging to $\A$.

In addition, the direction and magnitude variables also exhibit a clear sequential learning pattern. For the direction variable $\X_t$, the singular values of $\bfPhi_t \bfPhi_t^\top$, which represent the squared cosines of the principal angles between $\Span(\X_t)$ and $\Span(\U)$ \cite{bjorck1973numerical}, are shown in Figure \ref{fig:saddle_singular_Phi}. These quantities increase one by one, indicating that $\X_t$ learns the eigenspace of the objective matrix $\A$ sequentially. Let $\U_\perp \in \mathbb{R}^{m \times (m-r_A)}$ denote an orthonormal basis of the orthogonal complement of $\Span(\U)$. The matrix $\bfPsi_t:=\U_\perp^\top\X_t$ measures the component of $\X_t$ outside the target eigenspace. As shown in Figure \ref{fig:Psi}, the singular values of $\bfPsi_t$ decrease correspondingly, confirming that $\X_t$ gradually eliminates directions orthogonal to the true subspace. Geometrically, this process can be viewed as a sequence of subspace-identification events, where one principal angle vanishes at a time. Each saddle-escape event therefore corresponds to the discovery of an additional eigenvector direction and moves the trajectory from the neighborhood of $\A_\rho$ toward that of $\A_{\rho+1}$.

Meanwhile, the magnitude variable $\bfTheta_t$ adapts to the directions discovered by $\X_t$. Its singular values increase sequentially, indicating that the algorithm progressively captures the magnitudes associated with the aligned directions; see Figure \ref{fig:saddle_singular_Theta}.

Another noteworthy observation is that escaping saddle points requires more iterations at later stages; see Figure \ref{fig:saddle_loss}. The following lemma provides a theoretical explanation.

\begin{lemma}\label{lemma_saddle_lower_bound}
    Consider the same flow dynamics as in Theorem \ref{theorem_saddle_to_saddle} with randomly initialized $\X(0)\in\st(m,r),r>r_A$. For any $t\geq0$, $0\leq \rho\leq r_A-1$, with high probability over the initialization, it holds that 
    $$
        \|\X(t)\bfTheta(t)\X(t)^\top-\A_\rho\|_\fro^2\geq\sum_{j=\rho+1}^{r_A}\sigma_\rho^2\cdot\frac{(r-r_A)^8}{c_1^4m^4r^4}.
    $$
\end{lemma}

This lemma establishes a uniform lower bound on the distance between the optimization trajectory and each saddle point. Since the bound decreases with $\rho$, the trajectory can approach later saddles more closely than earlier ones, as evidenced by the minimum value of $\|\X_t\bfTheta_t\X_t^\top  - \A_\rho\|_\fro^2$ in Figure \ref{fig:saddle_norm}. As a result, more iterations are needed to escape later saddles.

\section{Recursive DMD}\label{sec:recursive}
As analyzed in Section~\ref{sec:saddle_phase}, the saddle-to-saddle dynamics of overparameterized DMD is equivalent to incrementally recovering the best rank-$\rho$ approximation of $\A$. Interestingly, even when rank-$r$ variables $\X\in\mathbb{R}^{m\times r}$ and $\bfTheta\in\mathbb{S}^r$ are used, the optimization dynamics still proceed by sequentially extracting one singular direction with its associated singular value at a time. In this sense, the learning process is rank-by-rank.

The sequential nature of the learning dynamics questions the necessity of overparameterization. Indeed, a natural alternative that directly follows this sequential behavior is to use a rank-1 DMD parameterization, with vector direction $\x\in\mathbb{R}^m$ and scalar magnitude $\theta\in\mathbb{R}$, to recover the leading eigenpair of $\A$, and then repeat this procedure to progressively reconstruct $\A$. We refer to this method as recursive DMD. While it is clearly more memory efficient than the rank-r DMD, in this section, we examine the subtle tradeoffs more closely.

\begin{algorithm}[t]
\caption{Rank-1 DMD for solving \eqref{prob_dmd_rank1}}\label{alg_rank1}
\begin{algorithmic}[1]
\State \textbf{Input:} Initial point $\x_0$, objective matrix $\A$, stepsizes $\eta, \mu$, and number of iterations $T$
\For{$t = 0, 1, \ldots, T$}
    \State Update the magnitude variable $\theta_{t}=\theta_{t-1}-\frac{\mu}{2}(\theta_{t-1}-\x_t^\top\A\x_t)$
    \State Calculate $\y_{t+1}=\x_t+\eta(\I_{m}-\x_t\x_t^\top)\A\x_t$
    \State Update the direction variable $\x_{t+1}=\y_{t+1}/\|\y_{t+1}\|$
\EndFor
\State \textbf{Output:} $\x_T,\theta_T$
\end{algorithmic}
\end{algorithm}

\subsection{Rank-1 DMD for leading eigenpair learning} We first formalize the mechanism of rank-1 DMD and analyze its convergence properties, which will serve as the building block for the recursive DMD introduced later.

Setting $r=1$ in problem \eqref{prob_dmd}, the direction and magnitude variables reduce to a vector and a scalar, respectively. As a result, the optimization problem becomes:
\begin{align}\label{prob_dmd_rank1}
    \min\limits_{\x\in\st(m,1),\theta\in\mathbb{R}}g(\x,\theta):=\frac{1}{4}\|\theta\x\x^\top-\A\|_\fro^2.
\end{align}
In this case, the Stiefel manifold $\st(m,1)=\{\x\in\mathbb{R}^m|\x^\top\x=1\}$ reduces to a sphere. This further simplifies the optimization for direction variables. At iteration $t$, the Riemannian update using surrogate loss can be written as $\x_{t+1}=\frac{\y_{t+1}}{\|\y_{t+1}\|}
$, where $\y_{t+1}=\x_t+\eta(\I_m-\x_t\x_t^\top)\A\x_t$. In summary, the resulting rank-1 DMD scheme is presented in Algorithm \ref{alg_rank1}. 

Next, we establish its convergence properties.
Consider solving \eqref{prob_dmd_rank1} under random initialization, where the initial direction $\x_0$ is sampled uniformly from the Stiefel manifold $\st(m,1)$. Such an initialization can be generated by drawing $\z_0\in\mathbb{R}^m$ with i.i.d. Gaussian entries $\mathcal{N}(0,1)$ and setting $\x_0=\frac{\z_0}{\|\z_0\|}$ \cite{chikuse2012statistics}.
\begin{theorem}\label{theorem_DMD_rank1}
    Consider solving the rank-1 approximation problem \eqref{prob_dmd_rank1} initialized with a random $\x_0\in\st(m,1)$. Suppose that $m\geq3$ and $\A\in\mathbb{S}^m$ is a symmetric matrix of rank $r_A$. Let the compact eigendecomposition of $\A$ be $\A=\U\bfLambda\U^\top$, where $\U\in\mathbb{R}^{m\times r_A}$ has orthonormal columns, $\bfLambda=\diag(\lambda_1,\lambda_2,\ldots,\lambda_{r_A})$ with $\lambda_1>\lambda_2\geq\ldots\geq\lambda_{r_A}$, $\lambda_1>0$ and $|\lambda_{r_A}|<\lambda_1$.

    For any optimality error $\varepsilon\in(0,1)$,  with high probability over the initialization, Algorithm \ref{alg_rank1} using stepsizes $\eta\leq\mathcal{O}(\frac{1}{\lambda_1})$ and $\mu=2$ guarantees $\|\theta_{t_\varepsilon}\x_{t_\varepsilon}\x_{t_\varepsilon}^\top-\A_1\|_\fro\leq\mathcal{O}(\lambda_1\varepsilon)$ after at most $t_\varepsilon=\mathcal{O}(\frac{\log(m)+\log(\frac{1}{\varepsilon})}{\eta(\lambda_1-\lambda_2)})$ iterations. 
\end{theorem}
\begin{remark}
    Here we only require $\A$ to be symmetric, rather than positive semidefinite. In particular, for any PSD matrix $\A$, as long as $\sigma_1(\A)>\sigma_2(\A)$, the eigenvalue assumptions in Theorem \ref{theorem_DMD_rank1} are satisfied. Moreover, by the Eckart–Young–Mirsky theorem~\cite{golub1987generalization}, the best rank-1 approximation $\A_1$ is given by $\A_1=\lambda_1\uu_1\uu_1^\top$, where $\uu_1$ denotes the first column of $\U$.
\end{remark}

This theorem shows that rank-1 DMD can achieve exact convergence to the best rank-1 approximation of $\A$ in a linear rate. Now suppose that Algorithm \ref{alg_rank1} returns $(\theta_T, \x_T)$ an exact solution, i.e., $\theta_T\x_T\x_T^\top=\lambda_1\uu_1\uu_1^\top$, one can perform a deflation step by setting $\A^{(1)}:=\A-\theta_T\x_T\x_T^\top$ and then apply Algorithm \ref{alg_rank1} to $\A^{(1)}$ to learn its leading eigenpair, which corresponds to the second leading eigenpair of $\A$. We next show that this intuition remains valid even when the output of Algorithm \ref{alg_rank1} contains approximation error.

\subsection{Recursive DMD for problem \eqref{prob_dmd}}

By recursively applying Algorithm \ref{alg_rank1} together with deflation, we obtain the recursive scheme for solving problem \eqref{prob_dmd}. The resulting algorithm is summarized in Algorithm \ref{alg_recursive_rank1}, which we refer to as recursive DMD.

\begin{algorithm}[t]
    \caption{Recursive DMD for solving \eqref{prob_dmd}}\label{alg_recursive_rank1}
    \begin{algorithmic}[1]
    \State \textbf{Input:} Objective matrix $\A$, target error $\varepsilon$, largest number of rounds $r$, stepsizes $\eta, \mu$, and number of iterations $T_{\star}$ for each round of rank-1 DMD
    \State \textbf{Initialize:} $\A^{(0)}=\A$
    \For{$j = 0,1, \ldots, r-1$}
        \State Generate random $\x_0^{(j)}\in\st(m,1)$
        \State Run Algorithm \ref{alg_rank1} with inputs $\x_0^{(j)},\A^{(j)},\eta,\mu$, and $T_{\star}$ \Comment{rank-1 DMD}
        \State Obtain the output $\x_{T_{\star}}^{(j)},\theta_{T_{\star}}^{(j)}$
        \State Update $\A^{(j+1)}=\A^{(j)}-\theta_{T_{\star}}^{(j)}\x_{T_\star}^{(j)}(\x_{T_\star}^{(j)})^\top$ 
        \Comment{Deflation}
        \If{$\frac{1}{4}\|\A^{(j+1)}\|_\fro^2 \le \varepsilon$}
            \State Set $J = j + 1$
            \State \textbf{break}
        \EndIf
    \EndFor
    \State \textbf{Output:} $\x_{T_{\star}}^{(j)},\theta_{T_{\star}}^{(j)},j=0,\ldots,J-1$
    \end{algorithmic}
    \end{algorithm}

Recall that the rank of $\A\in\mathbb{S}_+^m$ is denoted by $r_A$. As before, we assume that the largest singular value of $\A$ is $\sigma_1(\A)=1$ and the smallest nonzero singular value is $\sigma_{r_A}(\A)=\frac{1}{\kappa}$, where $\kappa>1$ represents the condition number of $\A$.
\begin{theorem}\label{theorem_recursive_rank1}
    Consider solving the matrix factorization problem \eqref{prob_dmd} using Algorithm \ref{alg_recursive_rank1} with stepsizes $\eta=\mathcal{O}(1),\mu=2$ and largest number of rounds $r>r_A$. For any $\varepsilon\in(0,\frac{1}{\kappa^2})$, we set the number of iterations $T_{\star}=\mathcal{O}\big(\kappa(\log(m)+r_A)+\kappa\log(\frac{1}{\varepsilon})\big)$ for each round of rank-1 DMD. Suppose that $m\geq3$ and $\frac{\sigma_i(\A)}{\sigma_{i+1}(\A)}\geq2,i=1,\ldots,r_A-1$. With high probability over the initialization, we have that $\frac{1}{4}\|\sum_{j=0}^{J-1}\theta_{T_\star}^{(j)}\x_{T_\star}^{(j)}(\x_{T_\star}^{(j)})^\top-\A\|_\fro^2\leq\varepsilon$ after at most $T_\varepsilon=\mathcal{O}\big(\kappa(r_A\log(m)+r_A^2)+r_A\kappa\log(\frac{1}{\varepsilon})\big)$ iterations in total. Here $J$ is the total number of calls to Algorithm \ref{alg_rank1}, and satisfies $J\leq r_A$.
\end{theorem}
\begin{remark}
    We define $\X=[\x_{T_\star}^{(0)},\ldots,\x_{T_\star}^{(J-1)}]\in\mathbb{R}^{m\times J}$ and $\bfTheta=\diag(\theta_{T_\star}^{(0)},\ldots,\theta_{T_\star}^{(J-1)})\in\mathbb{S}^J$, then $\frac{1}{4}\|\X\bfTheta\X^\top-\A\|_\fro^2\leq\varepsilon$. Although $\X$ is not necessarily orthonormal and thus does not lie on a Stiefel manifold, the resulting matrix $\X\bfTheta\X^\top$ is nevertheless positive semidefinite, since  $\theta_{T_\star}^{(j)}>0,j=0,\ldots,J-1$ (see the second part of the proof of Theorem \ref{theorem_recursive_rank1}). Consequently, this yields a low-rank positive semidefinite factorization that approximates $\A$.
\end{remark}

As established in Theorem \ref{theorem_recursive_rank1}, although there are no guarantees to precisely recover the leading eigenpair of $\A^{(j)}$, the recursive DMD can still drive the reconstruction error below any target value $\varepsilon$.

\subsection{Comparison of overparameterized DMD with recursive DMD}\label{sec:comparison_op_re}

We now give a comparison between overparameterized DMD and recursive DMD across several key aspects.

\textit{Working memory:} For overparameterized DMD, one needs to maintain the direction variable $\X\in\mathbb{R}^{m\times r}$ and the magnitude variable $\bfTheta\in\mathbb{S}^{r}$ throughout the optimization, resulting in a working memory of $\mathcal{O}(mr+r^2)$. In contrast, recursive DMD only updates a single pair $(\x,\theta)$ with $\x\in\mathbb{R}^m$ and $\theta\in\mathbb{R}$. Hence, the working memory during each update step is merely $\mathcal{O}(m)$. This enables recursive DMD to scale to larger problem sizes under fixed memory constraints.

\textit{Per-iteration complexity:} For overparameterized DMD, each iteration requires computing the Riemannian gradient $\G_t$, which costs $\mathcal{O}(m^2r+mr^2)$, and performing a polar retraction, which costs $\mathcal{O}(mr^2+r^3)$. Since $r\le m$, the per-iteration computational complexity of overparameterized DMD is dominated by $\mathcal{O}(m^2r)$. In contrast, each iteration of recursive DMD is only $\mathcal{O}(m^2)$, lower than that of overparameterized DMD by a factor proportional to $r$.

We also note that the $\kappa$-dependence of the iteration complexity for recursive DMD scales as $\mathcal{O}(\kappa)$, improving upon the $\mathcal{O}(\kappa^2)$ scaling of overparameterized DMD. This improvement stems from the \textit{additional} assumption on the well-separated spectrum of $\A$, i.e., $\frac{\sigma_i(\A)}{\sigma_{i+1}(\A)}\geq2,i=1,\ldots,r_A-1$.  We further conjecture that the $\kappa$-dependence of the iteration complexity for overparameterized DMD may be improved to $\mathcal{O}(\kappa)$ under the separated spectrum assumptions as well, and it is left for future research.

A detailed numerical comparison between overparameterized DMD and recursive DMD is provided in Section \ref{sec:num_comparison_OP_RE}.

\section{Numerical experiments}\label{sec:num_exp}
In this section, we conduct numerical experiments to validate our theoretical results of overparameterized DMD and recursive DMD\footnote{All the experiments are conducted on a MacBook Pro equipped with an Apple M2 Max processor using MATLAB R2022b.}. In the following experiments, the objective matrix is generated as $\A=\U\bfSigma\U^\top\in\mathbb{R}^{m\times m}$, where $\U\in\mathbb{R}^{m\times r_A}$ is a random matrix with orthonormal columns, and $\bfSigma\in\mathbb{S}_+^{r_A}$ is a diagonal matrix with entries evenly spaced on a logarithmic scale over the interval $[1/\kappa,1]$, yielding condition number $\kappa$. 

\begin{figure}[t]
  \centering
  \begin{subfigure}[b]{0.325\textwidth}
    \centering
    \includegraphics[width=\textwidth]{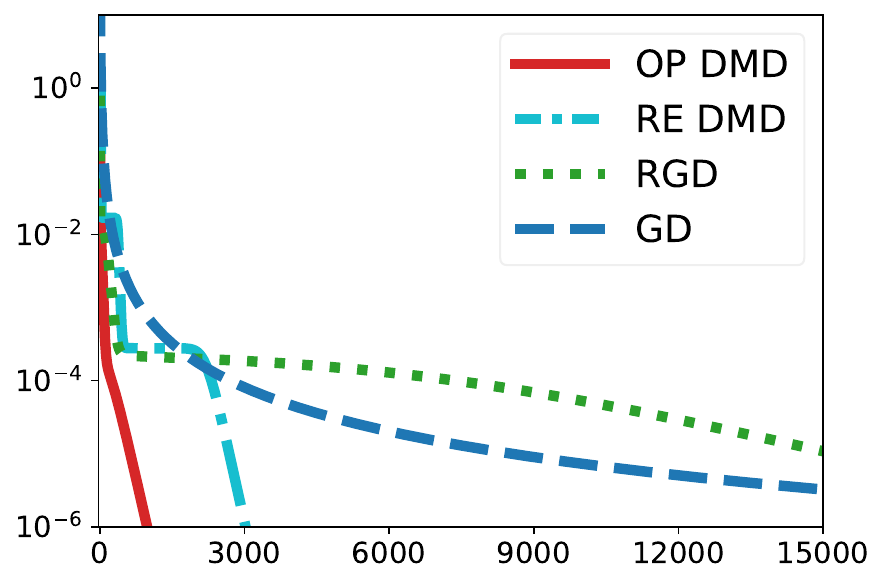}
    \caption{Matrix factorization}
    \label{fig:all_fac}
  \end{subfigure}
    \begin{subfigure}[b]{0.325\textwidth}
    \centering
    \includegraphics[width=\textwidth]{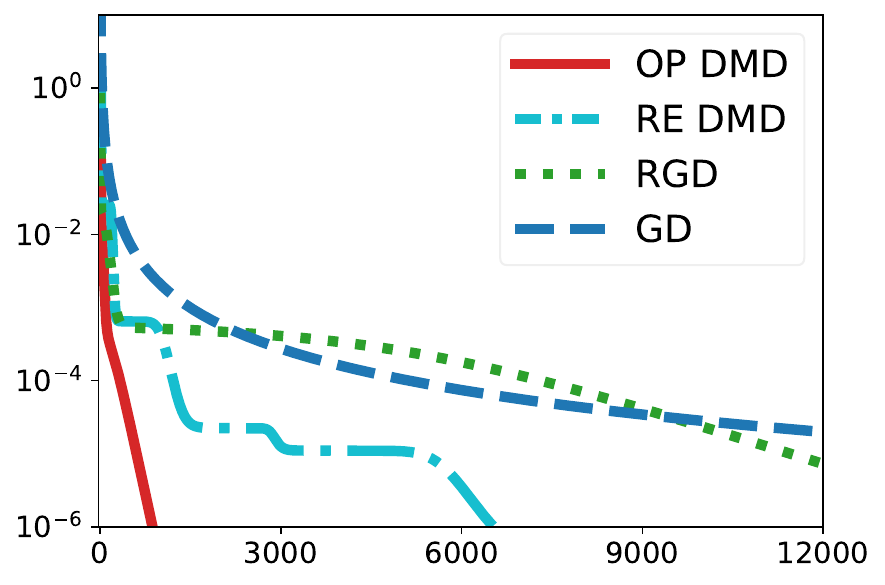}
    \caption{Matrix sensing}
    \label{fig:all_sensing}
  \end{subfigure}
  \begin{subfigure}[b]{0.325\textwidth}
    \centering
    \includegraphics[width=\textwidth]{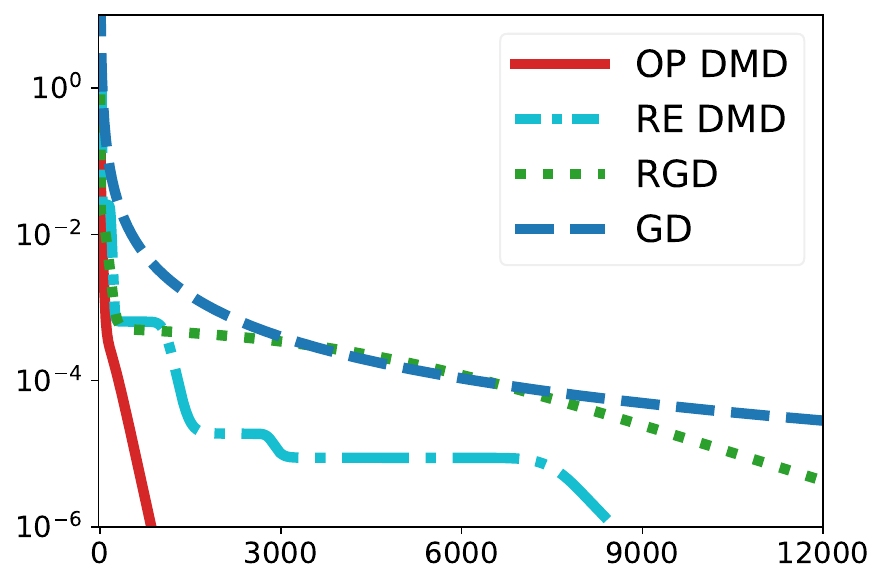}
    \caption{Matrix completion}
    \label{fig:all_completion}
  \end{subfigure}
  \caption{Comparison of overparameterized DMD (OP DMD) and recursive DMD (RE DMD) with RGD and GD on matrix factorization, matrix sensing, and matrix completion (squared reconstruction error vs. iterations).}
  \label{fig:all}
\end{figure}

\subsection{Faster convergence of overparameterized and recursive DMD}\label{sec:comparison_all}
We compare overparameterized DMD and recursive DMD with RGD \cite{lion2025polar} and GD on matrix factorization, matrix sensing and matrix completion.
All experiments are conducted with random initialization. Specifically, overparameterized DMD and RGD are initialized with $  \X_0=\Z_0(\Z_0^\top\Z_0)^{-1/2}$, where $\Z_0\in\mathbb{R}^{m\times r}$ has i.i.d. standard Gaussian entries. For GD, we set $\Y_0=\X_0$. For recursive DMD, we initialize $\x_0=\frac{\z_0}{\|\z_0\|}$,
where $\z_0\in\mathbb{R}^{m}$ also has i.i.d. standard Gaussian entries.

As natural extensions to the factorization problem, 
matrix sensing \cite{li2018algorithmic, zhong2015efficient, tong2021accelerating, shen2026escaping} aims to recover the objective matrix $\A \in \mathbb{S}_+^m$ from a collection of $n$ measured data $\{(\M_i,y_i)\}_{i=1}^n$, where each sensing matrix $\M_i\in\mathbb{S}^m$ is symmetric and the corresponding observation is $y_i=\tr(\M^\top_i\A)$. Matrix completion \cite{candes2012exact, keshavan2010matrix, ongie2021tensor, eriksson2012high} refers to the recovery of $\A$ from a set of observed entries. Let $\mathcal M:\mathbb S^m\to\mathbb R^n$ denote the sensing operator defined by $[\mathcal M(\A)]_i=\tr(\M_i^\top\A)$, and let $\mathcal M_{\bfOmega}:\mathbb S^m\to\mathbb S^m$ denote the sampling operator $\mathcal M_{\bfOmega}(\A)=\bfOmega\odot\A$, where $\bfOmega\in\{0,1\}^{m \times m}$ is a symmetric sampling mask and $\odot$ denotes element-wise product. With the same spirit of \eqref{prob_dmd}, these problems can be written in the DMD form as
\begin{align}\label{prob_sensing_bm}
    \text{\emph{Matrix sensing: }}\min_{\X\in\st(m,r),\bfTheta\in\mathbb{S}^r} \; \frac{1}{4} \| \mathcal{M}(\X \bfTheta \X^\top) - \mathcal{M}(\A) \|^2 .
\end{align}
\vspace{-1em}
\begin{align}\label{prob_completion_bm}
    \text{\emph{Matrix completion: }}\min_{\X\in\st(m,r),\bfTheta\in\mathbb{S}^r} \; \frac{1}{4} \| \mathcal{M}_\bfOmega(\X \bfTheta \X^\top) - \mathcal{M}_\bfOmega(\A) \|_\fro^2 .
\end{align}
Substituting $\X \bfTheta \X^\top$ with $\Y\Y^\top$ for ${\Y \in \mathbb{R}^{m \times r}}$ gives the classical BM formulation. 

Overparameterized DMD extends naturally to \eqref{prob_sensing_bm} and \eqref{prob_completion_bm} by replacing the full observation operator in factorization with the corresponding sensing or sampling operator. Specifically, for matrix sensing, the surrogate gradient in \eqref{rgd_X} becomes $\hat{\G}_t=\mathcal{M}^*\mathcal{M}(\X_t\X_t^\top-\A)\X_t$. Here, $\mathcal{M}^*:\mathbb{R}^n\to\mathbb{S}^{m}$ is the adjoint of operator $\mathcal{M}$. The magnitude variable $\bfTheta$ is updated by $\bfTheta_{t}=\bfTheta_{t-1}-\frac{\mu}{2}\X_t^\top[\mathcal{M}^*\mathcal{M}(\X_t\bfTheta_{t-1}\X_t^\top-\A)]\X_t$, with $\bfTheta_0$ set as $\X_0^\top[\mathcal{M}^*\mathcal{M}(\A)]\X_0$. For matrix completion, the updates are obtained similarly by replacing $\mathcal{M}$ with the sampling operator $\mathcal{M}_{\bfOmega}$. When we apply recursive DMD, each round of rank-1 DMD is terminated once $\|\theta_t\x_t\x_t^\top-\theta_{t+10}\x_{t+10}\x_{t+10}^\top\|_\fro\leq10^{-7}$, after which a deflation step is performed. In matrix factorization, the deflation step follows Algorithm \ref{alg_recursive_rank1}, while in matrix sensing and completion, deflation is applied directly to the sensed or observed measurements, respectively, using the linearity of the sensing and sampling operators. The same operator substitution is also applied to RGD and GD, following the update schemes in \cite{lion2025polar} and \cite{xiong2024how}.

For matrix factorization, we consider an instance of $m=500,r_A=3,r=30$, and $\kappa=60$. Overparameterized DMD, recursive DMD and RGD are applied to \eqref{prob_dmd} with stepsizes $\eta=0.25$ and $\mu=2$, while GD is applied to \eqref{prob_bm} with stepsize $\eta=0.25$. For matrix sensing, we set $m=200,r_A=3,r=30$, and $\kappa=40$. The number of sensing matrices is $n=8000$. These matrices are generated as $\M_i=\frac{1}{2\sqrt{n}}(\mathbf{R}_i+\mathbf{R}_i^\top),i=1,\ldots,n$, where $\mathbf{R}_i\in\mathbb{R}^{m\times m}$ has i.i.d. standard Gaussian entries. For matrix completion, we set $m=300$, $r_A=3$, $r=40$, and $\kappa=40$. The sampling mask is generated as $\bfOmega=\frac{1}{2}(\bfXi+\bfXi^\top)$, where each entry of $\bfXi\in\mathbb{R}^{m\times m}$ is independently set to $1$ with probability $p=0.8$. For both matrix sensing and completion, overparameterized DMD, recursive DMD, and RGD are run with stepsizes $\eta=0.15$ and $\mu=2$, while GD is run with stepsize $\eta=0.15$.

Figures \ref{fig:all_fac}-\ref{fig:all_completion} plot the squared reconstruction error versus the number of iterations. Across the three experiments, we observe that overparameterized DMD converges faster than both GD and RGD in the initial phase, which is consistent with the weaker $\kappa$-dependence of our iteration complexity analysis. After this stage, overparameterized DMD enters a linear convergence regime and continues to decrease the error until exact recovery of the true matrix is achieved. In contrast, RGD requires a substantially longer time to escape saddle points, subsequently exhibits a slower linear convergence rate than overparameterized DMD, in agreement with strictly improved iteration complexity of overparameterized DMD relative to RGD. Moreover, GD quickly slows down to a sublinear rate after the initial phase, resulting in significantly larger reconstruction errors at the same iteration count. Recursive DMD, although requiring more iterations than overparameterized DMD, still attains a comparable linear convergence behavior and outperforms both RGD and GD.

\subsection{Overparameterized DMD converges faster with larger $r$}
Next, we demonstrate that overparameterized DMD leverages overparameterization for faster convergence. To this end, we consider problem instances of \eqref{prob_dmd} and \eqref{prob_bm} under different $r$. In this experiment, we focus on a setting with $m=1000$, $r_A=3$, and $\kappa=100$. The level of overparameterization is chosen from $r \in \{10, 100, 300\}$. Overparameterized DMD is run on \eqref{prob_dmd} with stepsizes $\eta=0.02$ and $\mu=2$, and GD is run on \eqref{prob_bm} with stepsize $\eta=0.02$.

\begin{wrapfigure}{r}{0.477\textwidth}
    \vspace{-1em}
    \centering
    \includegraphics[width=0.477\textwidth]{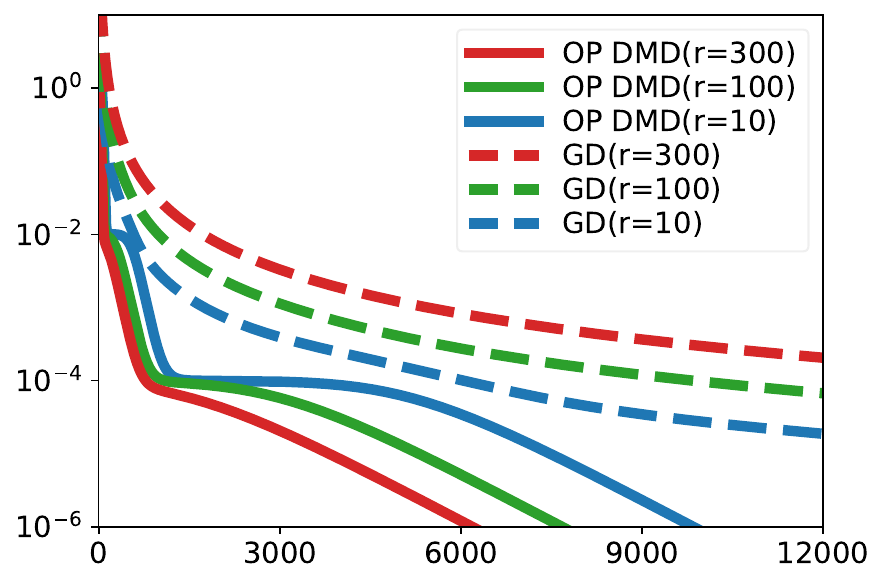}
    \caption{Comparison of overparameterized DMD (OP DMD) with GD under different $r$.}
    \label{fig:dif_r}
    \vspace{-1em}
\end{wrapfigure}

The squared reconstruction error versus the number of iterations is plotted in Figure \ref{fig:dif_r}. The results show that overparameterized DMD converges faster as $r$ increases. This behavior is consistent with our analysis in Theorem \ref{theorem_DMD}. 
In comparison, although the theoretical iteration complexity of GD given by \cite{xiong2024how} is independent of $r$, our empirical results indicate that a larger $r$ leads to slower convergence. Moreover, Figure \ref{fig:dif_r} clearly shows that overparameterized DMD escapes from the saddle phase faster with larger $r$, as reflected in shorter plateaus or earlier onset of linear convergence. This aligns well with our theoretical observations and discussions in Sections \ref{sec:benefits_DMD} and \ref{sec:saddle_phase}.

\subsection{Numerical comparison of overparameterized DMD with recursive DMD}\label{sec:num_comparison_OP_RE}

\begin{figure}[t]
\centering
  \begin{subfigure}[b]{0.48\textwidth}
    \centering
    \includegraphics[width=\textwidth]{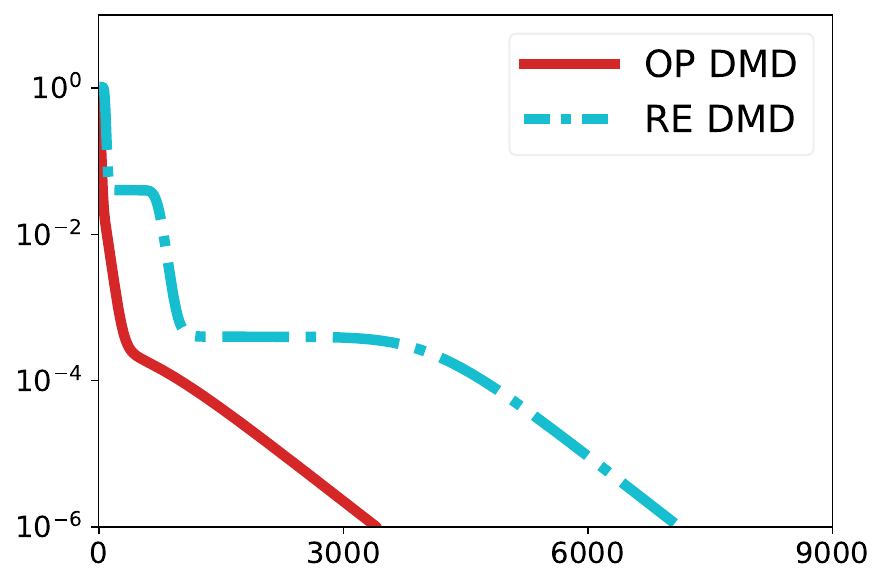}
    \caption{Well-separated (error vs. iterations)}
    \label{fig:well_sep}
  \end{subfigure}
  \hfill
  \begin{subfigure}[b]{0.48\textwidth}
    \centering
    \includegraphics[width=\textwidth]{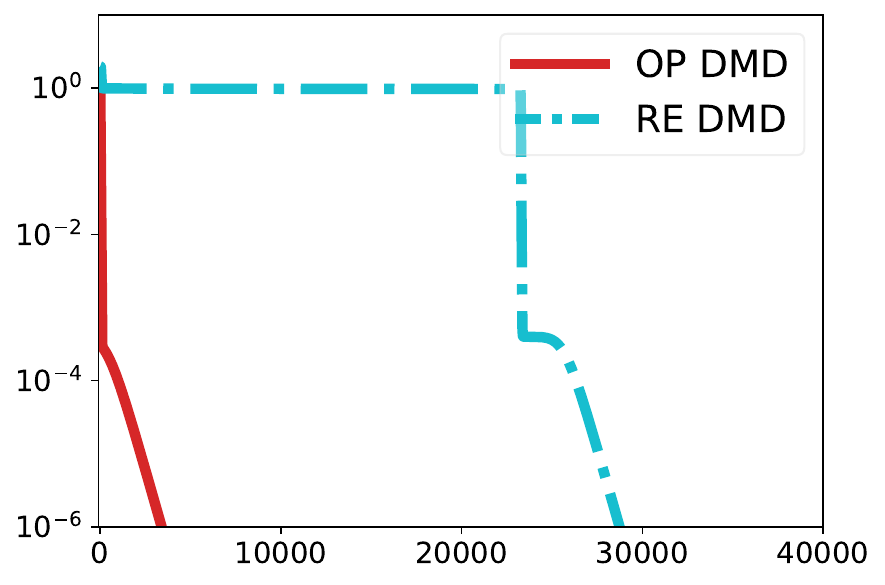}
    \caption{Not well-separated (error vs. iterations)}
    \label{fig:not_wel_sep}
  \end{subfigure}

  \begin{subfigure}[b]{0.48\textwidth}
    \centering
    \includegraphics[width=\textwidth]{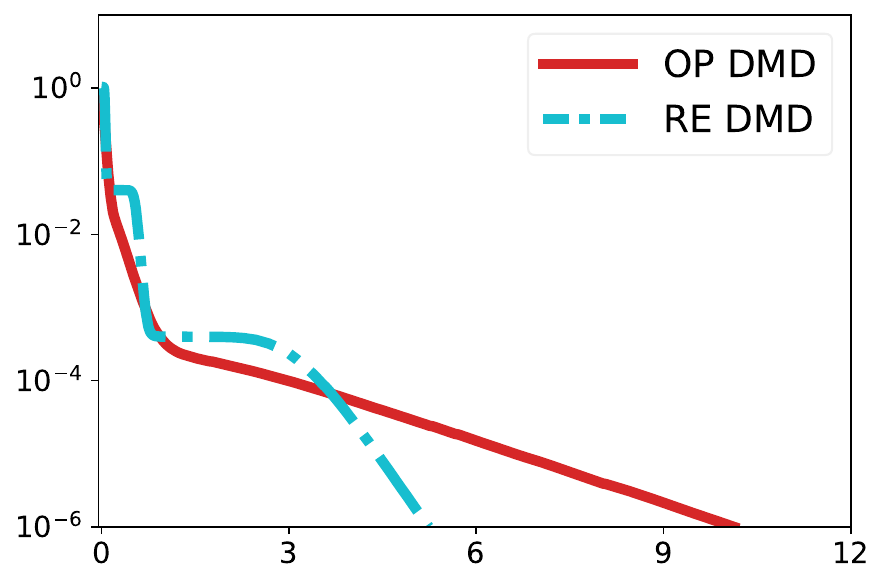}
    \caption{Well-separated (error vs. time)}
    \label{fig:well_sep_time}
  \end{subfigure}
  \hfill
  \begin{subfigure}[b]{0.48\textwidth}
    \centering
    \includegraphics[width=\textwidth]{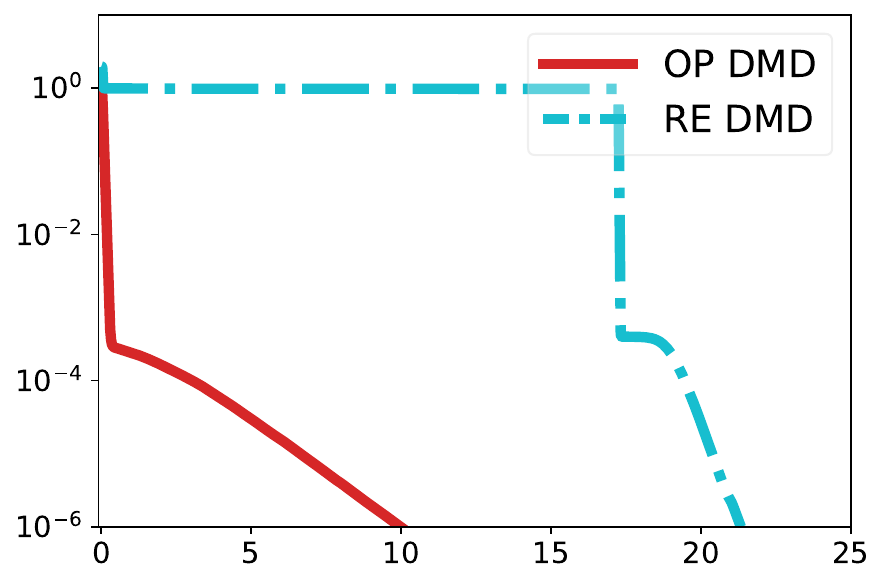}
    \caption{Not well-separated (error vs. time)}
    \label{fig:not_wel_sep_time}
  \end{subfigure}

  \caption{Comparison of overparameterized DMD (OP DMD) with recursive DMD (RE DMD) on matrix factorization under different singular value separations. Top row: squared reconstruction error vs. iterations. Bottom row: squared reconstruction error vs. time.}
  \label{fig:recursive_comparison}
\end{figure}

Lastly, we present two experiments of matrix factorization to show the tradeoffs between overparameterized DMD and recursive DMD. In the first experiment, we consider a rank-$3$ objective matrix $\A$ with well-separated singular values, $\sigma_1(\A)=1,\sigma_2(\A)=0.2,\sigma_3(\A)=0.02$. We apply both overparameterized DMD and recursive DMD to problem \eqref{prob_dmd}. The problem dimensions are set to $m=200$, $r_A=3$, and $r=30$. In the second experiment, we consider a rank-$3$ objective matrix whose leading singular values are not well separated, $\sigma_1(\A)=1,\sigma_2(\A)=0.99,\sigma_3(\A)=0.02$. The same problem dimensions, namely $m=200$, $r=30$, and $r_A=3$, are used. For both experiments, we set the stepsizes to $\eta=0.05$ and $\mu=2$ for the two algorithms, and set the target error to $\varepsilon=10^{-6}$. In each round of rank-1 DMD, the iteration is terminated once $\|\theta_t\x_t\x_t^\top-\theta_{t+10}\x_{t+10}\x_{t+10}^\top\|_\fro\leq10^{-7}$.

As shown in Figures \ref{fig:well_sep} and \ref{fig:not_wel_sep}, recursive DMD generally requires more iterations than overparameterized DMD to reach the same target error. However, each iteration of recursive DMD is significantly cheaper. Table \ref{tab:computation_time_op_re} shows that the per-iteration computational time of recursive DMD is substantially lower than that of overparameterized DMD. Consequently, recursive DMD achieves a smaller overall runtime when the singular values are well-separated; see Figure \ref{fig:well_sep_time}. When the spectrum is not well-separated, the extra iterations offset the per-iteration savings, leading to a longer runtime; see Figure \ref{fig:not_wel_sep_time}.

In a nutshell, when $\A$ exhibits well-separated singular values, recursive DMD is preferable for memory and computational efficiency. Otherwise, overparameterized DMD provides a more robust alternative.

\section{Conclusion}
\label{sec:conclusions}

\begin{table}[t]
\centering
\caption{Computational time of overparameterized DMD (OP DMD) and recursive DMD (RE DMD) on matrix factorization with different singular value separations.}
\vspace{1em}
\label{tab:computation_time_op_re}
\begin{tabular}{llccc}
\toprule
Instance & Algorithm & Iterations & Total time (s) & Time per-iteration (s) \\
\midrule
Well-separated 
& OP DMD & 3512 & 10.13 & $2.88\times10^{-3}$ \\
& RE DMD & 7120 & 5.24 & $0.74\times10^{-3}$ \\
\midrule
Not well-separated 
& OP DMD & 3405 & 9.98 & $2.93\times10^{-3}$ \\
& RE DMD & 28740 & 21.33 & $0.74\times10^{-3}$ \\
\bottomrule
\end{tabular}
\end{table}

In this work, we propose direction–magnitude decomposition (DMD) as a unified framework for low-rank matrix optimization that improves efficiency when the target rank is unknown. By decoupling direction and magnitude components, DMD provides a principled approach to overcoming the optimization challenges induced by overparameterization. We develop overparameterized DMD and establish its theoretical advantages on matrix factorization, including improved convergence rates and more favorable condition number dependence. Furthermore, we characterize a saddle-to-saddle dynamical behavior that reveals an incremental eigenpair learning mechanism underlying the optimization process. Building on this perspective, we propose recursive DMD that achieves reduced memory and computational complexity while retaining strong convergence guarantees. Extensive numerical experiments on matrix factorization, matrix sensing, and matrix completion corroborate our theoretical findings and demonstrate the practical effectiveness of the proposed framework. Overall, our results suggest that direction–magnitude decomposition offers a promising paradigm for efficient low-rank matrix optimization beyond the exactparameterized setting.

\bibliographystyle{plain}
\bibliography{references}

\clearpage
\addcontentsline{toc}{section}{Appendix}
\begin{center}
\vspace*{20pt}
{\large\bfseries Appendix}
\end{center}
\appendix

\section{Deferred proofs}\label{supplement.proofs}
\subsection{Proof of Theorem \ref{theorem_DMD}}\label{supplement.proofs.theorem_DMD}
\leavevmode\par
\begin{proof}
    For simplicity, we take $\eta=\frac{c_\eta(r-r_A)^2}{mrr_A\kappa}$, where $c_\eta=\frac{1}{8c_1}$ is a universal constant.
    
    By using the definition $\bfPhi_t$, we can write $\U^\top\G_t$ and $\G_t^\top\G_t$ as follows:
    \begin{align*}
        \U^\top\G_t&=-\U^\top(\I_m-\X_t\X_t^\top)\A\X_t\\
        &=-\U^\top(\I_m-\X_t\X_t^\top)\U\bfSigma\U^\top\X_t\\
        &=-(\I_{r_A}-\bfPhi_t\bfPhi_t^\top)\bfSigma\bfPhi_t,\\
        \G_t\G_t^\top&=\X_t^\top\A(\I_m-\X_t\X_t^\top)^2\A\X_t\\
        &=\X_t^\top\A(\I_m-\X_t\X_t^\top)\A\X_t\\
        &=\bfPhi_t\bfSigma(\I_{r_A}-\bfPhi_t\bfPhi_t^\top)\bfSigma\bfPhi_t.
    \end{align*}
    
    Since $\sigma_1(\bfSigma),\sigma_1(\bfPhi_t)\leq1$, we obtain that $\beta_t:=\sigma_1(\I_{r_A}-\bfPhi_t\bfPhi_t^\top)\geq\sigma_1(\G_t^\top\G_t)$.
    
    From the update of $\X_t$, we have that
    \begin{align*}
        \X_{t+1}\X_{t+1}^\top=(\X_t-\eta\G_t)(\I_r+\eta^2\G_t^\top\G_t)^{-1}(\X_t-\eta\G_t)^\top.
    \end{align*}
    
    Multiplying both sides on the left by $\U^\top$ and on the right by $\U$, it follows that
    \begin{align}\label{form_PhiPhi_top}
        \bfPhi_{t+1}\bfPhi_{t+1}^\top&\stackrel{(a)}{=}(\I_{r_A}+\eta(\I_{r_A}-\bfPhi_t\bfPhi_t^\top)\bfSigma)\bfPhi_t(\I_r+\eta^2\G_t^\top\G_t)^{-1}\bfPhi_t^\top(\I_{r_A}+\eta\bfSigma(\I_{r_A}-\bfPhi_t\bfPhi_t^\top))^\top\nonumber\\
        &\stackrel{(b)}{\succeq}(1-\eta^2\G^\top_t\G_t)(\I_{r_A}+\eta(\I_{r_A}-\bfPhi_t\bfPhi_t^\top)\bfSigma)\bfPhi_t\bfPhi_t^\top(\I_{r_A}+\eta\bfSigma(\I_{r_A}-\bfPhi_t\bfPhi_t^\top))^\top\nonumber\\
        &\succeq(1-\eta^2\beta_t)(\I_{r_A}+\eta(\I_{r_A}-\bfPhi_t\bfPhi_t^\top)\bfSigma)\bfPhi_t\bfPhi_t^\top(\I_{r_A}+\eta\bfSigma(\I_{r_A}-\bfPhi_t\bfPhi_t^\top))^\top,
    \end{align}
    where $(a)$ is from the expression of $\U^\top\G_t$; and $(b)$ is by Lemma \ref{apdx.lemma.inverse}.

    Let the SVD of $\bfPhi_t$ be $\Q_t\bfLambda_t\PP_t^\top,\Q_t\in\mathbb{R}^{r_A\times r_A},\bfLambda_t\in\mathbb{R}^{r_A\times r_A},\PP_t\in\mathbb{R}^{r\times r_A};\bfS_t=\Q_t^\top\bfSigma\Q_t$.

    Substituting $\bfPhi_t=\Q_t\bfLambda_t\PP_t^\top$ into \eqref{form_PhiPhi_top}, we have that
    \begin{align*}
        \bfPhi_{t+1}\bfPhi_{t+1}^\top&\succeq(1-\eta^2\beta_t)\Q_t[\I_{r_A}+\eta(\I_{r_A}-\bfLambda_t^2)\bfS_t]\bfLambda_t^2[\I_{r_A}+\eta\bfS_t(\I_{r_A}-\bfLambda_t^2)]\Q_t^\top\\
        &\succeq(1-\eta^2\beta_t)\Q_t[\bfLambda_t^2+\eta(\I_{r_A}-\bfLambda_t^2)\bfS_t\bfLambda_t^2+\eta\bfLambda_t^2\bfS_t(\I_{r_A}-\bfLambda_t^2)]\Q_t^\top.
    \end{align*}
    
    Taking trace on both sides, it follows that
    \begin{align}\label{form_ineq_Phi}
        (1-\eta^2\beta_t)^{-1}\tr(\bfPhi_{t+1}\bfPhi_{t+1}^\top)&\geq\tr(\bfPhi_t\bfPhi_t^\top)+\eta\tr((\I_{r_A}-\bfLambda_t^2)\bfS_t\bfLambda_t^2+\bfLambda_t^2\bfS_t(\I_{r_A}-\bfLambda_t^2))\nonumber\\
        &\geq\tr(\bfPhi_t\bfPhi_t^\top)+2\eta\sigma_{r_A}(\bfS_t)\tr((\I_{r_A}-\bfLambda_t^2)\bfLambda_t^2)\nonumber\\
        &=\tr(\bfPhi_t\bfPhi_t^\top)+\frac{2\eta}{\kappa}\tr((\I_{r_A}-\bfLambda_t^2)\bfLambda_t^2)\nonumber\\
        &=\tr(\bfPhi_t\bfPhi_t^\top)+\frac{2\eta}{\kappa}\tr((\I_{r_A}-\bfPhi_t\bfPhi_t^\top)\bfPhi_t\bfPhi_t^\top).
    \end{align}

    By rewriting \eqref{form_ineq_Phi}, we arrive at 
    \begin{align}\label{form_core_ineq}
        \tr(\bfPhi_{t+1}\bfPhi_{t+1}^\top)-\tr(\bfPhi_t\bfPhi_t^\top)\geq\frac{2\eta}{\kappa}(1-\eta^2\beta_t)\tr((\I_{r_A}-\bfPhi_t\bfPhi_t^\top)\bfPhi_t\bfPhi_t^\top)-\eta^2\beta_t\tr(\bfPhi_t\bfPhi_t^\top).
    \end{align}

    Phase I: $\tr(\I_{r_A}-\bfPhi_t\bfPhi_t^\top)\geq0.5$.

    From \eqref{form_core_ineq}, we obtain that
    \begin{align*}
        \tr(\bfPhi_{t+1}\bfPhi_{t+1}^\top)-\tr(\bfPhi_t\bfPhi_t^\top)&\stackrel{(c)}{\geq}\frac{2\eta}{\kappa}(1-\eta^2\beta_t)\sigma_{r_A}^2(\bfPhi_t)\tr(\I_{r_A}-\bfPhi_t\bfPhi_t^\top)-\eta^2\beta_t\tr(\bfPhi_t\bfPhi_t^\top)\\
        &\stackrel{(d)}{\geq}\frac{\eta}{\kappa}(1-\eta^2)\sigma_{r_A}^2(\bfPhi_t)\tr(\I_{r_A}-\bfPhi_t\bfPhi_t^\top)-\eta^2r_A\\
        &\stackrel{(e)}{\geq}\frac{\eta(r-r_A)^2}{2c_1mr\kappa}(1-\eta^2)-\eta^2r_A\\
        &\stackrel{(f)}{\geq}\frac{(r-r_A)^4}{64c_1^2m^2r^2r_A\kappa^2},
    \end{align*}
    where $(c)$ is from Lemma \ref{lemma.apdx.trace-lower-bound}; $(d)$ is by $\beta_t=\sigma_1(\I_{r_A}-\bfPhi_t\bfPhi_t^\top)\leq1$ and $\tr(\bfPhi_t\bfPhi_t^\top)\leq r_A$; $(e)$ follows from Lemma \ref{init_with_high_probability} and Lemma \ref{lemma.singular_ascending}; and $(f)$ is by our choice of $\eta$.

    This inequality implies that at each step, $\tr(\bfPhi_t\bfPhi_t^\top)$ increases at least by $\Delta=\frac{(r-r_A)^4}{64c_1^2m^2r^2r_A\kappa^2}$. Consequently, after at most $(r_A-0.5)/\Delta\leq\frac{64c_1^2m^2r^2r_A^2\kappa^2}{(r-r_A)^4}$ iterations, DMD leaves Phase I.

    Phase II: $\tr(\I_{r_A}-\bfPhi_t\bfPhi_t^\top)<0.5$.

    This corresponds to a near-optimal regime. An immediate implication of this phase is that $\tr(\bfPhi_t\bfPhi_t^\top) \geq  r_A - 0.5$. Given that the singular values of $\bfPhi_t\bfPhi_t^\top$ lie in $[0,1]$, we have $\sigma_{r_A}^2(\bfPhi_t)\geq0.5$. Together with $\beta_t=\sigma_1(\I_{r_A}-\bfPhi_t\bfPhi_t^\top)\leq0.5$, we can simplify \eqref{form_core_ineq} as
    \begin{align}\label{form_phase_2}
        \tr(\bfPhi_{t+1}\bfPhi_{t+1}^\top)-\tr(\bfPhi_t\bfPhi_t^\top)&\geq\frac{2\eta}{\kappa}(1-\eta^2\beta_t)\sigma_{r_A}^2(\bfPhi_t)\tr(\I_{r_A}-\bfPhi_t\bfPhi_t^\top)-\eta^2\beta_t\tr(\bfPhi_t\bfPhi_t^\top)\nonumber\\
        &\geq\frac{\eta}{\kappa}(1-\eta^2\beta_t)\tr(\I_{r_A}-\bfPhi_t\bfPhi_t^\top)-\eta^2\beta_t\tr(\bfPhi_t\bfPhi_t^\top)\nonumber\\
        &\geq\frac{\eta}{\kappa}(1-\frac{\eta^2}{2})\tr(\I_{r_A}-\bfPhi_t\bfPhi_t^\top)-\eta^2\beta_t\tr(\bfPhi_t\bfPhi_t^\top)\nonumber\\
        &\geq\frac{\eta}{\kappa}(1-\frac{\eta^2}{2})\tr(\I_{r_A}-\bfPhi_t\bfPhi_t^\top)-\eta^2\beta_tr_A.
    \end{align}

    Since $\beta_t=\sigma_1(\I_{r_A}-\bfPhi_t\bfPhi_t^\top)\leq\tr(\I_{r_A}-\bfPhi_t\bfPhi_t^\top)$, we can further simplify \eqref{form_phase_2} as
    \begin{align*}
        \tr(\I_{r_A}-\bfPhi_{t+1}\bfPhi_{t+1}^\top)&\leq[1-\frac{\eta}{\kappa}(1-\frac{\eta^2}{2})+\eta^2r_A]\tr(\I_{r_A}-\bfPhi_t\bfPhi_t^\top)\\
        &=(1-\frac{\eta}{\kappa}+\eta^2r_A+\frac{\eta^3}{2\kappa})\tr(\I_{r_A}-\bfPhi_t\bfPhi_t^\top)\\
        &\stackrel{(g)}{\leq}(1-\frac{c_\eta(r-r_A)^2}{2mrr_A\kappa^2})\tr(\I_{r_A}-\bfPhi_t\bfPhi_t^\top),
    \end{align*}
    where $(g)$ is by our choice of $\eta$.

    Combining the analysis of Phase I, II and Lemma \ref{lemma.required-acc-sensing}, we have that after at most $$t_\varepsilon=\frac{64c_1^2m^2r^2r_A^2\kappa^2}{(r-r_A)^4}+\frac{\log(\varepsilon)}{\log(1-\frac{c_\eta(r-r_A)^2}{2mrr_A\kappa^2})}\leq\frac{64c_1^2m^2r^2r_A^2\kappa^2}{(r-r_A)^4}+\frac{8mrr_A\kappa^2}{c_\eta(r-r_A)^2}\log(\frac{1}{\varepsilon})$$ iterations, we have $f(\X_{t_\varepsilon},\bfTheta_{t_\varepsilon})\leq\varepsilon$.
    \end{proof}

\subsection{Proof of Lemma \ref{lemma_saddle_points}} \label{supplement.proofs.lemma_saddle_points}
\leavevmode\par
\begin{proof}
    Let $\U\bfSigma\U^\top$ be the compact SVD of $\A$, where $\U=[\uu_1,\uu_2,\ldots,\uu_{r_A}]$ and $\bfSigma=\diag(\lambda_1,\lambda_2,\ldots,\lambda_{r_A})$, with $\lambda_1\geq\lambda_2\geq\cdots\geq\lambda_{r_A}>0$. Here, $\diag(\lambda_1,\lambda_2,\ldots,\lambda_{r_A})$ denotes the diagonal matrix whose diagonal entries are $\lambda_1,\lambda_2,\ldots,\lambda_{r_A}$.

We first consider $\rho\geq1$. From the Eckart–Young–Mirsky theorem, we have that the best rank-$\rho$ approximation of $\A$ under the Frobenius norm is $\A_\rho=\U_1\bfSigma_1\U_1^\top$, where $\U_1=[\uu_1,\uu_2,\ldots,\uu_\rho]$ and $\bfSigma_1=\diag(\lambda_1,\lambda_2,\ldots,\lambda_\rho)$, without considering the ordering of the eigenvalues. 

We begin by analyzing the form of $\X$ and $\bfTheta$. Since $\rank(\A_\rho)=\rank(\U_1)=\rho$ and $\range(\A_\rho)\subseteq\range(\U_1)$, it follows that $\range(\A_\rho)=\range(\U_1)$. Together with $\range(\A_\rho)=\range(\X\bfTheta\X^\top)\subseteq\range(\X)$, we can obtain $\range(\U_1)\subseteq\range(\X)$. Therefore, there exists a matrix $\Q\in\mathbb{R}^{r\times \rho}$, such that $\U_1=\X\Q$. By the definition of $\U_1$, we derive that
\[
    \U_1^\top\U_1=\Q^\top\X^\top\X\Q=\Q^\top\Q=\I_\rho,
\]
which implies that $\Q$ is a column-orthonormal matrix.

We extend $\Q$ to an $r \times r$ orthogonal matrix $\tilde{\Q}=[\Q,\PP]$. Let $\V_1=\X\PP$, then $[\U_1,\V_1]=[\X\Q,\X\PP]=\X\tilde{\Q}$. Since $\tilde{\Q}^\top\X^\top\X\tilde{\Q}=\tilde{\Q}^\top\tilde{\Q}=\I_r$, then $[\U_1,\V_1]$ is also a column-orthonormal matrix, which means that
\[
    \V_1=[\bfv_1,\bfv_2,\ldots,\bfv_{r-\rho}],\,\text{with}\,\bfv_1,\bfv_2,\ldots,\bfv_{r-\rho}\in\U_1^\perp;\,\V_1^\top\V_1=\I_{r-\rho}.
\]

Let $\U_2=[\uu_{\rho+1},\uu_{\rho+2},\ldots,\uu_{r_A}]$, and then $\U=[\U_1,\U_2]$. By substituting $\U$ and $\X$, we obtain
\begin{align*}
    \X^\top\U&=\tilde{\Q}\begin{bmatrix}
        \U_1^\top\\
        \V_1^\top
    \end{bmatrix}[\U_1,\U_2]\\
    &\stackrel{(a)}{=}\tilde{\Q}\begin{bmatrix}
        \I_\rho & \bm{0}\\
        \bm{0} & \V_1^\top\U_2
    \end{bmatrix},
\end{align*}
where $(a)$ is from $\bfv_1,\bfv_2,\ldots,\bfv_{r-\rho}\in\U_1^\perp$.

From $\bfTheta=\X^\top\A\X$ and $\A=\U\bfSigma\U^\top$, we have that $\bfTheta=\X^\top\U\bfSigma\U^\top\X$.

Let $\C:=\V_1^\top\U_2\in\mathbb{R}^{(r-\rho)\times(r_A-\rho)}$ and substituting the expression of $\X^\top\U$ into $\bfTheta$, we obtain
\begin{align*}
    \bfTheta&=\tilde{\Q}
    \begin{bmatrix}
        \I_\rho & \bm{0}\\
        \bm{0} & \C
    \end{bmatrix}
    \begin{bmatrix}
        \bfSigma_1 & \bm{0}\\
        \bm{0} & \bfSigma_2
    \end{bmatrix}
    \begin{bmatrix}
        \I_\rho & \bm{0}\\
        \bm{0} & \C^\top
    \end{bmatrix}
    \tilde{\Q}^\top\\
    &=\tilde{\Q}\begin{bmatrix}
        \bfSigma_1 & \bm{0}\\
        \bm{0} & \C\bfSigma_2\C^\top
    \end{bmatrix}\tilde{\Q}^\top,
\end{align*}
where $\bfSigma_1=\diag(\lambda_1,\ldots,\lambda_{\rho})$ and $\bfSigma_2=\diag(\lambda_{\rho+1},\ldots,\lambda_{r_A})$.

Then, we have that
\begin{align*}
    \X\bfTheta\X^\top&=\X\tilde{\Q}\begin{bmatrix}
        \bfSigma_1 & \bm{0}\\
        \bm{0} & \C\bfSigma_2\C^\top
    \end{bmatrix}\tilde{\Q}^\top\X^\top\\
    &=[\U_1,\V_1]\begin{bmatrix}
        \bfSigma_1 & \bm{0}\\
        \bm{0} & \C\bfSigma_2\C^\top
    \end{bmatrix}[\U_1,\V_1]^\top\\
    &=\U_1\bfSigma_1\U_1^\top+\V_1\C\bfSigma_2\C^\top\V_1^\top.
\end{align*}
Together with $\X\bfTheta\X^\top=\A_\rho=\U_1\bfSigma_1\U_1^\top$, it follows that $\V_1\C\bfSigma_2\C^\top\V_1^\top=\bm{0}$. Multiplying $\V_1^\top$ on the left side and $\V_1$ on the right side, and applying $\V_1^\top\V_1=\I_{r-\rho}$, we have that
\begin{align*}
    \C\bfSigma_2\C^\top=\bm{0}.
\end{align*}
Since $\bfSigma_2=\diag(\lambda_{\rho+1},\ldots,\lambda_{r_A})\succ\bm{0}$, we obtain $\C=\V_1^\top\U_2=\bm{0}$. 

This implies that $\bfv_1,\bfv_2,\ldots,\bfv_{r-\rho}\in\U_2^\perp$. Moreover, since $\bfv_1,\bfv_2,\ldots,\bfv_{r-\rho}\in\U_1^\perp$ as well, we conclude that $\bfv_1,\bfv_2,\ldots,\bfv_{r-\rho}\in\U^\perp$.

Substituting $\X=[\U_1,\V_1]\tilde{\Q}^\top$ into $\X\bfTheta\X^\top=\A_\rho$, we can obtain
\begin{align*}
    [\U_1,\V_1]\tilde{\Q}^\top\bfTheta\tilde{\Q}[\U_1,\V_1]^\top&=\A_\rho\\
    &=\U_1\diag(\lambda_1,\lambda_2,\ldots,\lambda_\rho)\U_1^\top\\
    &=[\U_1,\V_1]\diag(\lambda_1,\lambda_2,\ldots,\lambda_\rho,0,\ldots,0)[\U_1,\V_1]^\top.
\end{align*}

Expanding both sides of the equation, together with $\bfv_1,\bfv_2,\ldots,\bfv_{r-\rho}\in\U^\perp$
, we can obtain
\begin{align*}
\tilde{\Q}^\top\bfTheta\tilde{\Q}=\diag(\lambda_1,\lambda_2,\ldots,\lambda_\rho,0,\ldots,0).
\end{align*}
This implies that
\[
    \bfTheta=\tilde{\Q}\diag(\lambda_1,\lambda_2,\ldots,\lambda_\rho,0,\ldots,0)\tilde{\Q}^\top.
\]

To proceed, we first verify that $(\tilde{\X},\tilde{\bfTheta}):=\left([\U_1,\V_1],\diag(\lambda_1,\lambda_2,\ldots,\lambda_\rho,0,\ldots,0)\right)$ is indeed a saddle point and then prove $(\X,\bfTheta)$ is also a saddle point.

The Riemannian gradient w.r.t. $\X$, denoted by $\G$ in \eqref{rgd_X}, evaluated at $\tilde{\X}$, can be written as
\begin{align*}
    \G&=(\I_m-\tilde{\X}\tilde{\X}^\top)(\tilde{\X}\tilde{\X}^\top-\A)\tilde{\X}+\frac{1}{2}\tilde{\X}(\tilde{\X}^\top(\tilde{\X}\tilde{\X}^\top-\A)-(\tilde{\X}\tilde{\X}^\top-\A)^\top\tilde{\X})\\
    &=-\A\tilde{\X}+\tilde{\X}\tilde{\X}^\top\A\tilde{\X}\\
    &=-\U_1\bfSigma_1+(\U_1\U_1^\top+\V_1\V_1^\top)\U_1\bfSigma_1\\
    &=\bm{0}.
\end{align*}

Consequently, $\tilde{\X}$ remains unchanged after the update.

Then, the Euclidean gradient w.r.t. $\bfTheta$ evaluated at $\tilde{\bfTheta}$ is
\begin{align*}
    \frac{1}{2}(\tilde{\bfTheta}-\tilde{\X}^\top\A\tilde{\X})=\bm{0}.
\end{align*}

Therefore, $(\tilde{\X},\tilde{\bfTheta})$ is a stationary point in the Riemannian sense.

We now show that $(\tilde{\X},\tilde{\bfTheta})$ is neither a local minimum nor a local maximum of the objective function.

For any $0<\nu<\lambda_{r_A}$, we will construct a pair $(\tilde{\X}_+,\tilde{\bfTheta}_+)$, such that $f(\tilde{\X}_+,\tilde{\bfTheta}_+)>f(\tilde{\X},\tilde{\bfTheta})$, 
$d\left((\tilde{\X}_+,\tilde{\bfTheta}_+),(\tilde{\X},\tilde{\bfTheta})\right):=\sqrt{\|\tilde{\X}_+-\tilde{\X}\|_\fro^2+\|\tilde{\bfTheta}_+-\tilde{\bfTheta}\|_\fro^2}\leq\nu$ and $\tilde{\X}_+^\top\tilde{\X}_+=\I_r$.

Let $\tilde{\X}_+=\tilde{\X}=[\U_1,\V_1]$ and $\tilde{\bfTheta}_+=\diag(\lambda_1-\nu,\lambda_2,\ldots,\lambda_\rho,0,\ldots,0)$. By construction, $\tilde{\X}_+^\top\tilde{\X}_+=\I_r$ and $d\left((\tilde{\X}_+,\tilde{\bfTheta}_+),(\tilde{\X},\tilde{\bfTheta})\right)=\sqrt{\nu^2}\leq\nu$ hold. The value of the objective function is
\begin{align*}
    f(\tilde{\X}_+,\tilde{\bfTheta}_+)&=\frac{1}{4}\|\tilde{\X}_+\tilde{\bfTheta}_+\tilde{\X}_+^\top-\A\|_\fro^2\\
    &=\frac{1}{4}\|(\lambda_1-\nu)\uu_1\uu_1^\top+\sum_{j=2}^{\rho}\lambda_j\uu_j\uu_j^\top-\sum_{j=1}^{r_A}\lambda_j\uu_j\uu_j^\top\|_\fro^2\\
    &=\frac{1}{4}\|\nu \uu_1\uu_1^\top+\sum_{j=\rho+1}^{r_A}\lambda_j\uu_j\uu_j^\top\|_\fro^2\\
    &\stackrel{(b)}{=}\frac{1}{4}(\nu^2\|\uu_1\uu_1^\top\|_\fro^2+\|\tilde{\X}\tilde{\bfTheta}\tilde{\X}^\top-\A\|_\fro^2)\\
    &>f(\tilde{\X},\tilde{\bfTheta}),
\end{align*}
where $(b)$ is by the orthogonality of $\{\uu_1,\uu_2,\ldots,\uu_{r_A}\}$.

We now try to construct a pair $(\tilde{\X}_-,\tilde{\bfTheta}_-)$, such that $f(\tilde{\X}_-,\tilde{\bfTheta}_-)<f(\tilde{\X},\tilde{\bfTheta})$,\\
$d\left((\tilde{\X}_-,\tilde{\bfTheta}_-),(\tilde{\X},\tilde{\bfTheta})\right):=\sqrt{\|\tilde{\X}_--\tilde{\X}\|_\fro^2+\|\tilde{\bfTheta}_--\tilde{\bfTheta}\|_\fro^2}\leq\nu$, and $\tilde{\X}_-^\top\tilde{\X}_-=\I_r$.

Since $\bfv_j\in\U^\perp$ for any $j\in\{1,2,\ldots,r-\rho\}$, it follows that $\bfv_j\in\Span\{\uu_{r_A+1},\ldots,\uu_m\}$. Accordingly, we consider
\begin{align*}
    \tilde{\X}_-&=[\U_1,k\bfv_1+s\uu_{\rho+1},\bfv_2,\ldots,\bfv_{r-\rho}],\\
    \tilde{\bfTheta}_-&=\diag(\lambda_1,\lambda_2,\ldots,\lambda_\rho,\nu_0,0,\ldots,0),
\end{align*}
where $k,s,\nu_0>0,k^2+s^2=1$ and $k,s,\nu_0$ will be given later.
We can easily verify that $\tilde{\X}_-^\top\tilde{\X}_-=\I_r$ holds. The distance is
\begin{align*}
    d\left((\tilde{\X}_-,\tilde{\bfTheta}_-),(\tilde{\X},\tilde{\bfTheta})\right)&=\sqrt{\|\tilde{\X}_--\tilde{\X}\|_\fro^2+\|\tilde{\bfTheta}_--\tilde{\bfTheta}\|_\fro^2}\\
    &=\sqrt{\|(k-1)\bfv_1+s\uu_{\rho+1}\|^2+\nu_0^2}\\
    &=\sqrt{(k-1)^2+s^2+\nu_0^2}\\
    &=\sqrt{2-2k+\nu_0^2}.
\end{align*}

Let $k=1-\frac{\nu^2}{4}$, $s=\sqrt{1-k^2}$ and $\nu_0\leq\frac{\nu}{2}$, then $d\left((\tilde{\X}_-,\tilde{\bfTheta}_-),(\tilde{\X},\tilde{\bfTheta})\right)\leq\sqrt{\frac{\nu^2}{2}+\frac{\nu^2}{4}}\leq\nu$.
The value of the objective function is
\begin{align*}
    &f(\tilde{\X}_-,\tilde{\bfTheta}_-)\\
    &=\frac{1}{4}\|\tilde{\X}_-\tilde{\bfTheta}_-\tilde{\X}^\top-\A\|_\fro^2\\
    &=\frac{1}{4}\|\nu_0(k^2\bfv_1\bfv_1^\top+ks\uu_{\rho+1}\bfv_1^\top+ks\bfv_1\uu_{\rho+1}^\top)+(\nu_0s^2-\lambda_\rho)\uu_{\rho+1}\uu_{\rho+1}^\top-\sum_{j=\rho+2}^{r_A}\lambda_j\uu_j\uu_j^\top\|_\fro^2\\
    &\stackrel{(c)}{=}\frac{1}{4}\left(\nu_0^2(k^4+k^2s^2\|\uu_{\rho+1}\bfv_1^\top\|_\fro^2+k^2s^2\|\bfv_1\uu_{\rho+1}^\top\|_\fro^2)+(\nu_0s^2-\lambda_\rho)^2\right)+f(\tilde{\X},\tilde{\bfTheta})-\frac{1}{4}\lambda_\rho^2\\
    &=\frac{1}{4}\nu_0^2\left(k^4+2k^2s^2+s^4\right)-\frac{1}{2}\nu_0\lambda_\rho s^2+f(\tilde{\X},\tilde{\bfTheta}),
\end{align*}
where $(c)$ is from the orthogonality of $\{\uu_1,\uu_2,\ldots,\uu_{r_A},\bfv_1\}$.
Let $\nu_0>0$ be sufficiently small. Then $\frac{1}{4}\nu_0^2\left(k^4+2k^2s^2+s^4\right)-\frac{1}{2}\nu_0\lambda_\rho s^2<0$. This ensures that the perturbed pair leads to a strictly smaller objective value, i.e., $f(\tilde{\X}_-,\tilde{\bfTheta}_-)<f(\tilde{\X},\tilde{\bfTheta})$.

Therefore, we have verified that $(\tilde{\X},\tilde{\bfTheta})$ is a saddle point. Building upon this result, we now proceed to show that $(\X,\bfTheta)=(\tilde{\X}\tilde{\Q}^\top, \tilde{\Q}\tilde{\bfTheta}\tilde{\Q}^\top)$ is also a saddle point.

Plugging in the expression of $(\X,\bfTheta)$, we obtain the Riemannian gradient w.r.t. $\X$ as follows:
\begin{align*}
     \G&=-\A\X+\X\X^\top\A\X\\
       &=-\A\tilde{\X}\tilde{\Q}^\top+\tilde{\X}\tilde{\X}^\top\A\tilde{\X}\tilde{\Q}^\top\\
       &=(-\A\tilde{\X}+\tilde{\X}\tilde{\X}^\top\A\tilde{\X})\tilde{\Q}^\top\\
       &=\bm{0}.
\end{align*}
And the Euclidean gradient w.r.t. $\bfTheta$ is 
\begin{align*}
    \frac{1}{2}(\bfTheta-\X^\top\A\X)&=\frac{1}{2}(\tilde{\Q}\tilde{\bfTheta}\tilde{\Q}^\top-\tilde{\Q}\tilde{\X}^\top\A\tilde{\X}\tilde{\Q}^\top)\\
    &=\frac{1}{2}\tilde{\Q}(\tilde{\bfTheta}-\tilde{\X}^\top\A\tilde{\X})\tilde{\Q}^\top\\
    &=\bm{0}.
\end{align*}

Therefore, $(\X,\bfTheta)$ is a stationary point in the Riemannian sense.

Let $(\X_+,\bfTheta_+)=(\tilde{\X}_+\tilde{\Q}^\top,\tilde{\Q}\bfTheta_+\tilde{\Q}^\top)$, $(\X_-,\bfTheta_-)=(\tilde{\X}_-\tilde{\Q}^\top,\tilde{\Q}\bfTheta_-\tilde{\Q}^\top)$. The distance is
\begin{align*}
    d\left((\X_+,\bfTheta_+),(\X,\bfTheta)\right)&=\sqrt{\|\X_+-\X\|_\fro^2+\|\bfTheta_+-\bfTheta\|_\fro^2}\\
    &=\sqrt{\|(\tilde{\X}_+-\tilde{\X})\tilde{\Q}^\top\|_\fro^2+\|\tilde{\Q}(\tilde{\bfTheta}_+-\tilde{\bfTheta})\tilde{\Q}^\top\|_\fro^2}\\
    &=\sqrt{\|\tilde{\X}_+-\tilde{\X}\|_\fro^2+\|\tilde{\bfTheta}_+-\tilde{\bfTheta}\|_\fro^2}\\
    &=d\left((\tilde{\X}_+,\tilde{\bfTheta}_+),(\tilde{\X},\tilde{\bfTheta})\right).
\end{align*}
In the same manner, we can obtain that $d\left((\X_-,\bfTheta_-),(\X,\bfTheta)\right)=d\left((\tilde{\X}_-,\tilde{\bfTheta}_-),(\tilde{\X},\tilde{\bfTheta})\right)$.
By the orthogonality of $\tilde{\Q}$, the following three identities hold:
\begin{align*}
    \X\bfTheta\X^\top&=\tilde{\X}\tilde{\Q}^\top\tilde{\Q}\tilde{\bfTheta}\tilde{\Q}^\top\tilde{\Q}\tilde{\X}^\top=\tilde{\X}\tilde{\bfTheta}\tilde{\X}^\top,\\
    \X_+\bfTheta_+\X_+^\top&=\tilde{\X}_+\tilde{\Q}^\top\tilde{\Q}\tilde{\bfTheta}_+\tilde{\Q}^\top\tilde{\Q}\tilde{\X}_+^\top=\tilde{\X}_+\tilde{\bfTheta}_+\tilde{\X}_+^\top,\\
    \X_-\bfTheta_-\X_-^\top&=\tilde{\X}_-\tilde{\Q}^\top\tilde{\Q}\tilde{\bfTheta}_-\tilde{\Q}^\top\tilde{\Q}\tilde{\X}_-^\top=\tilde{\X}_-\tilde{\bfTheta}_-\tilde{\X}_-^\top.
\end{align*}

Then, we have $f(\X,\bfTheta)=f(\tilde{\X},\tilde{\bfTheta})$, $f(\X_+,\bfTheta_+)=f(\tilde{\X}_+,\tilde{\bfTheta}_+)$, $f(\X_-,\bfTheta_-)=f(\tilde{\X}_-,\tilde{\bfTheta}_-)$. Thus, we obtain the strict inequality $f(\X_-,\bfTheta_-)<f(\X,\bfTheta)<f(\X_+,\bfTheta_+)$.
Therefore, $(\X,\bfTheta)$ is also a saddle point.

We now turn to the case $\rho=0$, i.e., $\X\bfTheta\X^\top=\A_0=\bm{0}$. Consequently, $\bfTheta=\X^\top\A_0\X=\X^\top\A\X=\bm{0}$. Let $\A^{1/2}:=\U\bfSigma^{1/2}\U^\top$, where $\bfSigma^{1/2}=\diag(\lambda_1^{1/2},\ldots,\lambda_{r_A}^{1/2})$. Then, we have that 
\begin{align*}
    0&=\tr(\X^\top\A\X)\\
    &=\tr(\X^\top(\A^{1/2})^\top\A^{1/2}\X)\\
    &=\|\A^{1/2}\X\|_\fro^2.
\end{align*}
Thus, it is guaranteed that $\A^{1/2}\X=\bm{0}$, which implies that $\A\X=\A^{1/2}\A^{1/2}\X=\bm{0}$. 

Multiplying both sides with $\U^\top$, we arrive at $\bfSigma\U^\top\X=\bm{0}$. Together with $\bfSigma\succ\bm{0}$, we have that $\U^\top\X=\bm{0}$. Let $\X$ be expressed as $\X = [\mathbf{x}_1, \mathbf{x}_2, \ldots, \mathbf{x}_r]$, where each $\mathbf{x}_j$ is a column vector. Hence, each $\mathbf{x}_j$ lies in $\U^\perp$ for $j\in\{1,2,\ldots,r\}$.

Therefore, the Riemannian gradient w.r.t. $\X$ is that
\begin{align*}
    \G&=-\A\X+\X\X^\top\A\X\\
    &=(-\U\bfSigma+\X\X^\top\U\bfSigma)\U^\top\X\\
    &=\bm{0}.
\end{align*}
And the Euclidean gradient w.r.t. $\bfTheta$ admits that
\begin{align*}
    \frac{1}{2}(\bfTheta-\X^\top\A\X)&=\bm{0}.
\end{align*}

Therefore, $(\X,\bfTheta)$ is a stationary point in the Riemannian sense.

For any $0<\nu<\lambda_{r_A}$, we construct the pair $(\X_+,\bfTheta_+)$ as follows:
\begin{align*}
    \X_+&=[k\mathbf{x}_1+s\uu_1,\mathbf{x}_2,\ldots,\mathbf{x}_r],\\
    \bfTheta_+&=\diag(-\nu_1,0,\ldots,0),
\end{align*}
where $k=1-\frac{\nu^2}{4},s=\sqrt{1-k^2}$, and $0<\nu_1\leq\frac{\nu}{2}$.
We can easily verify that $\X_+^\top\X_+=\I_r$ and the distance is
\begin{align*}
    d\left((\X_+,\bfTheta_+),(\X,\bfTheta)\right)&=\sqrt{\|\X_+-\X\|_\fro^2+\|\bfTheta_+-\bfTheta\|_\fro^2}\\
    &=\sqrt{\|(k-1)\mathbf{x}_1+s\uu_1\|^2+\nu_1^2}\\\
    &=\sqrt{(k-1)^2+s^2+\nu_1^2}\\
    &\leq\sqrt{\frac{\nu^2}{2}+\frac{\nu^2}{4}}\\
    &\leq\nu.
\end{align*}

The value of the objective function is 
\begin{align*}
    f(\X_+,\bfTheta_+)&=\frac{1}{4}\|\X_+\bfTheta_+\X_+^\top-\A\|_\fro^2\\
    &=\frac{1}{4}\|-\nu_1\left(k^2\mathbf{x}_1\mathbf{x}_1^\top+s^2\uu_1\uu_1^\top\right)-\sum_{j=1}^{r_A}\lambda_j \uu_j\uu_j^\top\|_\fro^2\\
    &=\frac{1}{4}\|\nu_1k^2\mathbf{x}_1\mathbf{x}_1^\top+\nu_1s^2\uu_1\uu_1^\top+\sum_{j=1}^{r_A}\lambda_j \uu_j\uu_j^\top\|_\fro^2\\
    &\stackrel{(d)}{=}\frac{1}{4}\left(\nu_1^2k^4+\nu_1^2s^4+2\nu_1\lambda_1s^2+\|\X\bfTheta\X^\top-\A\|_\fro^2\right)\\
    &>f(\X,\bfTheta),
\end{align*}
where $(d)$ is due to the orthogonality of $\{\uu_1,\uu_2,\ldots,\uu_{r_A},\mathbf{x}_1\}$.
Now consider the pair $(\X_-,\bfTheta_-)$ defined as:
\begin{align*}
    \X_-&=[k\mathbf{x}_1+s\uu_1,\mathbf{x}_2,\ldots,\mathbf{x}_r],\\
    \bfTheta_-&=\diag(\nu_2,0,\ldots,0),
\end{align*}
where $k=1-\frac{\nu^2}{4},s=\sqrt{1-k^2}$, and $0<\nu_2\leq\frac{\nu}{2}$.
It can be verified that $\X_-^\top\X_-=\I_r$, and the distance is
\begin{align*}
    d\left((\X_-,\bfTheta_-),(\X,\bfTheta)\right)&=\sqrt{\|\X_--\X\|_\fro^2+\|\bfTheta_--\bfTheta\|_\fro^2}\\
    &=\sqrt{\|(k-1)\mathbf{x}_1+s\uu_1\|^2+\nu_2^2}\\\
    &=\sqrt{(k-1)^2+s^2+\nu_2^2}\\
    &\leq\sqrt{\frac{\nu^2}{2}+\frac{\nu^2}{4}}\\
    &\leq\nu.
\end{align*}

The value of the objective function is 
\begin{align*}
    f(\X_-,\bfTheta_-)&=\frac{1}{4}\|\X_-\bfTheta_-\X_-^\top-\A\|_\fro^2\\
    &=\frac{1}{4}\|\nu_2\left(k^2\mathbf{x}_1\mathbf{x}_1^\top+s^2\uu_1\uu_1^\top\right)-\sum_{j=1}^{r_A}\lambda_j \uu_j\uu_j^\top\|_\fro^2\\
    &=\frac{1}{4}\|\nu_2k^2\mathbf{x}_1\mathbf{x}_1^\top+\nu_2s^2\uu_1\uu_1^\top-\sum_{j=1}^{r_A}\lambda_j \uu_j\uu_j^\top\|_\fro^2\\
    &\stackrel{(e)}{=}\frac{1}{4}\left(\nu_2^2k^4+\nu_2^2s^4-2\nu_2\lambda_1s^2+\|\X\bfTheta\X^\top-\A\|_\fro^2\right)\\
    &=\frac{1}{4}\left(\nu_2^2(k^4+s^4)-2\nu_2\lambda_1s^2\right)+f(\X,\bfTheta),
\end{align*}
where $(e)$ is by the orthogonality of $\{\uu_1,\uu_2,\ldots,\uu_{r_A},\mathbf{x}_1\}$.
Let $\nu_2>0$ be sufficiently small. Then $\frac{1}{4}\left(\nu_2^2(k^4+s^4)-2\nu_2\lambda_1s^2\right)<0$. This guarantees that $f(\X_-,\bfTheta_-)<f(\X,\bfTheta)$.
Therefore, $(\X,\bfTheta)$ is also a saddle point when $\rho=0$.
\end{proof}

\subsection{Proof of Theorem \ref{theorem_saddle_to_saddle}}\label{supplement.proofs.theorem_saddle_to_saddle}
\leavevmode\par
\begin{proof}
    For simplicity, we assume that $\frac{\sigma_i(\A)}{\sigma_{i+1}(\A)}\geq\beta\geq10,i=1,2,\ldots,r_A-1$.
    
    Let $\phi_{i,j}(t):=\uu_i^\top\X(t)\X(t)^\top\uu_j$, $t\geq0$. By chain rule, we have that
    \begin{align}\label{form.dynamic.phi}
        \dot{\phi}_{i,i}(t)&=\uu_i^\top\dot{\X}(t)\X(t)^\top\uu_i+\uu_i^\top\X(t)\dot{\X}(t)^\top\uu_i\nonumber\\
        &=\uu_i^\top(\I_m-\X(t)\X(t)^\top)\A\X(t)\X(t)^\top\uu_i+\uu_i^\top\X(t)\X(t)^\top\A(\I_m-\X(t)\X(t)^\top)\uu_i\nonumber\\
        &\stackrel{(a)}{=}2\sigma_i\uu_i^\top\X(t)\X(t)^\top\uu_i-2\sum_{j=1}^{r_A}\sigma_j(\uu_i^\top\X(t)\X(t)^\top\uu_j)^2\nonumber\\
        &=2\sigma_i\phi_{i,i}(t)-2\sigma_i\phi_{i,i}^2(t)-2\sum_{j=1,j\neq i}^{r_A}\sigma_j\phi_{i,j}^2(t),
    \end{align}
    where $(a)$ is from the fact that $\uu_i^\top\A=\sigma_i\uu_i^\top$ and $\A=\sum_{j=1}^{r_A}\sigma_j\uu_j\uu_j^\top$.

    We first consider the dynamics of $\phi_{1,1}(t)$.

    From \eqref{form.dynamic.phi}, we have that
    \begin{align*}
        \dot{\phi}_{1,1}(t)&=2\sigma_1\phi_{1,1}(t)-2\sigma_1\phi_{1,1}^2(t)-2\sum_{j=2}^{r_A}\sigma_j\phi_{1,j}^2(t)\\
        &\geq2\sigma_1\phi_{1,1}(t)-2\sigma_1\phi_{1,1}^2(t)-2\sigma_2\sum_{j=2}^{r_A}\phi_{1,j}^2(t)\\
        &\stackrel{(b)}{\geq}2(\sigma_1-\sigma_2)\phi_{1,1}(t)-2(\sigma_1-\sigma_2)\phi_{1,1}^2(t),
    \end{align*}
    where $(b)$ is by Lemma \ref{apdx.lemma.control_second_order}.

    Let $y_1(t)\in\mathbb{R}$ and $\dot{y}_1(t)=2(\sigma_1-\sigma_2)y_1(t)-2(\sigma_1-\sigma_2)y_1^2(t)$, $y_1(0)=\phi_{1,1}(0)$.

    As stated in \cite{richards1959flexible}, the trajectory $y_1(t)$ admits the closed-form solution $$y_1(t)=\frac{1}{1+(-1+\frac{1}{y_1(0)})e^{-2(\sigma_1-\sigma_2)t}},\quad t\geq0.$$

    Let $\bar{T}_1\in\{t\geq0\,|\,y_1(t)=1-(\frac{r}{m})^3\}$ (Note that this is a singleton because $y_1(t)$ is strictly increasing with respect to $t$).
    
    We can obtain that
    \begin{align*}
        \bar{T}_1&=\frac{1}{2(\sigma_1-\sigma_2)}\log(\frac{(1-\xi^3)(-1+\frac{1}{y_1(0)})}{\xi^3})\\
        &\leq\frac{1}{2(\sigma_1-\sigma_2)}\log(\frac{1}{\xi^3y_1(0)})\\
        &\stackrel{(c)}{\leq}\frac{5}{2(1-\frac{1}{\beta})\sigma_1}\log(\frac{1}{\xi})\\
        &\leq\frac{4}{\sigma_1}\log(\frac{1}{\xi}),
    \end{align*}
    where $\xi:=\frac{r}{m}$ and $(c)$ is from $\sigma_1\geq\beta\sigma_2$ and Lemma \ref{apdx.lemma.init_r_m}, which shows that $y_1(0)\geq\frac{1}{2}\xi$ w.h.p..

    By the comparison principle \cite{hartman2002ordinary}, it follows that $\phi_{1,1}(t)\geq y_1(t)$ for all $t\geq0$.
    
    Thus, there exists $T_1\leq\bar{T}_1\leq\frac{4}{\sigma_1}\log(\frac{m}{r})$, s.t., $\phi_{1,1}(T_1)\geq1-\xi^3$.
    
    Again from $\dot{\phi}_{1,1}(t)\geq2(\sigma_1-\sigma_2)(\phi_{1,1}(t)-\phi_{1,1}^2(t))\geq0$, we have that $\phi_{1,1}(t)\geq\phi_{1,1}(T_1)\geq1-\xi^3$ for all $t\geq T_1$.

    We now prove the existence of $T_k,k\geq1$, such that $\phi_{j,j}(t)\geq1-\xi^3,j=1,2,\ldots,k$ for all $t\geq T_{k}$, by mathematical induction. Assume that $\phi_{j,j}(t)\geq 1-\xi^3,j=1,2,\ldots,k-1$ for all $t\geq T_{k-1}$.

    From \eqref{form.dynamic.phi}, we have that
    \begin{align*}
        \dot{\phi}_{k,k}(t)&=2\sigma_{k}\phi_{k,k}(t)-2\sigma_{k}\phi_{k,k}^2(t)-2\sum_{j=1,j\neq k}^{r_A}\sigma_j\phi_{k,j}^2(t)\\
        &\geq2\sigma_{k}\phi_{k,k}(t)-2\sigma_{k}\phi_{k,k}^2(t)-2\sum_{j=1}^{k-1}\sigma_j\phi_{k,j}^2(t)-2\sigma_{k+1}\sum_{j=k+1}^{r_A}\phi_{k,j}^2(t)\\
        &\stackrel{(d)}{\geq}2(\sigma_{k}-\sigma_{k+1})\phi_{k,k}(t)-2(\sigma_{k}-\sigma_{k+1})\phi_{k,k}^2(t)-2\sum_{j=1}^{k-1}\sigma_j\phi_{k,j}^2(t),
    \end{align*}
    where $(d)$ is by Lemma \ref{apdx.lemma.control_second_order}.

    From Lemma \ref{apdx.lemma.crossed_terms}, we have the following upper bounds
    \begin{align*}
        \phi_{k,j}^2(t)\leq(1-\phi_{k,k}(t))(1-\phi_{j,j}(t))\leq 1-\phi_{j,j}(t)\leq\xi^3,\quad j=1,2\ldots,k-1.
    \end{align*}

    Thus, $\dot{\phi}_{k,k}(t)\geq2(\sigma_{k}-\sigma_{k+1})\phi_{k,k}(t)-2(\sigma_{k}-\sigma_{k+1})\phi_{k,k}^2(t)-2\xi^3\sum_{j=1}^{k-1}\sigma_j$. Together with $\frac{\sigma_i}{\sigma_{i+1}}\geq\beta\geq10$, we obtain that 
    $$\dot{\phi}_{k,k}(t)\geq2(\sigma_{k}-\sigma_{k+1})\phi_{k,k}(t)-2(\sigma_{k}-\sigma_{k+1})\phi_{k,k}^2(t)-3\xi^3.$$

    From Lemma \ref{apdx.lemma.singular_ascending_cont} and Lemma \ref{apdx.lemma.diag_lower_bound}, we have that $\phi_{k,k}(t)\geq\sigma_{r_A}(\U^\top\X(0)\X(0)^\top\U)$. Together with Lemma \ref{init_with_high_probability}, it follows that $\phi_{k,k}(t)\geq\frac{(r-r_A)^2}{c_1mr}\stackrel{(e)}{\geq}\frac{\xi}{c_4}$ w.h.p.. Here, $(e)$ is by $r\geq c_rr_A$ and $c_4=\frac{c_1}{(1-\frac{1}{c_r})^2}>1$.

    Then, $\dot{\phi}_{k,k}(t)\geq2(\sigma_k-\sigma_{k+1}-3c_4\xi^2)\phi_{k,k}(t)-2(\sigma_k-\sigma_{k+1})\phi_{k,k}^2(t)$.

    Let $y_{k}(t)\in\mathbb{R}$ and $\dot{y}_k(t)=2(\sigma_k-\sigma_{k+1}-3c_4\xi^2)y_k(t)-2(\sigma_k-\sigma_{k+1})y^2_k(t)$, $y_k(0)=\phi_{k,k}(0)$.

    As stated in \cite{richards1959flexible}, the trajectory $y_k(t)$ admits the closed-form 
    \begin{align*}
        y_k(t)=\frac{1-\frac{3c_4\xi^2}{\sigma_k-\sigma_{k+1}}}{1+(-1+\frac{1-\frac{3c_4\xi^2}{\sigma_k-\sigma_{k+1}}}{y_k(0)})e^{-2(\sigma_k-\sigma_{k+1}-3c_4\xi^2)t}},\quad t\geq0.
    \end{align*}

    Let $\tilde{T}_k\in\{t\geq0\,|\,y_k(t)=1-\frac{6c_4\xi^2}{\sigma_k-\sigma_{k+1}}\}$ (Note that this is a singleton because $y_k(t)$ is strictly increasing with respect to $t$).
    
    We can obtain that
    \begin{align*}
        \tilde{T}_k=\frac{1}{2(\sigma_k-\sigma_{k+1}-3c_4\xi^2)}\log(\frac{1-\frac{3c_4\xi^2}{\sigma_k-\sigma_{k+1}}-y_k(0)}{y_k(0)}\cdot\frac{\sigma_{k}-\sigma_{k+1}-6c_4\xi^2}{3c_4\xi^2}).
    \end{align*}

    By the comparison principle \cite{hartman2002ordinary}, it is guaranteed that $\phi_{k,k}(t)\geq y_k(t)\geq y_k(\tilde{T}_k)=1-\frac{6c_4\xi^2}{\sigma_k-\sigma_{k+1}}$ for all $t\geq\tilde{T}_k$.

    When $t\geq\tilde{T}_k$, we have that $1-\phi_{k,k}(t)\leq\frac{6c_4\xi^2}{\sigma_k-\sigma_{k+1}}$. Together with $\sum_{j=1}^{k-1}\sigma_{j}\phi_{k,j}^2(t)\leq\sum_{j=1}^{k-1}\sigma_{j}(1-\phi_{k,k}(t))(1-\phi_{j,j}(t))$, it follows that
    \begin{align*}
        \dot{\phi}_{k,k}(t)&\geq2(\sigma_{k}-\sigma_{k+1})\phi_{k,k}(t)-2(\sigma_{k}-\sigma_{k+1})\phi_{k,k}^2(t)-2\sum_{j=1}^{k-1}\sigma_j(1-\phi_{k,k}(t))\xi^3\\
        &\geq2(\sigma_{k}-\sigma_{k+1})\phi_{k,k}(t)-2(\sigma_{k}-\sigma_{k+1})\phi_{k,k}^2(t)-3\xi^3(1-\phi_{k,k}(t))\\
        &=2(\sigma_{k}-\sigma_{k+1}-\frac{3\xi^3(1-\phi_{k,k}(t))}{2\phi_{k,k}(t)})\phi_{k,k}(t)-2(\sigma_{k}-\sigma_{k+1})\phi_{k,k}^2(t) \\
        &=2(\sigma_{k}-\sigma_{k+1}-\frac{\xi^3}{2}(\sigma_k-\sigma_{k+1})\cdot\frac{3(1-\phi_{k,k}(t))}{(\sigma_k-\sigma_{k+1})\phi_{k,k}(t)})\phi_{k,k}(t)-2(\sigma_{k}-\sigma_{k+1})\phi_{k,k}^2(t) \\
        &\stackrel{(f)}{\geq}2(\sigma_{k}-\sigma_{k+1}-\frac{\xi^3}{2}(\sigma_{k}-\sigma_{k+1}))\phi_{k,k}(t)-2(\sigma_{k}-\sigma_{k+1})\phi_{k,k}^2(t),
    \end{align*}
    where $(f)$ is from $\frac{3(1-\phi_{k,k}(t))}{(\sigma_k-\sigma_{k+1})\phi_{k,k}(t)}\leq\frac{36c_4\xi^2}{(\sigma_k-\sigma_{k+1})^2}\leq1$.

    Let $y(t)\in\mathbb{R}$ and $\dot{y}(t)=2(\sigma_{k}-\sigma_{k+1}-\frac{\xi^3}{2}(\sigma_{k}-\sigma_{k+1}))y(t)-2(\sigma_{k}-\sigma_{k+1})y^2(t)$, $y_k(\tilde{T}_k)=1-\frac{6c_4\xi^2}{\sigma_k-\sigma_{k+1}}$.

    As stated in \cite{richards1959flexible}, the trajectory $y_k(t)$ admits the closed-form 
    \begin{align*}
        y_k(t+\tilde{T}_k)=\frac{1-\frac{\xi^3}{2}}{1+(-1+\frac{1-\frac{\xi^3}{2}}{y_k(\tilde{T}_k)})e^{-2(\sigma_{k}-\sigma_{k+1}-\frac{\xi^3}{2}(\sigma_{k}-\sigma_{k+1}))t}},\quad t\geq0.
    \end{align*}

    Let $\hat{T}_k\in\{t\geq0\,|\,y_k(t+\tilde{T}_k)=1-\xi^3\}$ (Note that this is a singleton because $y_k(t)$ is strictly increasing with respect to $t$).

    We can obtain that
    \begin{align*}
        \hat{T}_k=\frac{1}{2(\sigma_k-\sigma_{k+1}-\frac{\xi^3}{2}(\sigma_k-\sigma_{k+1}))}\log(\frac{2(1-\xi^3)(1-\frac{\xi^3}{2}-y_k(\tilde{T}_k))}{\xi^3y_k(\tilde{T}_k)}).
    \end{align*}

    Thus, we can get the following upper bound of $\tilde{T}_k+\hat{T}_k$:
    \begin{align*}
        \tilde{T}_k+\hat{T}_k&\leq\frac{1}{2(\sigma_k-\sigma_{k+1}-3c_4\xi^2)}\log\bigg[\frac{1-\frac{3c_4\xi^2}{\sigma_k-\sigma_{k+1}}-y_k(0)}{y_k(0)}\cdot\frac{\sigma_{k}-\sigma_{k+1}-6c_4\xi^2}{3c_4\xi^2}\\
        &\hspace{12.5em}\cdot\frac{2(1-\xi^3)(1-\frac{\xi^3}{2}-y_k(\tilde{T}_k))}{\xi^3y_k(\tilde{T}_k)}\bigg]\\
        &\stackrel{(g)}{\leq}\frac{1}{2(\sigma_k-\sigma_{k+1}-3c_4\xi^2)}\log(\frac{5}{y_k(0)\xi^3})\\
        &\stackrel{(h)}{\leq}\frac{1}{2(\sigma_k-\sigma_{k+1}-3c_4\xi^2)}
        \log(\frac{5c_4}{\xi^4})\\
        &\leq\frac{1}{2(\sigma_k-\sigma_{k+1}-3c_4\xi^2)}
        \log(\frac{1}{\xi^5})\\
        &\stackrel{(i)}{\leq}\frac{5}{2(1-\frac{2}{\beta})\sigma_{k}}\log(\frac{1}{\xi})\\
        &\leq\frac{4}{\sigma_k}\log(\frac{1}{\xi}),
    \end{align*}
    where $(g)$ is by $\xi=\frac{r}{m}\leq\frac{1}{c_\delta^2\kappa}$; $(h)$ is from $y_k(0)=\phi_{k,k}(0)\geq\frac{\xi}{c_4}$ w.h.p.; and $(i)$ follows from our assumptions that $\sigma_{k}\geq\beta\sigma_{k+1}$ and $\xi\leq\frac{1}{c_\delta^2\kappa}$.

    Then, $T_k\leq T_{k-1}+\tilde{T}_k+\hat{T}_k\leq T_{k-1}+\frac{4}{\sigma_{k}}\log(\frac{1}{\xi})$, and by the comparison principle \cite{hartman2002ordinary}, it is guaranteed that $\phi_{k,k}(t)\geq1-\xi^3$ for all $t\geq T_k$.

    Therefore, $T_k\leq\frac{4}{\sigma_1}\log(\frac{m}{r})
        +
        \sum_{j=2}^k \frac{4}{\sigma_{j}} \log(\frac{m}{r}),
        \quad k=1,2,\ldots,r_A-1.$

    We then prove that $\|\X(T_k)\bfTheta(T_k)\X(T_k)^\top - \A_k\|_\fro^2 \;\le\; 5\sigma_1^2\delta,\,k=1,2,\ldots,r_A-1$,        
    where $\A_k:=\argmin\limits_{\rank(\hat{\A})\leq k}\, \|\hat{\A}-\A\|_\fro^2$ is the best rank-$k$ approximation of $\A$ in Frobenius norm.

    From Lemma \ref{apdx.lemma.trace_control_norm} and $\beta\geq10$, it is sufficient to show that $\phi_{i,i}(T_k)\geq1-\delta,\,i=1,2,\ldots,k$ and $\phi_{i,i}(T_k)\leq\sqrt{\delta},\,i=k+1,\ldots,r_A$ for $k=1,2,\ldots,r_A-1$.

    As proved above, it is guaranteed that $\phi_{i,i}(T_k)\geq1-\delta^3\geq1-\delta,\,i=1,2,\ldots,k$.

    For $\phi_{i,i}(t), i\geq2$, \eqref{form.dynamic.phi} shows that $\dot{\phi}_{i,i}(t)\leq2\sigma_i\phi_{i,i}(t)-2\sigma_i\phi_{i,i}^2(t)$.

    Let $z_i(t)\in\mathbb{R}$ and $\dot{z}_i(t)=2\sigma_iz_i(t)-2\sigma_iz_i^2(t)$, $z_i(0)=\phi_{i,i}(0)$. 
    
    By the comparison principle \cite{hartman2002ordinary}, we have that $\phi_{i,i}(t)\leq z_i(t)$ for all $t\geq0$, which means it is sufficient to prove that $z_i(T_k)\leq \sqrt{\delta}$.

    As stated in \cite{richards1959flexible}, the trajectory $z_{i}(t)$ admits the closed-form solution
    \begin{align*}
        z_i(t)=\frac{1}{1+(-1+\frac{1}{z_i(0)})e^{-2\sigma_it}}\leq z_i(0)e^{-2\sigma_i t}.
    \end{align*}

    At $T_k$, for all $i\geq k+1$, we have that 
    \begin{align*}
        z_i(T_k)&\leq z_i(0)e^{-2\sigma_i T_k}\\
        &\stackrel{(j)}{\leq}\frac{3}{2}\xi \cdot\xi^{-2\sigma_i(\frac{5}{2(\beta-1)\sigma_2}+\sum_{j=2}^k\frac{5}{2(\beta-2)\sigma_{j+1}})}\\
        &\stackrel{(k)}{\leq}\frac{3}{2}\xi^{1-\frac{5}{\beta-2}\sum_{j=2}^{k+1}\beta^{-(i-j)}}\\
        &\leq\frac{3}{2}\xi^{1-\frac{5}{\beta-2}\sum_{j=-\infty}^{k+1}\beta^{-(i-j)}}\\
        &\leq\frac{3}{2}\xi^{1-\frac{5\beta}{(\beta-1)(\beta-2)}}\\
        &\stackrel{(l)}{\leq}\sqrt{\delta},
    \end{align*}
    where $(j)$ is by Lemma \ref{apdx.lemma.init_r_m}, which shows that $z_i(0)=\phi_{i,i}(0)\leq\frac{3}{2}\xi$ w.h.p.; $(k)$ is from our assumption that $\frac{\sigma_i}{\sigma_{i+1}}\geq\beta=10,\,i=1,2,\ldots,r_A-1$; and $(l)$ comes from $\xi\leq\delta^2$ and $\delta\leq\frac{1}{c_\delta\sqrt{\kappa}}$.

    Lastly, we prove that $\|\X(T_0)\bfTheta(T_0)\X(T_0)^\top-\A_0\|_\fro^2=\|\X(0)\bfTheta(0)\X(0)^\top\|_\fro^2\leq5\delta$.
    
    From Lemma \ref{apdx.lemma.init_r_m}, it follows that $\max\limits_{1\leq i \leq r_A}\phi_{i,i}(0)\leq\frac{3}{2}\xi$ holds w.h.p.. Together with $\frac{3}{2}\xi\leq\frac{3}{2}\delta^2\leq\sqrt{\delta}$ and Lemma \ref{apdx.lemma.trace_control_norm}, we have that $\|\X(T_0)\bfTheta(T_0)\X(T_0)^\top-\A_0\|_\fro^2\leq5\delta$.

    This completes the proof.
\end{proof}

\subsection{Proof of Lemma \ref{lemma_saddle_lower_bound}} \label{supplement.proofs.lemma_saddle_lower_bound}
\leavevmode\par
\begin{proof}
    Let $\U\bfSigma\U^\top$ be the compact SVD of $\A$, where $\U=[\uu_1,\uu_2,\ldots,\uu_{r_A}]\in\mathbb{R}^{m\times r_A},\bfSigma\in\mathbb{S}^{r_A}$. Let $\U_\perp\in\mathbb{R}^{m\times(m-r_A)}$ be the orthogonal complement of $\U$.
    
    From Lemma \ref{init_with_high_probability} and Lemma \ref{apdx.lemma.diag_lower_bound}, $\uu_k^\top\X(t)\X(t)^\top \uu_k\geq\frac{(r-r_A)^2}{c_1mr},k=1,2,\ldots,r_A$ holds for all $t\geq0$ w.h.p. over the initialization. By the update, we have $\bfTheta(t)=\X(t)^\top\A\X(t)$. 
    
    It therefore follows that
    \begin{align*}
        &\|\X(t)\bfTheta(t)\X(t)^\top-\A_\rho\|_\fro^2\\
        &~~~~~~~~~~=\|\X(t)\X(t)^\top\A\X(t)\X(t)^\top-\A_\rho\|_\fro^2\\
        &~~~~~~~~~~\stackrel{(a)}{=}\|\sum_{k=1}^{r_A}\sigma_k\X(t)\X(t)^\top\uu_k\uu_k^\top\X(t)\X(t)^\top-\sum_{k=1}^{\rho}\sigma_k\uu_k\uu_k^\top\|_\fro^2\\
        &~~~~~~~~~~\stackrel{(b)}{=}\|[\U,\U_\perp]^\top(\sum_{k=1}^{r_A}\sigma_k\X(t)\X(t)^\top\uu_k\uu_k^\top\X(t)\X(t)^\top-\sum_{k=1}^{\rho}\sigma_k\uu_k\uu_k^\top)[\U,\U_\perp]\|_\fro^2,
    \end{align*}
    where $(a)$ is from the Eckart-Young-Mirsky theorem, which gives $\A_\rho=\sum_{j=1}^\rho\sigma_j\uu_j\uu_j^\top$; $(b)$ is by the orthogonal invariance of Frobenius norm.

    From the definition of Frobenius norm, i.e., $\|\Y\|_\fro^2=\sum_{i,j=1}^m\Y_{i,j}^2\geq\sum_{j=i+1}^{r_A}\Y_{j,j}^2$ for any $\Y\in\mathbb{R}^{m\times m}$, we have that
    \begin{align*}
    &\|\X(t)\bfTheta(t)\X(t)^\top-\A_\rho\|_\fro^2\\
    &~~~~~~~~~~\geq\sum_{j=\rho+1}^{r_A}\Big[\uu_j^\top(\sum_{k=1}^{r_A}\sigma_k\X(t)\X(t)^\top\uu_k\uu_k^\top\X(t)\X(t)^\top-\sum_{k=1}^{\rho}\sigma_k\uu_k\uu_k^\top)\uu_j\Big]^2\\
        &~~~~~~~~~~=\sum_{j=\rho+1}^{r_A}\Big[\sum_{k=1}^{r_A}\sigma_k(\uu_j^\top\X(t)\X(t)^\top\uu_k)^2\Big]^2\\
        &~~~~~~~~~~\geq\sum_{j=\rho+1}^{r_A}\big[\sigma_j(\uu_j^\top\X(t)\X(t)^\top\uu_j)^2\big]^2\\
        &~~~~~~~~~~\geq\sum_{j=\rho+1}^{r_A}\sigma_j^2\cdot\frac{(r-r_A)^8}{c_1^4m^4r^4}.
    \end{align*}
    This completes the proof.
\end{proof}

\subsection{Proof of Theorem \ref{theorem_DMD_rank1}} \label{supplement.proofs.theorem_DMD_rank1}
\leavevmode\par
\begin{proof}
    For the proof, we take $\eta=\frac{1}{2\lambda_1}>0$.
    
    Let $\U=[\uu_1,\ldots,\uu_{r_A}]$ where $\uu_1,\ldots,\uu_{r_A}\in\mathbb{R}^m$. Let $\U_\perp=[\uu_{r_A+1},\ldots,\uu_m]\in\mathbb{R}^{m\times (m-r_A)}$ be the orthogonal complement of $\U$, where $\uu_{r_A+1},\ldots,\uu_m\in\mathbb{R}^m$.

    Since $\x_0\in\st(m,1)$ and $\x_{t+1}=\y_{t+1}/\|\y_{t+1}\|$, which guarantees that $\x_t\in\st(m,1)$ for all $t\geq0$, we can write $\x_t$ as $\x_t=\sum_{i=1}^m\alpha_i^{(t)}\uu_i$, with $\sum_{i=1}^m(\alpha_i^{(t)})^2=1$.

    From the update of $y_{t+1}$, we have that
    \begin{align*}
        y_{t+1}&=[1+\eta(\A-\x_t^\top\A\x_t)]\x_t\\
        &=\sum_{i=1}^m[1+\eta(\A-\x_t^\top\A\x_t)]\alpha_i^{(t)}\uu_i\\
        &=\sum_{i=1}^m[1+\eta(\lambda_i-\x_t^\top\A\x_t)]\alpha_i^{(t)}\uu_i,
    \end{align*}
    where $\lambda_{r_A+1}=\ldots\lambda_{m}=0$.

    Let $\beta_{i}^{(t)}=\alpha_i^{(t)}[1+\eta(\lambda_i-\x_t^\top\A\x_t)]$. By the update of $\x_{t+1}$, we obtain
    \begin{align}\label{x_expressed_by_beta}
        \x_{t+1}=\frac{\sum_{i=1}^m\beta_i^{(t)}\uu_i}{\sqrt{\sum_{i=1}^m(\beta_i^{(t)})^2}}.
    \end{align}

    Let $r_t:=\frac{\sum_{i=2}^m(\alpha_i^{(t)})^2}{(\alpha_1^{(t)})^2}\geq0$. Then, the relationship between $r_{t+1}$ and $r_t$ is as follows:
    \begin{align*}
        r_{t+1}&=\frac{\sum_{i=2}^m(\alpha_i^{(t+1)})^2}{(\alpha_1^{(t+1)})^2}\\
        &=\frac{\sum_{i=2}^m(\uu_i^\top\x_{t+1})^2}{(\uu_1^\top\x_{t+1})^2}\\
        &\stackrel{(a)}{=}\frac{\sum_{i=2}^m(\beta_i^{(t)})^2}{(\beta_1^{(t)})^2}\\
        &=\frac{\sum_{i=2}^m[\alpha_i^{(t)}(1+\eta(\lambda_i-\x_t^\top\A\x_t))]^2}{[\alpha_1^{(t)}(1+\eta(\lambda_1-\x_t^\top\A\x_t))]^2}\\
        &=r_t\cdot\frac{\sum_{i=2}^m[\alpha_i^{(t)}(1+\eta(\lambda_i-\x_t^\top\A\x_t))]^2}{(1+\eta(\lambda_1-\x_t^\top\A\x_t))^2\sum_{i=2}^m(\alpha_i^{(t)})^2},
    \end{align*}
    where $(a)$ is by substituting $x_{t+1}$ with the expression in \eqref{x_expressed_by_beta}.

    From the assumption, the eigenvalues of $\A$ are $\lambda_1>\lambda_2\geq\ldots\geq\lambda_{r_A}$. Together with $\x_t\in\st(m,1)$, it follows that $\x_t^\top\A\x_t\in[\lambda_{r_A},\lambda_1]$ for all $t\geq0$ and $\lambda_i-\x_t^\top\A\x_t\in[\lambda_i-\lambda_1,\lambda_i-\lambda_{r_A}]$.

    Now, we can further simplify the the relationship between $r_{t+1}$ and $r_t$ as follows:
    \begin{align*}
        r_{t+1}&=r_t\cdot\frac{\sum_{i=2}^m[\alpha_i^{(t)}(1+\eta(\lambda_i-\x_t^\top\A\x_t))]^2}{(1+\eta(\lambda_1-\x_t^\top\A\x_t))^2\sum_{i=2}^m(\alpha_i^{(t)})^2}\\
        &\stackrel{(b)}{\leq}r_t\cdot\frac{(1+\eta(\lambda_2-\x_t^\top\A\x_t))^2\sum_{i=2}^m(\alpha_i^{(t)})^2}{(1+\eta(\lambda_1-\x_t^\top\A\x_t))^2\sum_{i=2}^m(\alpha_i^{(t)})^2}\\
        &=r_t\Big(\frac{1+\eta(\lambda_2-\x_t^\top\A\x_t)}{1+\eta(\lambda_1-\x_t^\top\A\x_t)}\Big)^2\\
        &=r_t\Big(1-\frac{\eta(\lambda_1-\lambda_2)}{1+\eta(\lambda_1-\x_t^\top\A\x_t)}\Big)^2\\
        &\leq r_t\Big(1-\frac{\eta(\lambda_1-\lambda_2)}{1+2\eta\lambda_1}\Big)^2\\
        &\stackrel{(c)}{=}r_t\Big(1-\frac{\eta(\lambda_1-\lambda_2)}{2}\Big)^2,
    \end{align*}
    where $(b)$ follows from the fact that $\lambda_i-\x_t^\top\A\x_t\in[\lambda_i-\lambda_1,\lambda_i-\lambda_{r_A}]\subseteq[-2\lambda_1,2\lambda_1]$ and $\eta=\frac{1}{2\lambda_1}$, which guarantees that $1+\eta(\lambda_i-\x_t^\top\A\x_t)\geq0$; and $(c)$ is by $\eta=\frac{1}{2\lambda_1}$.
    
    From Lemma \ref{apdx.lemma.init_vec}, we have that $(\alpha_1^{(0)})^2=u_1^\top\x_0\x_0^\top u_1\geq\frac{1}{cm}$ w.h.p., where $c$ is a universal constant. Then, $r_0=\frac{\sum_{i=2}^m(\alpha_i^{(0)})^2}{(\alpha_1^{(0)})^2}=\frac{1-(\alpha_1^{(0)})^2}{(\alpha_1^{(0)})^2}\leq cm$ w.h.p..

    Thus, for any $t\geq0$, $r_{t}\leq r_0\big(1-\frac{\eta(\lambda_1-\lambda_2)}{2}\big)^{2t}\leq cm\big(1-\frac{\eta(\lambda_1-\lambda_2)}{2}\big)^{2t}$ w.h.p. over the initialization. 

    Since we have $\theta_t=\x_t^\top\A\x_t$ throughout the iterations, it is guaranteed that
    \begin{align*}
        \lambda_1-\theta_t&=\lambda_1-\x_t^\top\A\x_t\\
        &=\sum_{i=1}^m(\lambda_1-\lambda_i)(\alpha_i^{(t)})^2\\
        &=\sum_{i=2}^m(\lambda_1-\lambda_i)(\alpha_i^{(t)})^2.
    \end{align*}

    Then, $0\leq\lambda_1-\theta_t=\sum_{i=2}^m(\lambda_1-\lambda_i)(\alpha_i^{(t)})^2\leq2\lambda_1\sum_{i=2}^m(\alpha_i^{(t)})^2\leq2\lambda_1 r_t$.

    We now upper bound $\|\theta_t\x_t\x_t^\top-\A_1\|=\|\theta_t\x_t\x_t^\top-\lambda_1\uu_1\uu_1^\top\|$ by $r_t$:
    \begin{align*}
        \|\theta_t\x_t\x_t^\top-\lambda_1\uu_1\uu_1^\top\|&\leq\|(\theta_t-\lambda_1)\x_t\x_t^\top\|+\|\lambda_1(\x_t\x_t^\top-\uu_1\uu_1^\top)\|\\
        &=|\lambda_1-\theta_t|+\lambda_1\|\x_t\x_t^\top-\uu_1\uu_1^\top\|\\
        &\stackrel{(d)}{\leq}2\lambda_1 r_t+\lambda_1\sqrt{1-|\uu_1^\top\x_t|^2}\\
        &=2\lambda_1r_t+\lambda_1\sqrt{\frac{r_t}{1+r_t}}\\
        &\leq2\lambda_1r_t+\lambda_1\sqrt{r_t}\\
        &\stackrel{(e)}{\leq}3\lambda_1\sqrt{r_t},
    \end{align*}
    where $(d)$ is from Lemma \ref{apdx.lemma.sin_x_y}; and $(e)$ holds when $r_t\leq1$.

    Since $\rank(\theta_t\x_t\x_t^\top)=\rank(\lambda_1\uu_1\uu_1^\top)=1$, it is guaranteed that $\rank(\theta_t\x_t\x_t^\top-\lambda_1\uu_1\uu_1^\top)\leq2$. Then, $\|\theta_t\x_t\x_t^\top-\lambda_1\uu_1\uu_1^\top\|_\fro\leq\sqrt{2}\|\theta_t\x_t\x_t^\top-\lambda_1\uu_1\uu_1^\top\|\leq3\sqrt{2}\lambda_1\sqrt{r_t}$.

    Thus, to obtain $\|\x_{t_\varepsilon}\theta_{t_\varepsilon}\x_{t_\varepsilon}^\top-\A_1\|_\fro\leq\mathcal{O}(\lambda_1\varepsilon)$, it is sufficient to have $r_{t_\varepsilon}\leq\varepsilon^2$.
    
    From $r_t\leq cm\big(1-\frac{\eta(\lambda_1-\lambda_2)}{2}\big)^{2t}$, we just require $cm\big(1-\frac{\eta(\lambda_1-\lambda_2)}{2}\big)^{2t_\varepsilon}\leq\varepsilon^2$, i.e., $$t_\varepsilon\geq\frac{2\log(\frac{\varepsilon}{\sqrt{cm}})}{2\log(1-\frac{\eta(\lambda_1-\lambda_2)}{2})}=\mathcal{O}(\frac{\log(m)+\log(\frac{1}{\varepsilon})}{\eta(\lambda_1-\lambda_2)}).$$
\end{proof}

\subsection{Proof of Theorem \ref{theorem_recursive_rank1}} \label{supplement.proofs.theorem_recursive_rank1}
\leavevmode\par
\begin{proof}
    For simplicity, we take $\eta=\frac{1}{4}$.
    
    Let $\D_j:=\sum_{i=j+1}^{r_A}\sigma_i\uu_i\uu_i^\top,j=0,1,\ldots,r_A$. Then, $\D_0=\A$ and $\D_{r_A}=\bm{0}$.

    Let $\E_j:=\A^{(j)}-\D_j,j=0,\ldots,r_A$. Then, $\E_0=\bm{0}$ and $\E_{r_A}=\A^{(r_A)}-\D_{r_A}=\A^{(r_A)}$.
    
    Let $(\lambda_1^{(j-1)},\vv_1^{(j-1)})$ be the leading eigen couple of $\A^{(j-1)}$, i.e., $\lambda_1^{(j-1)}=\lambda_{max}(\A^{(j-1)})$, $\A^{(j-1)}\vv_1^{(j-1)}=\lambda_1^{(j-1)}\vv_1^{(j-1)}$ and $\|\vv_1^{(j-1)}\|=1$.

    From the definition of $\A^{(j)},\D_j$ and $\E_j$, for any $1\leq j\leq r_A$, we have that
    \begin{align*}
        \A^{(j-1)}=\D_{j-1}+\E_{j-1}=\sigma_j\uu_j\uu_j^\top+\D_j+\E_{j-1}.
    \end{align*}
    Since the largest eigenvalues of $\A^{(j-1)}$ and $\D_{j-1}$ are $\lambda_1^{(j-1)}$ and $\sigma_j$ respectively, we can derive the following inequality through Weyl's inequality \cite{franklin2000matrix}:
    \begin{align}\label{weyl_dif_lambda_sigma}
        |\lambda_1^{(j-1)}-\sigma_j|\leq\|\E_{j-1}\|.
    \end{align}
    From Davis-Kahan theorem \cite{yu2015useful}, we have that
    \begin{align}\label{davis_dif_lambda_sigma}
        \|\vv_1^{(j-1)}(\vv_1^{(j-1)})^\top-\uu_j\uu_j^\top\|\leq\frac{2\|\E_{j-1}\|}{\sigma_j-\sigma_{j+1}}.
    \end{align}
    Combining \eqref{weyl_dif_lambda_sigma} and \eqref{davis_dif_lambda_sigma}, we can derive the following inequality:
    \begin{align}\label{dif_largest_eigen}
        \|\lambda_1^{(j-1)}\vv_1^{(j-1)}(\vv_1^{(j-1)})^\top-\sigma_j\uu_j\uu_j^\top\|&\leq\|(\lambda_1^{(j-1)}-\sigma_j)\vv_1^{(j-1)}(\vv_1^{(j-1)})^\top\|\\
        &~~~~~~~+\|\sigma_j(\vv_1^{(j-1)}(\vv_1^{(j-1)})^\top-\uu_j\uu_j^\top)\|\nonumber\\
        &=|\lambda_1^{(j-1)}-\sigma_j|+\sigma_j\|\vv_1^{(j-1)}(\vv_1^{(j-1)})^\top-\uu_j\uu_j^\top\|\nonumber\\
        &\leq(1+\frac{2\sigma_j}{\sigma_j-\sigma_{j+1}})\|\E_{j-1}\|.
    \end{align}
    From $\A^{(j)}=\A^{(j-1)}-\B^{(j-1)}$, we obtain
    \begin{align*}
        \E_{j}&=\A^{(j)}-\D_j\\
        &=\A^{(j-1)}-\B_j-(\D_{j-1}-\sigma_j\uu_j\uu_j^\top)\\
        &=\A^{(j-1)}-\D_{j-1}-(\B_j-\sigma_j\uu_j\uu_j^\top).
    \end{align*}
    Thus, we have that 
    \begin{align*}
        \|\E_j\|&\leq\|\A^{(j-1)}-\D_{j-1}\|+\|\B^{(j-1)}-\sigma_j\uu_j\uu_j^\top\|\\
        &\leq\|\E_{j-1}\|+\|\B^{(j-1)}-\lambda_1^{(j-1)}\vv_1^{(j-1)}(\vv_1^{(j-1)})^\top\|+\|\lambda_1^{(j-1)}\vv_1^{(j-1)}(\vv_1^{(j-1)})^\top-\sigma_j\uu_j\uu_j^\top\|\\
        &\stackrel{(a)}{\leq}(2+\frac{2\sigma_j}{\sigma_j-\sigma_{j+1}})\|\E_{j-1}\|+\|\B^{(j-1)}-\lambda_1^{(j-1)}\vv_1^{(j-1)}(\vv_1^{(j-1)})^\top\|\\
        &\stackrel{(b)}{\leq}6\|\E_{j-1}\|+\|\B^{(j-1)}-\lambda_1^{(j-1)}\vv_1^{(j-1)}(\vv_1^{(j-1)})^\top\|\\
        &\stackrel{(c)}{\leq}6\|\E_{j-1}\|+\varepsilon_{j-1},
    \end{align*}
    where $(a)$ is from inequality \eqref{dif_largest_eigen}; $(b)$ is by the assumption that $\frac{\sigma_{i}}{\sigma_{i+1}}\geq2,i=1,2\ldots,r_A-1$; and we let $\varepsilon_{j-1}:=\|\B^{(j-1)}-\lambda_1^{(j-1)}\vv_1^{(j-1)}(\vv_1^{(j-1)})^\top\|$ in $(c)$.

    Recall that we let $(\lambda_1^{(j-1)},\vv_1^{(j-1)})$ be the leading eigen couple of $\A^{(j-1)}$. We also let $(\lambda_2^{(j-1)},\vv_2^{(j-1)})$ be the second largest eigen couple of $\A^{(j-1)}$ and $(\lambda_{min}^{(j-1)},\vv_{min}^{(j-1)})$ be the smallest eigen couple of $\A^{(j-1)}$, i.e., $\lambda_2^{(j-1)}=\lambda_2(\A^{(j-1)})$ and $\lambda_{min}^{(j-1)}=\lambda_{min}(\A^{(j-1)})$.
    
    In the first round, we have that $\A^{(0)}=\A$. Thus, it follows from the assumption that $\lambda_1^{(0)}>\lambda_2^{(0)}$, $\lambda_1^{(0)}>|\lambda_{min}^{(0)}|$ and $\eta_0=\frac{1}{4}\leq\frac{1}{\lambda_1^{(0)}}=1$, which means that Theorem \ref{theorem_DMD_rank1} is applicable.

    From Theorem \ref{theorem_DMD_rank1}, we can obtain that $\|\B^{(0)}-\lambda_1^{(0)}\vv_1^{(0)}(\vv_1^{(0)})^\top\|\leq\varepsilon_{step}$ after at most $\mathcal{O}(\frac{\log(m)+\log(\frac{1}{\varepsilon_{step}})}{\eta_0(\lambda_1^{(0)}-\lambda_2^{(0)})})$ iterations. We now set $\varepsilon_{step}=\frac{\sqrt{\varepsilon}}{\sqrt{r_A}6^{r_A}}$ to be the objective error in each round.
    
    Since $1=\lambda_1^{(0)}\geq2\lambda_2^{(0)}\geq\frac{2}{\kappa}$, which guarantees that $\frac{1}{\lambda_1^{(0)}-\lambda_2^{(0)}}\leq\kappa$, and we apply $\eta_0=\frac{1}{4}$, it follows that after running  $T_\star=\mathcal{O}\big(\kappa(\log(m)+r_A)+\kappa\log(\frac{1}{\varepsilon})\big)$ iterations in the first round, $\varepsilon_0=\|\B^{(0)}-\lambda_1^{(0)}\vv_1^{(0)}(\vv_1^{(0)})^\top\|\leq\frac{\sqrt{\varepsilon}}{\sqrt{r_A}6^{r_A}}$ is achieved. This implies that $\|\E_1\|\leq6\|\E_0\|+\varepsilon_0\leq\frac{\sqrt{\varepsilon}}{\sqrt{r_A}6^{r_A}}$.

    Now, assume that $\|\E_k\|\leq\varepsilon_{step}\cdot\frac{6^{k}-1}{5}\leq\frac{\sqrt{\varepsilon}}{5}$ after $k$ rounds, $1\leq k\leq r_A-1$.

    Since $\|\A^{(k)}-\D_k\|=\|\E_k\|$, we can obtain from Weyl's inequality \cite{franklin2000matrix} that
    \begin{align*}
    \frac{1}{\kappa}-\|\E_k\|\leq\sigma_{k+1}-\|\E_k\|\leq&\lambda_1^{(k)}\leq\sigma_{k+1}+\|\E_k\|\leq2,\\
        &\lambda_2^{(k)}\leq\sigma_{k+2}+\|\E_k\|,\\
        -\|\E_k\|\leq&\lambda_{min}^{(k)}.
    \end{align*}
    Thus, it follows that
    \begin{align*}
        \lambda_1^{(k)}-\lambda_2^{(k)}&\geq\sigma_{k+1}-\sigma_{k+2}-2\|\E_k\|\geq\sigma_{k+2}-\frac{(6^k-1)\cdot2\sqrt{\varepsilon}}{5\cdot6^{r_A}}\geq\frac{1}{\kappa}-\frac{\sqrt{\varepsilon}}{2}\stackrel{(d)}{\geq}\frac{1}{2\kappa},\\
        |\lambda_{min}^{(k)}|&\leq\|\E_k\|\leq\frac{\sqrt{\varepsilon}}{5}\stackrel{(d)}{\leq}\frac{1}{\kappa}-\frac{\sqrt{\varepsilon}}{5}\leq\lambda_1^{(k)},\\
        \eta&=\frac{1}{4}\leq\frac{1}{2\lambda_1^{(k)}},
    \end{align*}
    where $(d)$ follows from $\sqrt{\varepsilon}\leq\frac{1}{\kappa}$.
    
    This implies that the requirements of Theorem \ref{theorem_DMD_rank1} is satisfied and we can apply it in the $(k+1)$-th round.

    From our assumption, we run $T_\star=\mathcal{O}\big(\kappa(\log(m)+r_A)+\kappa\log(\frac{1}{\varepsilon})\big)$ iterations in the $(k+1)$-th round. Then, $\varepsilon_{k}=\|\B^{(k)}-\lambda_1^{(k)}\vv_1^{(k)}(\vv_1^{(k)})^\top\|\leq\varepsilon_{step}$, which implies that $$\|\E_{k+1}\|\leq6\|\E_k\|+\varepsilon_k\leq\varepsilon_{step}\cdot\frac{6^k-1}{5}\cdot6+\varepsilon_{step}=\varepsilon_{step}\cdot\frac{6^{k+1}-1}{5}.$$
    Thus, $\|\E_{k+1}\|\leq\varepsilon_{step}\cdot\frac{6^{k+1}-1}{5}$ holds after $k+1$ rounds.

    Therefore, after at most $r_A$ rounds, we have that $$\|\A-\tilde{\A}^{(r_A)}\|=\|\A-\sum_{j=0}^{r_A-1}\B^{(j)}\|=\|\A^{(r_A)}\|=\|\E_{r_A}\|\leq\frac{\sqrt{\varepsilon}}{\sqrt{r_A}}.$$
    Since $\rank(\tilde{\A}^{r_A})\leq r_A$ and $\rank(\A)=r_A$, it is guaranteed that $\rank(\A-\tilde{\A}^{(r_A)})\leq2r_A$. Thus, $\|\A-\tilde{\A}^{(r_A)}\|_\fro\leq\sqrt{2r_A}\cdot\frac{\sqrt{\varepsilon}}{\sqrt{r_A}}=\sqrt{2\varepsilon}$, which implies that $\frac{1}{4}\|\A-\tilde{\A}^{(r_A)}\|_\fro^2\leq\varepsilon$. And the total number of iterations is at most $\sum_{j=0}^{r_A-1}T_\star=\mathcal{O}\big(\kappa(r_A\log(m)+r_A^2)+r_A\kappa\log(\frac{1}{\varepsilon})\big)$.

    We now prove that $\theta_{T_\star}^{(j)}>0$ for $j=0,\ldots,r_A-1$, i.e., for all the $r_A$ rounds.

    Recall that we have shown $\lambda_1^{(j)}\geq\frac{1}{\kappa}-\frac{\sqrt{\varepsilon}}{5}\geq\frac{4}{5\kappa}>0$ and $\|\B^{(j)}-\lambda_1^{(j)}\vv_1^{(j)}(\vv_1^{(j)})^\top\|=\|\theta_{T_\star}^{(j)}\x_{T_\star}^{(j)}(\x_{T_\star}^{(j)})^\top-\lambda_1^{(j)}\vv_1^{(j)}(\vv_1^{(j)})^\top\|\leq\varepsilon_{step}=\frac{\sqrt{\varepsilon}}{\sqrt{r_A}6^{r_A}}\leq\frac{1}{6\kappa}$. For the contradiction, suppose that $\theta_{T_\star}^{(j)}\leq0$.

    Let $\Z:=\theta_{T_\star}^{(j)}\x_{T_\star}^{(j)}(\x_{T_\star}^{(j)})^\top-\lambda_1^{(j)}\vv_1^{(j)}(\vv_1^{(j)})^\top\in\mathbb{S}^m$. Then, we have that 
    \begin{align*}
        (\vv_1^{(j)})^\top\Z\vv_1^{(j)}&=\theta_{T_\star}^{(j)}((\vv_1^{(j)})^\top\x_{T_\star}^{(j)})^2-\lambda_1^{(j)}\\
        &\leq-\lambda_1^{(j)},
    \end{align*}
    which means that $\|\Z\|\geq\lambda_1^{(j)}\geq\frac{4}{5\kappa}$, contradictory with $\|\Z\|\leq\frac{1}{6\kappa}$.

    Therefore, $\theta_{T_\star}^{(j)}>0$ for $j=0,\ldots,r_A-1$.
\end{proof}

\section{Useful Lemmas} \label{supplement.lemmas}
\begin{lemma}\label{init_with_high_probability}
        Let event $F=\{\sigma_{r_A}^2(\U^\top\X_0)\geq\frac{(r-r_A)^2}{c_1mr}\}$, where $c_1>\max\{1,36C_1^2\}$ is a universal constant, with universal constant $C_1$ given in Lemma \ref{lemma.smallest-sigma}. With respect to the randomness in $\X_0$, event $F$ occurs with probability at least $$1-\exp(-m/2)-C_3^{r-r_A+1}-\exp(-C_2r),$$
        where $C_2>0$ and $C_3=\frac{6C_1}{\sqrt{c_1}}\in(0,1)$ are universal constants.
    \end{lemma}
    \begin{proof}
        Since the initialization $\X_0$ satisfies the conditions stated in Lemma \ref{lemma.phi0}, we can apply the lemma directly. In particular, substituting $\tau=\frac{6}{\sqrt{c_1}}$ yields the desired result.
    \end{proof}

\begin{lemma}\label{lemma.required-acc-sensing}
    If $\tr(\I_{r_A}-\bfPhi_t\bfPhi_t^\top)\leq\rho$, iterations guarantee that $f(\X_t,\bfTheta_t)\leq\rho$.
\end{lemma}
\begin{proof}
    We have that
    \begin{align*}
        \|\X_t\bfTheta_t\X_t^\top-\A\|_\fro&=\|\X_t\X_t^\top\A\X_t\X_t^\top\|_\fro\\
        &=\|\X_t\X_t^\top\A\X_t\X_t^\top-\A\X_t\X_t^\top+\A\X_t\X_t^\top-\A\|_\fro\\
        &\leq\|(\X_t\X_t^\top-\I_m)\A\X_t\X_t^\top\|_\fro+\|\A(\X_t\X_t^\top-\I_m)\|_\fro\\
        &\stackrel{(a)}{\leq}\|(\X_t\X_t^\top-\I_m)\U\|_\fro\|\bfSigma\U^\top\X_t\X_t^\top\|+\|\U\bfSigma\|\|\U^\top(\X_t\X_t^\top-\I_m)\|_\fro\\
        &\stackrel{(b)}{\leq}2\|(\I_m-\X_t\X_t^\top)\U\|_\fro,
    \end{align*}
    where $(a)$ is by $\|\A\B\|_\fro\leq\|\A\|_\fro\|\B\|$; and $(b)$ is by $\|\bfSigma\|,\|\U\|,\|\X_t\|\leq1$.

    Since we have that
    \begin{align*}
        \|(\I_m-\X_t\X_t^\top)\U\|_\fro^2&=\tr(\U^\top(\I_m-\X_t\X_t^\top)(\I_m-\X_t\X_t^\top)^\top\U)\\
        &=\tr(\I_{r_A}-\bfPhi_t\bfPhi_t^\top)\\
        &\leq\rho,
    \end{align*}
    and combine above inequalities, we have that
    \begin{align*}
        \|\X_t\bfTheta_t\X_t^\top-\A\|_\fro^2\leq4\|(\I_m-\X_t\X_t^\top)\U\|_\fro^2\leq4\rho.
    \end{align*}
    This finishes the proof.
    \end{proof}

\begin{lemma}\label{lemma.singular_ascending}
    Denote the orthonormal complement of $\U$ be $\U_\perp\in\mathbb{R}^{m\times (m-r_A)}$. Define the $(m-r_A)\times r$ matrix $\bfPsi_t:=\U_\perp^\top\X_t$ to characterize the alignment of $\X_t$ and $\U_\perp$. Assuming $\eta\leq1$, we have that $\bfPsi_{t+1}\bfPsi_{t+1}^\top\preceq\bfPsi_t\bfPsi_t^\top$. Moreover, if $r_A\leq\frac{m}{2}$, it is guaranteed to have $\sigma_{r_A}^2(\bfPhi_{t+1})\geq\sigma_{r_A}^2(\bfPhi_t)$.
\end{lemma}
\begin{proof}
    From update \eqref{update_X}, we have that
    \begin{align*}
        \bfPsi_{t+1}&=\U_\perp^\top(\X_t-\eta\G_t)(\I_r+\eta^2\G_t^\top\G_t)^{-1/2}\\
        &=(\bfPsi_t+\eta\U_\perp^\top(\I_m-\X_t\X_t^\top)\A\X_t)(\I_r+\eta^2\G_t^\top\G_t)^{-1/2}\\
        &=(\bfPsi_t-\eta\bfPsi_t\bfTheta_t)(\I_r+\eta^2\G_t^\top\G_t)^{-1/2}\\
        &=\bfPsi_t(\I_r-\eta\bfTheta_t)(\I_r+\eta^2\G_t^\top\G_t)^{-1/2}.
    \end{align*}
    
    With this, we can see that
    \begin{align*}
        \bfPsi_{t+1}\bfPsi_{t+1}^\top&=\bfPsi_t(\I_r-\eta\bfTheta_t)(\I_r+\eta^2\G_t^\top\G_t)^{-1}(\I_r-\eta\bfTheta_t)\bfPsi_t^\top\\
        &\stackrel{(a)}{\preceq}\bfPsi_t\bfPsi_t^\top,
    \end{align*}
    where $(a)$ follows from the fact that the three matrices in between are all PSD and their largest eigenvalue is smaller than 1 given our choice of $\eta$. This gives the proof of the first part of this lemma.

    To show $\sigma_{r_A}^2(\bfPhi_{t+1})\geq\sigma_{r_A}^2(\bfPhi_t)$, notice that given $2r_A\leq m$, we have from Lemma \ref{lemma.apdx.sin-cos}, which shows that $\sigma_{r_A}^2(\bfPhi_t)=1-\sigma_{r+1-r_A}^2(\bfPsi_t)$ and $\sigma_{r_A}^2(\bfPhi_{t+1})=1-\sigma_{r+1-r_A}^2(\bfPsi_{t+1})$. The conclusion is straightforward.
\end{proof}

 \begin{lemma}\label{lemma.apdx.trace-lower-bound}
	Suppose that $\PP$ and $\Q$ are $m \times m$ diagonal matrices, with non-negative diagonal entries. Let $\bfS\in\mathbb{S}^{m}$ be a positive definite matrix with smallest eigenvalue $\lambda_{\min}$, then we have that 
	\begin{align*}
		\tr(\PP\bfS\Q)  \geq \lambda_{\min}\tr(\PP\Q).
	\end{align*}
\end{lemma}
\begin{proof}
	Let $p_i$ and $q_i$ be the $(i,i)$-th entry of $\PP$ and $\Q$, respectively. Then we have that
	\begin{align*}
		\tr(\PP\bfS\Q) = \sum_i p_i \bfS_{i,i} q_i \geq \lambda_{\min}\sum_i p_i  q_i = \lambda_{\min}\tr(\PP\Q),
	\end{align*}
	where the last inequality comes from $\bfS$ being positive definite, i.e., $\bfS_{i,i} = \mathbf{e}_i^\top \bfS \mathbf{e}_i \geq \lambda_{\min}$. 
\end{proof}

\begin{lemma}\label{apdx.lemma.inverse}
	Given a PSD matrix $\A$, we have that $(\I + \A)^{-1} \succeq \I - \A$.	
    \end{lemma}
    \begin{proof}
	Diagonalizing both sides and using $1/(1 + \lambda) \geq 1 - \lambda, \forall \lambda \geq 0$ yields the result.
    \end{proof}
    
\begin{lemma}\label{lemma.apdx.sin-cos}
	Let $\X \in \st(m, r)$ and $\U \in \st(m, r_A)$. Let $\U_\perp \in \mathbb{R}^{m \times (m -r_A)}$ be an orthonormal basis for the orthogonal complement of $\Span(\U)$. Denote $\bfPhi = \U^\top \X\in\mathbb{R}^{r_A\times r}$ and $\bfPsi = \U_\perp^\top \X\in\mathbb{R}^{(m-r_A)\times r}$. It is guaranteed that $\sigma_i^2(\bfPhi) + \sigma_{r+1-i}^2(\bfPsi) = 1$ holds for $i \in \{1, 2, \ldots, r \}$.
\end{lemma}
\begin{proof}
	Since $\X$ lies in the Stiefel manifold, we have that 
	\begin{align}\label{eq.apdx.commute}
		\I_r & = \X^\top \X = \X^\top  \I_m \X = \X^\top [\U, \U_\perp]
		\begin{bmatrix} 
			\U^\top \\
			\U_\perp^\top
		\end{bmatrix}
		\X  \\
		& = \bfPhi^\top \bfPhi  + \bfPsi^\top \bfPsi.  \nonumber
	\end{align}
	
	Equation \eqref{eq.apdx.commute} shows that $\bfPsi^\top \bfPsi$ and $\bfPhi^\top \bfPhi $ commute, i.e.,
	\begin{align*}
		(\bfPhi^\top \bfPhi)	 (\bfPsi^\top \bfPsi) & = (\bfPhi^\top \bfPhi)	 (\I_r - \bfPhi^\top \bfPhi) = \bfPhi^\top \bfPhi  - \bfPhi^\top \bfPhi\bfPhi^\top \bfPhi \\
		& = (\I_r - \bfPhi^\top \bfPhi) (\bfPhi^\top \bfPhi)	 = (\bfPsi^\top \bfPsi) (\bfPhi^\top \bfPhi). 
	\end{align*}
	The commutativity shows that the eigenspaces of $\bfPhi^\top \bfPhi$ and $\bfPsi^\top \bfPsi$ coincide. As a result, we have again from \eqref{eq.apdx.commute} that $\sigma_i^2(\bfPhi) + \sigma_{r+1-i}^2(\bfPsi) = 1$ for $i \in \{1, 2, \ldots, r \}$.
\end{proof}

\begin{lemma}\label{apdx.lemma.control_second_order}
    Let $\X\in\st(m,r)$ and $\U\in\st(m,r_A)$. Let $\U_\perp\in\mathbb{R}^{m\times (m-r_A)}$ be an orthonormal basis for the orthogonal complement of $\Span(\U)$. Let $\bar{\U}=[\U,\U_\perp]\in\mathbb{R}^{m\times m}$. Denote by $\phi_{i,j}$ the $(i,j)$-th entry of $\bar{\U}^\top\X\X^\top\bar{\U}$. It is guaranteed that $\phi_{i,i}=\sum_{j=1}^m\phi_{i,j}^2$ for any $i\in\{1,2,\ldots,m\}$.
\end{lemma}
\begin{proof}
    Let $\bar{\U}=[\uu_1,\uu_2,\ldots,\uu_m]$, where $\uu_1,\uu_2,\ldots,\uu_m\in\mathbb{R}^m$. Then, we have that
    \begin{align*}
        \bar{\U}^\top\X\X^\top\bar{\U}&=\begin{bmatrix}
            \uu_1^\top\\
            \uu_2^\top\\
            \vdots\\
            \uu_m^\top
        \end{bmatrix}
        \X\X^\top\begin{bmatrix}
            \uu_1,\uu_2,\ldots,\uu_m
        \end{bmatrix}\\
        &=\begin{bmatrix}
            \uu_1^\top\X\X^\top\uu_1&\uu_1^\top\X\X^\top\uu_2&\cdots&\uu_1^\top\X\X^\top\uu_m\\
            \uu_2^\top\X\X^\top\uu_1&\uu_2\X\X^\top\uu_2&\cdots&\uu_2\X\X^\top\uu_m\\
            \vdots & \ddots & \ddots & \vdots\\
            \uu_m^\top\X\X^\top\uu_m & \uu_m^\top\X\X^\top\uu_2 & \cdots & \uu_m^\top\X\X^\top\uu_m
        \end{bmatrix}.
    \end{align*}

    Since $\bar{\U}^\top\bar{\U}=\bar{\U}\bar{\U}^\top=\I_m$ and $\X^\top\X=\I_r$, it follows that $(\bar{\U}^\top\X\X^\top\bar{\U})^2=\bar{\U}^\top\X\X^\top\bar{\U}$.

    Then, $[(\bar{\U}^\top\X\X^\top\bar{\U})^2]_{i,i}=(\bar{\U}^\top\X\X^\top\bar{\U})_{i,i}$, which means that
    \begin{align*}
        \sum_{j=1}^{m}(\uu_i^\top\X\X^\top\uu_j)^2=\uu_i^\top\X\X^\top\uu_i,
    \end{align*}
    i.e., $\phi_{i,i}=\sum_{j=1}^m\phi_{i,j}^2$ for any $i\in\{1,2,\ldots,m\}$.
\end{proof}

\begin{lemma}\label{apdx.lemma.trace_control_norm}
    Assume that $\frac{\sigma_i}{\sigma_{i+1}}\geq\beta,\,i=1,2,\ldots,r_A-1$, for some $\beta>1$. Let $\A\in\mathbb{S}^m$ be a PSD matrix and $\U\bfSigma\U^\top$ be its compact SVD, where $\U=[\uu_1,\uu_2,\ldots,\uu_{r_A}]\in\mathbb{R}^{m\times r_A}$ and $\bfSigma\in\mathbb{S}^{r_A}$. Let $\X(t)\in\st(m,r)$ and $\phi_{i,i}(t):=\uu_i^\top\X(t)\X(t)^\top\uu_i$. Assume that $\phi_{i,i}(t)\geq1-\varepsilon,\,i=1,2,\ldots,j$ and $\phi_{i,i}(t)\leq\sqrt{\varepsilon},\,i=j+1,\ldots,r_A$. Then it holds that $\|\X(t)\bfTheta(t)\X(t)^\top-\A_j\|_\fro^2\leq\frac{4\beta^2}{(\beta-1)^2}\cdot\sigma_1^2\varepsilon$, where $\A_j:=\argmin\limits_{\rank(\hat{\A})\leq j}\, \|\hat{\A}-\A\|_\fro^2$.
\end{lemma}
\begin{proof}
    From the Eckart–Young–Mirsky theorem, we have that the best rank-$j$ approximation of $\A$ under the Frobenius norm is $\A_j=\sum_{k=1}^{j}\sigma_k\uu_k\uu_k^\top$. Then, we have the following upper bound
    \begin{align*}
        \|\X(t)\bfTheta(t)\X(t)^\top-\A_j\|_\fro&=\|\X(t)\X(t)^\top(\sum_{k=1}^{r_A}\sigma_k\uu_k\uu_k^\top)\X(t)\X(t)^\top-\sum_{k=1}^{j}\sigma_k\uu_k\uu_k^\top\|_\fro\\
        &\leq\|\sum_{k=1}^{j}\sigma_k(\X(t)\X(t)^\top\uu_k\uu_k^\top\X(t)\X(t)^\top-\uu_k\uu_k^\top)\|_\fro\\
        &~~~+\|\sum_{k=j+1}^{r_A}\sigma_k\X(t)\X(t)^\top\uu_k\uu_k^\top\X(t)\X(t)^\top\|_\fro.
    \end{align*}

    For the first term,
    \begin{align*}                              
        &\|\sum_{k=1}^{j}\sigma_k(\X(t)\X(t)^\top\uu_k\uu_k^\top\X(t)\X(t)^\top-\uu_k\uu_k^\top)\|_\fro\\
        &~~~\leq\sum_{k=1}^{j}\sigma_k\|\X(t)\X(t)^\top\uu_k\uu_k^\top\X(t)\X(t)^\top-\uu_k\uu_k^\top\X(t)\X(t)^\top+\uu_k\uu_k^\top\X(t)\X(t)^\top-\uu_k\uu_k^\top\|_\fro\\
        &~~~\leq\sum_{k=1}^{j}\sigma_k(\|(\X(t)\X(t)^\top\uu_k-\uu_k)\uu_k^\top\X(t)\X(t)^\top\|_\fro+\|\uu_k(\uu_k^\top\X(t)\X(t)^\top-\uu_k^\top)\|_\fro)\\
        &~~~\stackrel{(a)}{\leq}2\sum_{k=1}^{j}\sigma_k\|\X(t)\X(t)^\top\uu_k-\uu_k\|_\fro\\
        &~~~=2\sum_{k=1}^{j}\sigma_k\sqrt{1-\phi_{k,k}(t)}\\
        &~~~\leq2\sqrt{\varepsilon}\sum_{k=1}^{j}\sigma_k,
    \end{align*}
    where $(a)$ is by $\|\X(t)\|,\|\uu_k\|\leq1$.
    
    For the second term,
    \begin{align*}
        \|\sum_{k=j+1}^{r_A}\sigma_k\X(t)\X(t)^\top\uu_k\uu_k^\top\X(t)\X(t)^\top\|_\fro&\leq\sum_{k=j+1}^{r_A}\sigma_k\|\X(t)\X(t)^\top\uu_k\uu_k^\top\X(t)\X(t)^\top\|_\fro\\
        &=\sum_{k=j+1}^{r_A}\sigma_k\|\uu_k^\top\X(t)\X(t)^\top\|_\fro^2\\
        &=\sum_{k=j+1}^{r_A}\sigma_k\phi_{k,k}(t)\\
        &\leq\sqrt{\varepsilon}\sum_{k=j+1}^{r_A}\sigma_k.
    \end{align*}

    Combining these upper bounds, we arrive at
    \begin{align*}
        \|\X(t)\bfTheta(t)\X(t)^\top-\A_j\|_\fro\leq2\sqrt{\varepsilon}\sum_{k=1}^{r_A}\sigma_k\leq\frac{2\beta}{\beta-1}\sqrt{\varepsilon},
    \end{align*}
    i.e., $\|\X(t)\bfTheta(t)\X(t)^\top-\A_j\|_\fro^2\leq\frac{4\beta^2}{(\beta-1)^2}\cdot\sigma_1^2\varepsilon$.
\end{proof}

\begin{lemma}\label{apdx.lemma.crossed_terms}
    Let $\{\uu_i\}_{i=1}^m\subset\mathbb{R}^m$ be an orthonormal basis of $\mathbb{R}^m$. Let $\X(t)\in\st(m,r)$ and $\phi_{i,j}(t):=\uu_i^\top\X(t)\X(t)^\top\uu_j$. Then it holds that for any $i\neq j$, $\phi_{i,j}^2(t)\leq\min\{\phi_{i,i}(t)\phi_{j,j}(t),$\\
    $(1-\phi_{i,i}(t))(1-\phi_{j,j}(t))\}$.
\end{lemma}
\begin{proof}
    By the Cauchy-Schwartz inequality, we can directly derive that
    \begin{align*}
        \phi_{i,j}^2(t)&=(\uu_i^\top\X(t)\X(t)^\top\uu_j)^2\\
        &\leq(\uu_i^\top\X(t)\X(t)^\top\uu_i)(\uu_j^\top\X(t)\X(t)^\top\uu_j)\\
        &=\phi_{i,i}(t)\phi_{j,j}(t).
    \end{align*}
    We just need to prove that $\phi_{i,j}^2(t)\leq(1-\phi_{i,i}(t))(1-\phi_{j,j}(t))$.

    Since $i\neq j$, we have that $\uu_i^\top\uu_j=0$, which leads to
    \begin{align*}
        \phi_{i,j}^2(t)&=(\uu_i^\top\X(t)\X(t)^\top\uu_j)^2\\
        &=[\uu_i^\top(\I_m-\X(t)\X(t)^\top)\uu_j]^2\\
        &=[\uu_i^\top(\I_m-\X(t)\X(t)^\top)(\I_m-\X(t)\X(t)^\top)\uu_j]^2\\
        &\stackrel{(a)}{\leq} \|\uu_i^\top(\I_m-\X(t)\X(t)^\top)\|^2\cdot \|\uu_j^\top(\I_m-\X(t)\X(t)^\top)\|^2\\
        &=(1-\phi_{i,i}(t))(1-\phi_{j,j}(t)),
    \end{align*}
    where $(a)$ is by Cauchy-Schwartz inequality.

    This completes the proof.
\end{proof}

\begin{lemma}\label{apdx.lemma.singular_ascending_cont}
    Denote the orthonormal complement of $\U$ as $\U_\perp\in\mathbb{R}^{m\times (m-r_A)}$. Define the $(m-r_A)\times r$ matrix $\bfPsi(t):=\U_\perp^\top\X(t)$ to characterize the alignment of $\X(t)$ and $\U_\perp$. We have that for any $t_2\geq t_1\geq0$, $\bfPsi(t_2)\bfPsi(t_2)^\top\preceq\bfPsi(t_1)\bfPsi(t_1)^\top$. Moreover, if $r_A\leq\frac{m}{2}$, it is guaranteed to have $\sigma_{r_A}^2(\bfPhi(t_2))\geq\sigma_{r_A}^2(\bfPhi(t_1))$.
\end{lemma}
\begin{proof}
    Given the definition of $\bfPsi(t)$, we have that
    \begin{align*}
        \dot{\bfPsi}(t)&=\U_\perp^\top\dot{\X}(t)\\
        &\stackrel{(a)}{=}-\U_\perp^\top\X(t)\X(t)^\top\A\X(t),
    \end{align*}
    where $(a)$ is from the dynamics of $\X(t)$.

    With this dynamic, we further have
    \begin{align*}
        \frac{\mathsf{d}\bfPsi(t)\bfPsi(t)^\top}{\mathsf{d}t}=-2\bfPsi(t)\A\bfPsi(t)^\top\preceq\bm{0}.
    \end{align*}

    Thus, for any $t_2\geq t_1\geq0$, $\bfPsi(t_2)\bfPsi(t_2)^\top\preceq\bfPsi(t_1)\bfPsi(t_1)^\top$.
    
    To show $\sigma_{r_A}^2(\bfPhi(t_2))\geq\sigma_{r_A}^2(\bfPhi(t_1))$, notice that given $2r_A\leq m$, we have from Lemma \ref{lemma.apdx.sin-cos}, which shows $\sigma_{r_A}^2(\bfPhi(t_2))=1-\sigma_{r+1-r_A}^2(\bfPsi(t_2))$ and $\sigma_{r_A}^2(\bfPhi(t_1))=1-\sigma_{r+1-r_A}^2(\bfPsi(t_1))$. The conclusion is straightforward.
\end{proof}

\begin{lemma}\label{apdx.lemma.diag_lower_bound}
    For any $t\geq0$, $1\leq k\leq r_A$, it is guaranteed that $$\uu_k^\top\X(t)\X(t)^\top\uu_k\geq\sigma_{r_A}(\U^\top\X(0)\X(0)^\top\U).$$
\end{lemma}
\begin{proof}
    From Lemma \ref{apdx.lemma.singular_ascending_cont}, we have that 
    $$\sigma_{r_A}(\U^\top\X(t)\X(t)^\top\U)\geq\sigma_{r_A}(\U^\top\X(0)\X(0)^\top\U).$$

   Together with 
   \begin{align*}
       \uu_k^\top\X(t)\X(t)^\top\uu_k&=\e_k^\top\U^\top\X(t)\X(t)^\top\U\e_k\\
       &\geq\sigma_{r_A}(\U^\top\X(t)\X(t)^\top\U)\e_k^\top\e_k\\
       &=\sigma_{r_A}(\U^\top\X(t)\X(t)^\top\U),
   \end{align*} the conclusion is straightforward.
\end{proof}

\begin{lemma}\label{apdx.lemma.aux888}
   Let $\A \in \mathbb{R}^{m \times n}$ be a matrix with full column rank and $\B \in \mathbb{R}^{n \times p}$ be a non-zero matrix. Let $\sigma_{\min}(\cdot)$ denote the smallest non-zero singular value. Then it holds that $\sigma_{\min}(\A\B) \geq\sigma_{\min}(\A) \sigma_{\min}(\B)$.
\end{lemma}
\begin{proof}
    Using the min-max principle for singular values,
    \begin{align*}
        \sigma_{\min}(\A\B) & = \min_{\| \x \|=1, \x \in \text{ColSpan}(\B)} \| \A\B\x \| \\
        & = \min_{\| \x \|=1, \x \in \text{ColSpan}(\B)} \Big{\|} \A \frac{\B\x }{\| \B\x  \|} \Big{\|} \cdot \| \B\x  \|\\
        & \stackrel{(a)}{=} \min_{\| \x \|=1, \| \y\|=1, \x \in \text{ColSpan}(\B), \y \in \text{ColSpan}(\B)} \| \A\y \|  \cdot \| \B\x  \| \\
        & \geq  \min_{\| \y\|=1, \y \in \text{ColSpan}(\B)} \| \A\y \|  \cdot \min_{\| \x \|=1,  \x \in \text{ColSpan}(\B)}  \| \B\x  \| \\
        & \geq  \min_{\| \y\|=1 } \| \A\y \|  \cdot \min_{\| \x \|=1,  \x \in \text{ColSpan}(\B)}  \| \B\x  \| \\
        & =\sigma_{\min}(\A) \sigma_{\min}(\B),
    \end{align*}
    where $(a)$ is by changing of variables, i.e., $\y = \B\x/ \| \B\x \|$.
\end{proof}

\begin{lemma}[\textbf{Theorem 2.2.1 of \cite{chikuse2012statistics}}]\label{lemma.unitform-st}
	If $\,\Z \in \mathbb{R}^{m \times r}$ has entries drawn i.i.d. from Gaussian distribution ${\mathcal N}(0,1)$, then $\X  = \Z (\Z^\top \Z)^{-1/2}$ is a random matrix uniformly distributed on $\st(m, r)$.	
\end{lemma}

\begin{lemma}\textbf{\textup{\cite{vershynin2010introduction}}}\label{lemma.largest-sigma}
    If $\,\Z \in \mathbb{R}^{m \times r}$ is a matrix whose entries are independently drawn from ${\mathcal N}(0,1)$. Then for every $\tau \geq 0$, with probability at least $1 - \exp(-\tau^2/2)$, we have 
    \begin{align*}
	\sigma_1(\Z) \leq \sqrt{m} + \sqrt{r} + \tau.
    \end{align*}
\end{lemma}

\begin{lemma}\textbf{\textup{\cite{rudelson2009smallest}}}\label{lemma.smallest-sigma}
    If $\,\Z\in \mathbb{R}^{m \times r}$ is a matrix whose entries are independently drawn from ${\mathcal N}(0,1)$. Suppose that $m \geq r$. Then for every $\tau \geq 0$, we have for two universal constants $C_1 > 0$ and $C_2 > 0$ that
    \begin{align*}
	\mathbb{P}\Big(  \sigma_r(\Z) \leq \tau (\sqrt{m} - \sqrt{r-1})  \Big) \leq (C_1 \tau)^{m -r +1} + \exp(-C_2 m).
    \end{align*}
\end{lemma}

\begin{lemma}\label{lemma.phi0}
    If $\,\U \in \st(m, r_A)$ is a fixed matrix, $\X \in \st(m, r)$ is uniformly sampled from $\st(m, r)$ using methods described in Lemma \ref{lemma.unitform-st}, and $r > r_A$, then we have that with probability at least $1 - \exp(-m/2) - (C_1 \tau)^{r -r_A +1} - \exp(-C_2 r)$,
	\begin{align*}
		\sigma_{r_A}(\U^\top \X) \geq  \frac{\tau(r - r_A +1)}{6\sqrt{mr}}.
	\end{align*}
\end{lemma}
\begin{proof}
	Since $\X \in \st(m, r)$ is uniformly sampled from $\st(m, r)$ using methods described in Lemma \ref{lemma.unitform-st}, we can write $\X = \Z(\Z^\top \Z)^{-1/2}$, where $\Z \in \mathbb{R}^{m \times r}$ has entries i.i.d. sampled from ${\mathcal N}(0,1)$. We thus have
	\begin{align*}
		\sigma_{r_A}(\U^\top \X) = \sigma_{r_A}\big(\U^\top \Z(\Z^\top \Z)^{-1/2} \big).
	\end{align*}
	
	We now consider $\U^\top \Z \in \mathbb{R}^{r_A \times r}$. It is clear that the entries of $\U^\top \Z$ are also i.i.d ${\mathcal N}(0,1)$ random variables. As a consequence of Lemma \ref{lemma.smallest-sigma}, we have that with probability at least $1 - (C_1 \tau)^{r -r_A +1} - \exp(-C_2 r)$, 
	\begin{align*}
		\sigma_{r_A}\big(\U^\top \Z \big) \geq \tau (\sqrt{r} - \sqrt{r_A- 1}).
	\end{align*}

	We also have from Lemma \ref{lemma.largest-sigma} that with probability at least $1 - \exp(-m/2)$,
	\begin{align*}
		\sigma_1(\Z^\top \Z) = \sigma_1^2(\Z) \leq (2\sqrt{m} + \sqrt{r})^2.
	\end{align*}

    Taking union bound, we have with probability at least $1 - \exp(-m/2) - (C_1 \tau)^{r -r_A +1} - \exp(-C_2 r)$, 
    \begin{align*}
	\sigma_{r_A}(\U^\top \X) \stackrel{(a)}{\geq} \frac{ \sigma_{r_A}\big(\U^\top \Z) }{\sigma_1(\Z)} = \frac{\tau (\sqrt{r} - \sqrt{r_A- 1})}{2\sqrt{m} + \sqrt{r}} \geq \frac{\tau(r - r_A +1)}{ 3 \sqrt{m} \cdot 2 \sqrt{r}}= \frac{\tau(r - r_A +1)}{6\sqrt{mr}},
    \end{align*}
    where $(a)$ comes from Lemma \ref{apdx.lemma.aux888}.
\end{proof}

\begin{lemma}[\textbf{Theorem 1.5.6 of \cite{muirhead2009aspects}}]\label{apdx.lemma.muirhead}
    Let $\p\sim\mathcal{N}(0,1)^m$. It is guaranteed that $\frac{\p}{\|\p\|}$ is uniformly random on the unit sphere $\mathcal{S}^{m-1}:=\{\x\in\mathbb{R}^m|\|\x\|=1\}$.
\end{lemma}

\begin{lemma}\label{apdx.lemma.beta_dist}
    Let $\Z\sim\mathcal{N}(0,1)^{m\times r}$ with $m>r$ and define $\X=\Z(\Z^\top\Z)^{-1/2}$. Let $\phi_{i,i}=\uu_i^\top\X\X^\top\uu_i$, $1\leq i\leq r_A$. Then, $\phi_{i,i}\stackrel{d}{=}\frac{\chi_r^2}{\chi_r^2+\chi_{m-r}^2}$, i.e., $\phi_{i,i}\sim\mathsf{Beta}\,(\frac{r}{2},\frac{m-r}{2})$.
\end{lemma}
\begin{proof}
    Since $\uu_i^\top\uu_i=1$, we can always find an orthogonal matrix $\Q$, s.t., $\Q\uu_i=\e_i$.

    Since $\Z\sim\mathcal{N}(0,1)^{m\times r}$, rotational invariance of the Gaussian distribution implies that $\Q\Z\sim\mathcal N(0,1)^{m\times r}$.

    By the construction of $\X$, we obtain $\Q\X=\Q\Z(\Z^\top\Z)^{-1/2}=\Q\Z((\Q\Z)^\top\Q\Z)^{-1/2}$. Therefore, $\X$ and $\Q\X$ share the same distribution.

    Since $\phi_{i,i}=\uu_i^\top\X\X^\top\uu_i=(\Q\uu_i)^\top(\Q\X)(\Q\X)^\top(\Q\uu_i)=\e_i^\top(\Q\X)(\Q\X)^\top\e_i$, without loss of generality, we just need to consider $\hat{\phi}_{i,i}:=\e_i^\top\X\X^\top\e_i$.

    From Lemma \ref{lemma.unitform-st}, it follows that $\X$ is uniformly random on $\st(m,r)$. Then, the random subspace $\Span(\X)$ has a uniform distribution on the Grassmannian $\gr(m,r)=\{\mathcal{S}\subset\mathbb{R}^m|\mathsf{dim}(\mathcal{S})=r\}$.

    As stated in Section 1.4.2 of \cite{chikuse2012statistics}, there exists a random orthogonal matrix $\bfO$, which is drawn uniformly random from the orthogonal group $\mathcal{O}(m)$ in the Haar measure, s.t., $\Span(\X)$ shares the same distribution with $\bfO\Xi$, where $\Xi=\Span\{\e_1,\e_2,\ldots,\e_r\}$.

    Thus, $\hat{\phi}_{i,i}=\|\mathsf{Proj}_{\Span(\X)}(\e_i)\|^2\stackrel{d}{=}\|\mathsf{Proj}_{\bfO\Xi}(\e_i)\|^2=\|\mathsf{Proj}_{\Xi}(\bfO^\top\e_i)\|^2$. By the invariance of Haar measure, we have that $\vv:=\bfO^\top\e_i$ is uniformly random on the unit sphere $\mathcal{S}^{m-1}:=\{\x\in\mathbb{R}^m|\|\x\|=1\}$.

    From Lemma \ref{apdx.lemma.muirhead}, it follows that $\vv\stackrel{d}{=}\frac{\p}{\|\p\|}$, where $\p\sim\mathcal{N}(0,1)^m$. Combining this and the distribution of $\hat{\phi}_{i,i}$ obtained above, we arrive at $\hat{\phi}_{i,i}\stackrel{d}{=}\sum_{j=1}^{r}\vv_j^2\stackrel{d}{=}\sum_{j=1}^{r}\frac{\p_j^2}{\|\p\|^2}$, where $\vv_j,\p_j$ are the $j$-th entries of $\vv$ and $\p$, respectively.

    Since $\p_1,\p_2,\ldots,\p_m$ are i.i.d. standard Gaussian, we have that $$\sum_{j=1}^r\p_j^2\sim\chi_r^2,\quad\sum_{j=r+1}^m\p_j^2\sim\chi_{m-r}^2.$$ Therefore, $\phi_{i,i}\stackrel{d}{=}\hat{\phi}_{i,i}\stackrel{d}{=}\frac{\chi_r^2}{\chi_r^2+\chi_{m-r}^2}$, i.e., $\phi_{i,i}\sim\mathsf{Beta}(\frac{r}{2},\frac{m-r}{2})$.
\end{proof}

\begin{lemma}\label{apdx.lemma.exp_markov}
    Let $X,Y$ be two random variables and $X\perp Y$. Then, $P(X-Y\geq0)\leq E(e^{\lambda X})E(e^{-\lambda Y})$ holds for all $\lambda>0$.
\end{lemma}
\begin{proof}
    Since $\lambda>0$, we have that $X-Y\geq0$ if and only if $e^{\lambda(X-Y)}\geq1$. Then, $P(X-Y\geq0)=P(e^{\lambda(X-Y)}\geq1)\stackrel{(a)}{\leq} E(e^{\lambda(X-Y)})$, where $(a)$ is by Markov's inequality.

    From $X\perp Y$, we can write $E(e^{\lambda(X-Y)})$ as $E(e^{\lambda X})E(e^{-\lambda Y})$.

    Therefore, $P(X-Y\geq0)\leq E(e^{\lambda X})E(e^{-\lambda Y})$.
\end{proof}

\begin{lemma}\label{apdx.lemma.log_approx}
    Assuming $-\frac{1}{4}\leq x\leq\frac{1}{4}$, then $\log(1+x)\geq x-x^2$.
\end{lemma}
\begin{proof}
    Define $f(x):=\log(1+x)-x+x^2$, $x\in[-\frac{1}{4},\frac{1}{4}]$.

    The derivative of $f(x)$ is $f'(x)=\frac{1}{1+x}-1+2x=\frac{x+2x^2}{1+x}$. Since $x+2x^2\begin{cases}
        <0,\,x\in[-\frac{1}{4},0)\\
        \geq0,\, x\in[0,\frac{1}{4}]
    \end{cases}$, we have that $f(x)\geq f(0)=0$.
\end{proof}

\begin{lemma}\label{apdx.lemma.init_r_m}
    Let $\Z\sim\mathcal{N}(0,1)^{m\times r}$, $\X=\Z(\Z^\top\Z)^{-1/2}$, $\phi_{i,i}=\uu_i^\top\X\X^\top\uu_i$, $1\leq i\leq r_A$, $m\geq 12r$, $\varepsilon=\frac{1}{2}$.

    Then, $P(\max\limits_{1\leq i\leq r_A}|\phi_{i,i}-\frac{r}{m}|\geq\varepsilon\cdot\frac{r}{m})\leq2r_A\exp(-c_3\varepsilon^2r)$ for some universal constant $c_3>0$.
\end{lemma}
\begin{proof}
    From Lemma \ref{apdx.lemma.init_r_m}, we have that $\phi_{i,i}\stackrel{d}{=}\frac{\chi_r^2}{\chi_r^2+\chi_{m-r}^2}$.
    
    Let $X\sim\chi_r^2,Y\sim\chi_{m-r}^2,X\perp Y$. Then, $\phi_{i,i}=\frac{X}{X+Y}$. We use Chernoff bounds to prove this.

    For the up tail, $\phi_{i,i}\geq(1+\varepsilon)\frac{r}{m}$ if and only if $\frac{X}{X+Y}\geq(1+\varepsilon)\frac{r}{m}$, i.e., $(1-\alpha)X-\alpha Y\geq0$, where $\alpha:=(1+\varepsilon)\frac{r}{m}\leq\frac{1}{2}$. This is equivalent to $X-\frac{\alpha}{1-\alpha}Y\geq0$.

    For any $\lambda\in(0,\frac{1}{2})$, $\lambda\cdot\frac{\alpha}{1-\alpha}<\frac{1}{2}$, we can obtain that 
    \begin{align*}
        P(X-\frac{\alpha}{1-\alpha}Y\geq0)&=P(e^{\lambda(X-\frac{\alpha}{1-\alpha}Y)}\geq1)\\
        &\stackrel{(a)}{\leq}E(e^{\lambda X})E(e^{-\lambda\cdot\frac{\alpha}{1-\alpha}Y})\\
        &=(1-2\lambda)^{-\frac{r}{2}}(1+2\lambda\cdot\frac{\alpha}{1-\alpha})^{-\frac{m-r}{2}},
    \end{align*}
    where $(a)$ is from Lemma \ref{apdx.lemma.exp_markov}.

    Let $g(\lambda):=(1-2\lambda)^{-\frac{r}{2}}(1+2\lambda\cdot\frac{\alpha}{1-\alpha})^{-\frac{m-r}{2}}$. A direct calculation of the derivative shows that the stationary point of $\log(g(\lambda))$ is $\lambda^*=\frac{\frac{\alpha}{1-\alpha}(m-r)-r}{2[r+\frac{\alpha}{1-\alpha}(m-r)]}\in(0,\frac{1}{2})$. 
    
    Substituting $\varepsilon=\frac{1}{2}$ in, it follows that $\frac{\alpha}{1-\alpha}=\frac{3r}{2m-3r}$, which leads to $\lambda^*=\frac{m}{10m-12r}\in[\frac{1}{10},\frac{1}{9}]$.

    Expanding the expression of $\log(g(\lambda))$, we have that
    \begin{align*}
        \log(g(\lambda^*))&=-\frac{r}{2}\log(1-2\lambda^*)-\frac{m-r}{2}\log(1+\lambda^*\cdot\frac{3r}{2m-3r})\\
        &\stackrel{(b)}{\leq}-\frac{r}{2}[-2\lambda^*-4(\lambda^*)^2]-\frac{m-r}{2}[2\lambda^*\cdot\frac{3r}{2m-3r}-4(\lambda^*\cdot\frac{3r}{2m-3r})^2]\\
        &=-\frac{r}{2}[(-2+\frac{6(m-r)}{2m-3r})\lambda^*+(-4-\frac{36r(m-r)}{(2m-3r)^2})(\lambda^*)^2]\\
        &\stackrel{(c)}{\leq}-\frac{r}{2}[\lambda^*-5(\lambda^*)^2]\\
        &\stackrel{(d)}{\leq}-\frac{r}{50},
    \end{align*}
    where $(b)$ is from Lemma \ref{apdx.lemma.log_approx}; $(c)$ is by $m\geq12r$; and $(d)$ follows from $\lambda^*\in[\frac{1}{10},\frac{1}{9}]$.

    Thus, the up tail bound is $P(\phi_{i,i}\geq\frac{3r}{2m})\leq e^{-\frac{r}{50}}$.

    For the down tail, $\phi_{i,i}\leq(1-\varepsilon)\frac{r}{m}$ if and only if $\frac{X}{X+Y}\leq(1-\varepsilon)\frac{r}{m}$, i.e., $(1-\beta)X-\beta Y\leq0$, where $\beta:=(1-\varepsilon)\frac{r}{m}\leq\frac{1}{2}$. This is equivalent to $X-\frac{\beta}{1-\beta}Y\leq0$.

    For any $\lambda\in(0,\frac{1}{2})$, $\lambda\cdot\frac{\beta}{1-\beta}<\frac{1}{2}$, we can obtain that 
    \begin{align*}
        P(X-\frac{\beta}{1-\beta}Y\leq0)&=P(e^{\lambda(-X+\frac{\beta}{1-\beta}Y)}\geq1)\\
        &\stackrel{(e)}{\leq}E(e^{-\lambda X})E(e^{\lambda\cdot\frac{\beta}{1-\beta}Y})\\
        &=(1+2\lambda)^{-\frac{r}{2}}(1-2\lambda\cdot\frac{\beta}{1-\beta})^{-\frac{m-r}{2}},
    \end{align*}
    where $(e)$ is from Lemma \ref{apdx.lemma.exp_markov}.

    Let $h(\lambda):=(1+2\lambda)^{-\frac{r}{2}}(1-2\lambda\cdot\frac{\beta}{1-\beta})^{-\frac{m-r}{2}}$.

    A direct calculation of the derivative shows that the stationary point of $\log(h(\lambda))$ is $\hat{\lambda}=\frac{r-\frac{\beta}{1-\beta}(m-r)}{\frac{2\beta}{1-\beta}m}>0$.

    Substituting $\varepsilon=\frac{1}{2}$ in, it follows that $\frac{\beta}{1-\beta}=\frac{r}{2m-r}$, which leads to $\hat{\lambda}=\frac{1}{2}$.

    Taking the limitation as $\lambda\to\frac{1}{2}^{-}$, we have that 
    \begin{align*}
        \lim\limits_{\lambda\to\frac{1}{2}^-}h(\lambda)&=-\frac{\log(2)}{2}r-\frac{m-r}{2}\log(1-\frac{r}{2m-r})\\
        &\stackrel{(f)}{\leq}-\frac{\log(2)}{2}r-\frac{m-r}{2}(-\frac{r}{2m-r}-\frac{r^2}{(2m-r)^2})\\
        &=-\frac{r}{2}[\log(2)-\frac{m-r}{2m-r}-\frac{r(m-r)}{(2m-r)^2}]\\
        &\stackrel{(g)}{\leq}-\frac{r}{2}[\log(2)-\frac{1}{2}-\frac{1}{20}]\\
        &\leq-\frac{r}{20},
    \end{align*}
    where $(g)$ is from Lemma \ref{apdx.lemma.log_approx}; and $(f)$ is by $m\geq12r$.

    Thus, the down tail bound is $P(\phi_{i,i}\leq\frac{r}{2m})\leq e^{-\frac{r}{20}}$.

    Then, $P(|\phi_{i,i}-\frac{r}{m}|\geq\frac{r}{2m})=P(\phi_{i,i}\geq\frac{3r}{2m})+P(\phi_{i,i}\leq\frac{r}{2m})\leq e^{-\frac{r}{50}}+e^{-\frac{r}{20}}\leq2e^{-\frac{r}{50}}$.

    Therefore, $P(\max\limits_{1\leq i\leq r_A}|\phi_{i,i}-\frac{r}{m}|\geq\frac{r}{2m})\leq\sum_{i=1}^{r_A} P(|\phi_{i,i}-\frac{r}{m}|\geq\frac{r}{2m})\leq2r_A\exp(-\frac{r}{50})$.
\end{proof}

\begin{lemma}\label{apdx.lemma.init_vec}
    Let $\uu\in\mathbb{R}^m,m\geq3$, $\uu^\top\uu=1$. Let $\z\sim\mathcal{N}(0,1)^m$, and define $\x=\z(\z^\top\z)^{-1/2}$. It is guaranteed that $P(\uu^\top\x\x^\top\uu\geq\frac{1}{cm})\geq1-\frac{2}{\sqrt{cm}\cdot\mathsf{Beta}(\frac{1}{2},\frac{m-1}{2})}$ for any constant $c>1$.
\end{lemma}
\begin{proof}
    From Lemma \ref{apdx.lemma.beta_dist}, we have that $y:=\uu^\top\x\x^\top\uu\sim\mathsf{Beta}(\frac{1}{2},\frac{m-1}{2})$.

    For any $\varepsilon\in(0,1)$, $P(y\leq\varepsilon)=\frac{1}{\mathsf{Beta}(\frac{1}{2},\frac{m-1}{2})}\mathlarger{\int}_{0}^{\varepsilon}y^{-\frac{1}{2}}(1-y)^{\frac{m-3}{2}}\,\text{d}y$.

    Since $m\geq3$ and $y\leq\varepsilon$, it follows that
    \begin{align*}
        P(y\leq\varepsilon)\leq\frac{1}{\mathsf{Beta}(\frac{1}{2},\frac{m-1}{2})}\mathlarger{\int}_{0}^{\varepsilon}y^{-\frac{1}{2}}\,\text{d}y=\frac{2\varepsilon^{\frac{1}{2}}}{\mathsf{Beta}(\frac{1}{2},\frac{m-1}{2})}.
    \end{align*}

    Substituting $\varepsilon=\frac{1}{cm}$ in, we arrive at $P(y\leq\frac{1}{cm})\leq\frac{2}{\sqrt{cm}\cdot\mathsf{Beta}(\frac{1}{2},\frac{m-1}{2})}$.

    Note that $\mathsf{Beta}(\frac{1}{2},\frac{m-1}{2})=\sqrt{\pi}\frac{\Gamma(\frac{m-1}{2})}{\Gamma(\frac{m}{2})}$, and by Gautschi's inequality, we have that
    \begin{align*}
        \sqrt{\frac{m}{2}-1}\leq\frac{\Gamma(\frac{m}{2})}{\Gamma(\frac{m-1}{2})}\leq\sqrt{\frac{m}{2}}.
    \end{align*}
    Then, $\frac{1}{\mathsf{Beta}(\frac{1}{2},\frac{m-1}{2})}\leq\frac{\sqrt{\frac{m}{2}}}{\sqrt{\pi}}$, which leads to $P(y\geq\frac{1}{cm})\geq1-\sqrt{\frac{2}{\pi c}}$.
\end{proof}

    \begin{lemma}\label{apdx.lemma.sin_x_y}
        Assume that $\x,\y\in\mathbb{R}^m$ are two unit vectors, i.e., $\x^\top\x=\y^\top\y=1$. Let $\X=\x\x^\top,\Y=\y\y^\top\in\mathbb{S}^m$. Then, $\|\X-\Y\|=\sqrt{1-|\x^\top\y|^2}$.
    \end{lemma}
    \begin{proof}
        We first consider the case that $|\x^\top\y|=1$.

        Since $\x$ and $\y$ are unit vectors, it follows that $\x=\y$ or $\x=-\y$, which guarantees that $\|\X-\Y\|=0$. Thus, we just need to consider cases that $|\x^\top\y|<1$.
        
        Let $c:=\x^\top\y\in(-1,1)$. Since $\x$ and $\y$ are unit vectors, we have that 
        \begin{align*}
            \X^2=\X,\quad\Y^2=\Y.
        \end{align*}
        Let $\A:=\X-\Y\in\mathbb{S}^m$. Then, $\|\A\|=\sqrt{\lambda_{max}(\A^\top\A)}=\sqrt{\lambda_{max}(\A^2)}$.

        Directly expanding the expression of $\A$, we obtain
        \begin{align*}
            \A^2&=(\X-\Y)^2\\
            &=\X^2+\Y^2-\X\Y-\Y\X\\
            &=\X+\Y-\X\Y-\Y\X.
        \end{align*}
        Together with $\X\Y=\x\x^\top\y\y^\top=c\x\y^\top$ and $\Y\X=\y\y^\top\x\x^\top=c\y\x^\top$, it follows that
        \begin{align*}
            \A^2=\X+\Y-c(\x\y^\top+\y\x^\top).
        \end{align*}
        For any $\uu\in\mathbb{R}^m$, $\uu\perp\x,\y$, we have that
        $\X\uu=\x\x^\top\uu=\bm{0}$ and $\Y\uu=\y\y^\top\uu=\bm{0}$. Then, $\A^2\uu=\bm{0}$, which means that $\A^2$ can only have non-zero eigenvalues on the space $\Span\{\x,\y\}$.

        Let $\z=\frac{\y-c\x}{\sqrt{1-c^2}}\in\mathbb{R}^m$. We can easily verify that $\z\perp\x$ and $\z\in\Span\{\x,\y\}$. Thus, $\Span\{\x,\y\}=\Span\{\x,\z\}$.

        Through simple calculation, we obtain
        \begin{align*}
            \A^2\x=(1-c^2)\x,\quad\A^2\uu=(1-c^2)\uu.
        \end{align*}
        Thus, the non-zero eigenvalue of $\A^2$ is $1-c^2$, which means that
        \begin{align*}
            \|\X-\Y\|=\|\A\|=\sqrt{\lambda_{max}(\A^2)}=\sqrt{1-c^2}=\sqrt{1-|\x^\top\y|^2}.
        \end{align*}
        This completes the proof.
    \end{proof}
    
\end{document}